\documentclass[oneside]{amsart}

\usepackage[letterpaper,body={14.0cm,22.0cm}, mag=1000]{geometry}
\usepackage{amssymb}
\usepackage{amsthm}
\usepackage{amscd}
\usepackage{enumitem}
\usepackage{float}
\usepackage{placeins}
\usepackage{caption}
\numberwithin{equation}{section}
\theoremstyle{plain}

\newtheorem{thm}{Theorem}[section]
 
 \newtheorem{lemma}[thm]{Lemma}
\newtheorem{prop}[thm]{Proposition}

\theoremstyle{definition}
\newtheorem{remark}[equation]{Remark}

\newcommand{\dlabel}[1]{\ifmmode \text{\ttfamily \upshape [#1] } \else
{\ttfamily \upshape [#1] }\fi \label{#1}}

\newcommand{\Ho}{\operatorname{H} }
\newcommand{\M}{\operatorname{M} }

\newcommand{\Z}{\operatorname{Z} }

\newcommand{\im}{\operatorname{Im} }

\newcommand{\gen}[1]{\left < #1 \right >}

\newcommand{\res}{\operatorname{res} }

\newcommand{\Hom}{\operatorname{Hom} }

\newcommand{\Ker}{\operatorname{Ker} }

\begin{document}

\setlength{\baselineskip}{15pt}

\title{THE SCHUR MULTIPLIERS OF  $p$-GROUPS OF ORDER $p^5$}

\author{Sumana Hatui}
\address{School of Mathematics, Harish-Chandra Research Institute, Chhatnag Road, Jhunsi, Allahabad 211019, INDIA}
\address{\& Homi Bhabha National Institute, Training School Complex, Anushakti Nagar, Mumbai 400085, India}
 \email{sumanahatui@hri.res.in, sumana.iitg@gmail.com}

\author{Vipul Kakkar}
\address{Department of Mathematics, Central University of Rajasthan,  NH-8, Bandar Sindri, Ajmer-305817, Rajasthan, INDIA}
 \email{vplkakkar@gmail.com}
 
 \author{Manoj K. Yadav}
\address{Harish-Chandra Research Institute, HBNI, Chhatnag Road, Jhunsi, Allahabad 211019, INDIA}
 \email{myadav@hri.res.in}
 
\subjclass[2010]{20D15, 20J06}
\keywords{finite $p$-group, Schur multiplier, non-abelian tensor square, exterior square}

\begin{abstract}
In this article, we compute the Schur multiplier, non-abelian tensor square and exterior square of non-abelian $p$-groups of order $p^5$. As an application we determine the capability of groups of order $p^5$.
\end{abstract}

\maketitle 

\section{Introduction}
The Schur multiplier $\M(G)$ of a group $G$ was invented by I. Schur in \cite{SJ1} and \cite{SJ}  as an obstruction for  a projective representation to become linear representation, and is  defined as second integral homology group $\Ho_2(G,\mathbb{Z})$, where $\mathbb{Z}$ is a trivial $G$-module. It is well known that for a finite group $G$, the group $\M(G)$ is isomorphic to the second cohomology group $\Ho^2(G,\mathbb{C}^{\star})$, where $\mathbb{C}^{\star}$ is a trivial $G$-module. 

The  \emph{non-abelian tensor product} $G \otimes H$ of two groups $G$ and $H$,  acting on each other and satisfying certain compatibility conditions,  was defined by Brown and Loday \cite{BL} as a generalization of abelian tensor product. In particular, when a group $G$ acts on itself by conjugation then $G \otimes G$ is called non-abelian tensor square which is defined as follows. Let $G$ acts on itself by conjugation, i.e., $h^g=g^{-1}hg$ for all $h, g \in G$.
Then the non-abelian tensor square $G \otimes G$ of  $G$ is the group generated by the symbols $g \otimes h$ for all $g,h \in G$, subject to the relations
$$gg' \otimes h = (g^{g'} \otimes h^{g'})(g' \otimes h)$$ and $$g \otimes hh' = (g \otimes h')({g}^{h'} \otimes h^{h'})$$
for all $g,g',h,h' \in G$. The
 \emph{non-abelian exterior square} of $G$, denoted by $G \wedge G$, is the quotient group  of $G\otimes G$ by $\nabla(G)$, where $\nabla(G)$ is the normal subgroup generated by the elements $g \otimes g$ for all $g \in G$. It follows from the definition that the map  $f: G\otimes G \rightarrow G'$, defined on the generators by $f(g \otimes h)=[g,h]$, is an epimorphism,  where $[g,h]=g^{-1}h^{-1}gh$. The epimorphism $f$ then induce an epimorphism $f' : G \wedge G \rightarrow G'$. The kernel of $f'$ is isomorphic to the Schur multiplier $\M(G)$ of $G$ \cite{BL}. 
 
We now present a different description of $G \otimes G$, introduced in \cite{RO}, which, sometimes, comes more handy for evaluating tensor square of a group $G$.
By $G^{\phi}$ we denote  the isomorphic image of a group $G$ via the isomorphism $\phi$. Consider the group
$$\nu(G):=\gen{G,G^{\phi} \mid \Re, \Re ^{\phi}, [g_1,g_2^{\phi}]^g=[g_1^g, (g_2^g)^{\phi}]=[g_1,g_2^\phi]^{g^\phi} \;\text{for all}\;  g,g_1,g_2 \in G}$$
where $\Re, {\Re}^\phi$ are the defining relations of $G$ and $G^{\phi}$ respectively.   Recall that the commutator subgroup of $G$ and $G^\phi$   in $\nu(G)$ is defined as $[G, G^\phi] = \gen{[g, h^\phi] \mid g, h \in G}$. Then the map $\Phi: G \otimes G \rightarrow [G,G^\phi]$, defined by 
$$\Phi(g  \otimes h)=[g,h^\phi],\;\; g,h \in G$$ 
is an isomorphism \cite{RO}. Let $\pi : [G,G^\phi] \to [G, G^\phi]/\gen{[g,g^\phi] \mid g \in G}$ be the natural projection. Then we have an isomorphism between $G \wedge G$ and  $[G, G^\phi]/\gen{[g,g^\phi] \mid g \in G}$. 

This article is devoted to computing the  Schur multiplier, non-abelian tensor square and non-abelian exterior square of non-abelian $p$-groups of order $p^5$.  We sometimes compute the Schur multiplier of a group and use this to compute the non-abelian tensor square, and sometimes the other way round. The technique of the proof highly depends on the given group, and therefore the article uses a blend of almost all known techniques. As an application we categorise capable and non-capable groups of order $p^5$. A group $G$ is said to be  \emph{capable}  if there exists a group $H$ such that $G \cong H/\Z(H)$, where $\Z(H)$ denotes the center of $H$.

There are $10$ isoclinism classes $\Phi_i, 1 \leq i \leq 10$, of groups of order $p^5$ as per the classification by R. James \cite{RJ}. The groups in $\Phi_1$ are abelian groups.
We recall some notations from \cite{RJ}. For an element $\alpha_{i+1}$ of a finite $p$-group $G$, by
$\alpha_{i+1}^{\left(p\right)}$, we mean  $\alpha_{i+1}^p \alpha_{i+2}^{p \choose 2} \cdots \alpha_{i+k}^{p \choose k} \cdots \alpha_{i+p}$, where $\alpha_{i+2},...,\alpha_{i+p}$ are suitably defined elements of $G$. Observe that, for the groups under consideration in this article, $\alpha_i^{(p)}=\alpha_i^p$ for $p >3$.
By $\nu$ we denote   the smallest positive integer which is a non-quadratic residue (mod $p$), and by $\zeta$ we  denote the smallest positive integer which is a primitive root (mod $p$). Relations of the form $[\alpha, \beta]= 1$ for generators $\alpha$ and  $\beta$ are omitted in the presentations of the groups. 

The layout of the article is as follows. In Section 2, we present preliminary results including the Schur multiplier and tensor square of groups of order $p^3$ and $p^4$, $p \ge 5$. In the following sections the Schur multiplier, exterior square and tensor square of groups $G$ of order $p^5$, $p \ge 5$ are computed. Section 3 deals with special $p$-groups of order $p^5$. In Section 4, we consider groups of maximal class. In Section 5, all remaining groups of order $p^5$ are dealt with, and  for $p \ge 5$ our main results are presented in Table \ref{Table1}. The results on the  capability of these groups are presented in Table \ref{Table2}.
In the last section we consider the groups of order $2^5$ and $3^5$, where the results  are given in  Table \ref{Table3}, Table \ref{Table4}, using HAP \cite{HAP} of GAP \cite{GAP}.
 Since our proofs heavily depend on the presentations of the groups under consideration, we include the presentations of groups of order $p^n$, $3 \le n \le 5$, from \cite{RJ} in Annexure A, after the bibliography. The set of  identities obtained by computation is usually sequential, i.e., the first might have been used for obtaining the second, and the first and second for the third, and so on. 

The presentation of the article is a bit unusual. We only prove lemmas, in each of which we compute required information for specific classes of groups sharing the technique of proof. Our main result is a combination of all the lemmas.

We conclude this section with setting some notations. The commutator and Frattini subroup of a finite group $G$ are denoted by  $G'$ and $\Phi(G)$ respectively. By $d(G)$ we denote the cardinality of a minimal generating set of a finitely generated group $G$. 
By $\mathbb{Z}_p^{(k)}$ we denote $\mathbb{Z}_p \times \mathbb{Z}_p \times \cdots \times \mathbb{Z}_p$($k$ times). For a group $G$, $\gamma_i(G)$ denotes the $i$-th term of the lower central series of  $G$ and $G^{ab}$ denotes the quotient group $G/\gamma_2(G)$. Notice that  $\gamma_2(G) = G'$.  The \emph{epicenter}, denoted by  $Z^*(G)$, of a group $G$ is defined to be the smallest central subgroup $K$ of $G$ such that $G/K$ is capable.


\section{Preliminaries}
In this section we list some  known results which we'll  use in the upcoming sections.
\begin{thm}[\cite{SJ}]\label{D}
For two groups $H$ and $K$,
$$\M(H \times K) \cong \M(H) \times \M(K) \times (H/H' \otimes K/K').$$
\end{thm}

\begin{thm}[\cite{KT}]\label{1}
If a group $G$ is a semidirect product of a normal subgroup $N$ and a subgroup $T$, and $M$ is a $G$-module with trivial $G$-action. Then the following sequence is exact
$$1 \rightarrow \Ho^1\big(T,\Hom(N,M)\big) \rightarrow \Ho^2(G,M)_2 \rightarrow  \Ho^2(N,M)^T \rightarrow \Ho^2\big(T,\Hom(N,M)\big),$$
where 
$$\Ho^2(G,M)_2=\Ker\big(\res^G_T:\Ho^2(G,M) \rightarrow \Ho^2(T,M)\big)$$
 and $\Ho^2(N,M)^T$ is $T$- stable subgroup of $\Ho^2(N,M)$.
\end{thm}

\begin{thm}[See Theorem 3.1 of \cite{MRRR}] \label{2}
Let $G$ be a finite group and $K$ any normal subgroup of $G$. Set $H=G/K$. Then 
\begin{enumerate}[label=(\roman*)]
\item  $|\M(H)|$ divides $|\M(G)||G' \cap K|$,
\item  $d\big(\M(H)\big) \leq d\big(\M(G)\big)+d(G' \cap K)$.
\end{enumerate}
\end{thm}

\begin{thm}[See Theorem 4.1 of \cite{MRRR}]\label{J}
Let $G$ be a finite group and $K$ a central subgroup of $G$. Set $A = G/K$. Then
$|\M(G)||G'\cap K|$ divides $|\M(A)| |\M(K)| |A^{ab} \otimes K|$.
\end{thm}

\begin{thm}[See Theorem 3.1 of \cite{MRR}]\label{J'}
Let  $G$  be  a  finite group and $N$ any normal subgroup such that $G/N$ is  cyclic. Then
\begin{enumerate}[label=(\roman*)]
\item $|\M(G)|$ divides $|\M(N)||N/N'|$,
\item  $d\big(\M(G)\big) \leq d\big(\M(N)\big) + d(N/N')$.
\end{enumerate}
\end{thm}

\begin{thm}[\cite{GT}]\label{G}
Let $Z$ be a central subgroup of a finite group $G$. Then the following sequence is exact 
$$ G/G' \otimes Z \xrightarrow{\lambda} \M(G) \xrightarrow{\mu} \M(G/Z) \rightarrow G'\cap Z \rightarrow 1,$$
where $\lambda$ is the Ganea map.
\end{thm}

\begin{thm}[\cite{F}]\label{G1}
Let $Z$ be a central subgroup of a finite group $G$. Consider the Ganea map $\lambda: G/G' \otimes Z \rightarrow \M(G)$. Then $Z \subseteq Z^*(G)$ if and only if $G/G' \otimes Z$ is contained in the kernal of $\lambda$. 
\end{thm}

\begin{lemma}[See Lemma 9 of \cite{PM}]\label{O}
If a group $G$ has nilpotency class $\leq 5$, then\\
\centerline{$[x^n,y]=[x,y]^n[x,y,x]^{n \choose 2}[x,y,x,x]^{n \choose 3}[x,y,x,x,x]^{n \choose 4}[x,y,x,[x,y]]^{\sigma(n)}$}
for $x,y \in G$ and any positive integer $n$, where $\sigma(n)=n(n-1)(2n-1)/6$.
\end{lemma}

The following results follow from \cite{BM} and  \cite{RO}.

\begin{lemma}\label{R}
For a group $G$, the following properties hold in $\nu (G)$.
\begin{enumerate}[label=(\roman*)]
\item If $G$ is nilpotent of class $c$, then $\nu(G)$ is nilpotent of class at most $c + 1$.
\item If $G$ is a $p$-group, then $\nu(G)$ is a $p$-group.
\item $[g_1^\phi,g_2,g_3]=[g_1,g_2^\phi,g_3]=[g_1,g_2,g_3^\phi]=[g_1^\phi,g_2^\phi,
g_3]=[g_1^\phi,g_2,g_3^\phi]=[g_1,g_2^\phi,g_3^\phi]$ for all $g_1,g_2,g_3 \in G$.
\item If either $g \in G'$ or $h \in H'$, then $[g,h^\phi]=[h,g^\phi]^{-1}$.
\item $[g,g^\phi]=1$, for all $g \in G'$.
\item $[[g_1,g_2^\phi],[h_1,h_2^\phi]]=[[g_1,g_2],[h_1,h_2]^\phi]$ for all $g_1,g_2,h_1,h_2 \in G$.
\item $[[g_1,g_2^{\phi}],[g_2,g_1^{\phi}]]=1$ for all $g_1,g_2\in G$.
\item If $g,g_1,g_2 \in G$ such that $[g,g_1]=1=[g,g_2]$, then $[g_1,g_2,g^{\phi}]=1$.
\item $[g,g^\phi]$ is central in $\nu(G)$ for all $g \in G$.
\end{enumerate}
\end{lemma}

The following result follows from the proof of \cite[Lemma 2.1(iv)]{RO}:

\begin{lemma}
For all $g_1, g_2 \in G$, $[g_1,g_2^{\phi}]=[g_2,g_1^{\phi}]^{-1}$ in $G \wedge G$, i.e., modulo $\nabla(G)$.
\end{lemma}

\begin{prop}[See Proposition 20 of  \cite{BM}]\label{B}
Let G be a polycyclic group with a polycyclic generating sequence $g_1,\ldots , g_k$.
Then $G \wedge G$ is generated by $\{[g_i,g_j^{\phi}], i > j\}$.
\end{prop}

The generating sets for the groups $G$, given in \cite{RJ}, form  polycyclic generating sequences. Then Proposition \ref{B} provides a  generating set for  $G \wedge G$. This information will be used several times throughout the article without any further reference.

By items $(vi)$ and $(viii)$ of Lemma \ref{R}, we get
\begin{lemma}\label{S}
If $G$ is of nilpotency class $2$, then $G \otimes G$ is abelian.
\end{lemma}

Let $\Gamma$ denote  the Whitehead's quadratic functor (see \cite{W}).
For finite abelian groups $G$ and $H$, we get
$$(i)\;\; \Gamma(G \times H)\cong \Gamma(G) \times \Gamma(H) \times (G \otimes H),$$
\begin{align*}
(ii)\;\;\;\;\;\;\;\Gamma(\mathbb{Z}_n)=\left\{\begin{array}{lll}
               \mathbb{Z}_n & \textnormal{ n odd} \\
               \mathbb{Z}_{2n} & \textnormal{ n even}\;. \;\;\;\;\;\;\;\;\;\;
            \end{array}\right.
\end{align*}

\begin{thm}[See Theorem 1.3 of \cite{BF}]\label{E}
Let $G^{ab}$ be finitely generated, and  have no element of order $2$. Then $G \otimes G \cong \Gamma (G^{ab}) \times (G \wedge G)$. In particular, if $G$ is a finite $p$-group, $p$ odd, then 
$G \otimes G \cong \Gamma (G^{ab}) \times (G \wedge G)$.
\end{thm}

\begin{prop}[See Proposition 11 of \cite{RB}]\label{D1}
For the groups $G$ and $H$, we have 
$$(G \times H) \otimes (G \times H)=(G \otimes G) \times (G \otimes H) \times (H \otimes G) \times (H \otimes H).$$
\end{prop}

\begin{thm} \label{SHHH}
We have the following table for non-abelian groups of order $p^3$ and $p^4$, $p \ge 5$.
 \begin{table}[H]
\centering
 \begin{tabular}{|c c c c c c c|} 
 \hline
 $G$ & $G^{ab}$ & $\Gamma (G^{ab})$ & $\M(G)$ & $G \wedge G$ &  $G \otimes G$ & Generators of $G \wedge G$ \\ [.5ex] 
 \hline
$\Phi_2(21)$ & $\mathbb{Z}_p^{(2)}$ &  $\mathbb{Z}_p^{(3)}$ &  $\{1\}$ & $\mathbb{Z}_p$ & $\mathbb{Z}_p^{(4)}$ &  $[\alpha_1,\alpha^\phi]$ \\

$\Phi_2(111)$ & $\mathbb{Z}_p^{(2)}$ &  $\mathbb{Z}_p^{(3)}$ &  $\mathbb{Z}_p^{(2)}$ & $\mathbb{Z}_p^{(3)}$ & $\mathbb{Z}_p^{(6)}$ & $[\alpha_1,\alpha^\phi], [\alpha_2,\alpha^\phi], [\alpha_2,\alpha_1^\phi] $ \\
 
$\Phi_2(211)a$ & $\mathbb{Z}_p^{(3)}$ &  $\mathbb{Z}_p^{(6)}$ &  $\mathbb{Z}_p^{(2)}$ & $\mathbb{Z}_p^{(3)}$ & $\mathbb{Z}_p^{(9)}$ & $[\alpha_1,\alpha^\phi], [\alpha_3,\alpha^\phi], [\alpha_3,\alpha_1^\phi] $ \\

$\Phi_2(1^4)$ & $\mathbb{Z}_p^{(3)}$ &  $\mathbb{Z}_p^{(6)}$ &  $\mathbb{Z}_p^{(4)}$ & $\mathbb{Z}_p^{(5)}$ & $\mathbb{Z}_p^{(11)}$ & $[\alpha_1,\alpha^\phi], [\alpha_2,\alpha^\phi], [\alpha_2,\alpha_1^\phi], $  \\

& & & & &  & $[\alpha_3,\alpha^\phi], [\alpha_3,\alpha_1^\phi] $ \\

$\Phi_2(31)$ & $\mathbb{Z}_{p^2} \times \mathbb{Z}_p$ &  $\mathbb{Z}_{p^2} \times \mathbb{Z}_p^{(2)}$ &  $\{1\}$ & $\mathbb{Z}_p$ & $\mathbb{Z}_{p^2} \times \mathbb{Z}_p^{(3)}$ & $[\alpha_1,\alpha^\phi]$\\

$\Phi_2(22)$ & $\mathbb{Z}_{p^2} \times \mathbb{Z}_p$ &  $\mathbb{Z}_{p^2} \times \mathbb{Z}_p^{(2)}$ &  $\mathbb{Z}_p$ & $\mathbb{Z}_{p^2}$ & $\mathbb{Z}_{p^2}^{(2)} \times \mathbb{Z}_p^{(2)}$ & $[\alpha_1,\alpha^\phi]$ \\

$\Phi_2(211)b$ & $\mathbb{Z}_p^{(3)}$ &  $\mathbb{Z}_p^{(6)}$ &  $\mathbb{Z}_p^{(2)}$ & $\mathbb{Z}_p^{(3)}$ & $\mathbb{Z}_p^{(9)}$ & $[\alpha_1,\alpha^\phi], [\gamma,\alpha^\phi], [\gamma,\alpha_1^\phi] $  \\

$\Phi_2(211)c$ & $\mathbb{Z}_{p^2} \times \mathbb{Z}_p$ &  $\mathbb{Z}_{p^2} \times \mathbb{Z}_p^{(2)}$ &  $\mathbb{Z}_p^{(2)}$ & $\mathbb{Z}_p^{(3)}$ & $\mathbb{Z}_{p^2} \times \mathbb{Z}_p^{(5)}$ & $[\alpha_1,\alpha^\phi], [\alpha_2,\alpha^\phi], [\alpha_2,\alpha_1^\phi] $\\

$\Phi_3(211)a$ & $\mathbb{Z}_p^{(2)}$ &  $\mathbb{Z}_p^{(3)}$ &  $\mathbb{Z}_p$ & $\mathbb{Z}_p^{(3)}$ & $\mathbb{Z}_p^{(6)}$ & $[\alpha_1,\alpha^\phi], [\alpha_2,\alpha^\phi], [\alpha_2,\alpha_1^\phi] $ \\

$\Phi_3(211)b_r$ & $\mathbb{Z}_p^{(2)}$ &  $\mathbb{Z}_p^{(3)}$ &  $\mathbb{Z}_p$ & $\mathbb{Z}_p^{(3)}$ & $\mathbb{Z}_p^{(6)}$ & $[\alpha_1,\alpha^\phi], [\alpha_2,\alpha^\phi], [\alpha_2,\alpha_1^\phi]$  \\

$\Phi_3(1^4)$ & $\mathbb{Z}_p^{(2)}$ &  $\mathbb{Z}_p^{(3)}$ &  $\mathbb{Z}_p^{(2)}$ & $\mathbb{Z}_p^{(4)}$ & $\mathbb{Z}_p^{(7)}$ & $[\alpha_1,\alpha^\phi], [\alpha_2,\alpha^\phi], [\alpha_2,\alpha_1^\phi], $ \\
&&&&&&$[\alpha_3,\alpha^\phi] $ \\
[.5ex] 
 \hline
 \end{tabular}
\end{table}
\end{thm}
\begin{proof}
Schur multipliers of groups of order $p^4$ are taken  from \cite{KO} for $|G'|=p$ and from \cite[page. 4177]{EG} for $|G'|=p^2$. The Schur multipliers of groups of order $p^3$ are well known. So we mainly work for computing  the exterior squares. Tensor squares will then follow easily by Theorem \ref{E}.
We work in the group $\nu(G)$. 

In the class $\Phi_2$, we only work out the exterior square of $\Phi_2(211)b$. The other cases go on the same lines.

\par 

Consider the group $G:=\Phi_2(211)b$.
Since $|\M(G)|=p^2$ and $|\gamma_2(G)| = p$, it follows that $|G\wedge G|=p^3$.
By Lemma \ref{R}(viii), we have
$$[\alpha_2,\gamma^\phi] = [\alpha_1,\alpha,\gamma^\phi]=1, ~\text{as}~ \gamma \in \Z(G). $$
By Lemma \ref{O}  the following identities hold:
\begin{eqnarray*}
&& [\alpha_2,\alpha^\phi] = [\gamma^p,\alpha^\phi]=[\gamma,\alpha^\phi]^p=[\gamma,{(\alpha^p)}^{\phi}]=1,\\
 && [\alpha_2,\alpha_1^\phi] = [\gamma^p,\alpha_1^\phi]=[\gamma,\alpha_1^\phi]^p=[\gamma,{(\alpha_1^p)}^\phi]=1,\\
 &&  [\alpha_1,\alpha^\phi]^p = [\alpha_1^p,\alpha^\phi] =1.
 \end{eqnarray*}
These identities, along with  Proposition \ref{B} and Lemma \ref{S}, imply that $G \wedge G$ is generated by $[\alpha_1,\alpha^\phi], [\gamma,\alpha^\phi], [\gamma,\alpha_1^\phi] $, all of which are of order $p$. Hence 
$$G \wedge G \cong \mathbb{Z}_p^{(3)}.$$

\par

Now we work out the exterior square of the group $G=\Phi_3(1^4)$.
Since $|\M(G)|=p^2$, $|G\wedge G|=p^4$.
By Proposition \ref{B}, $G \wedge G$ is generated by the set
$$\{[\alpha_1,\alpha^\phi], [\alpha_2,\alpha^\phi], [\alpha_2,\alpha_1^\phi],[\alpha_3,\alpha^\phi], [\alpha_3,\alpha_1^\phi], [\alpha_2,\alpha_3^\phi]\}.$$
By  Lemma \ref{R}(viii), we have
$$ [\alpha_2,\alpha_3^\phi]=[\alpha_1,\alpha,\alpha_3^\phi]=1~\text{as}~ \alpha_3 \in \Z(G).$$
Hall-Witt identity yields
\begin{eqnarray*}
 1 &=& [\alpha_2,\alpha,\alpha_1^\phi]^{\alpha^{-1}}[\alpha^{-1},\alpha_1^{-1},\alpha_2^\phi]^{\alpha_1}\hspace{2cm}\\
&=&[\alpha_3,\alpha_1^\phi][\alpha_1\alpha\alpha_2^{-1}\alpha^{-1}\alpha_1^{-1},\alpha_2^{\phi}]^{\alpha_1}\\
&=&[\alpha_3,\alpha_1^\phi][\alpha\alpha_2^{-1}\alpha^{-1},\alpha_2^{\phi}]\\
&=&[\alpha_3,\alpha_1^\phi][\alpha_3\alpha_2^{-1},\alpha_2^{\phi}]\\
&=&[\alpha_3,\alpha_1^\phi][\alpha_3,\alpha_2^\phi][\alpha_2,\alpha_2^\phi]^{-1}\\
&=&[\alpha_3,\alpha_1^\phi][\alpha_2,\alpha_2^\phi]^{-1}.
\end{eqnarray*}
Consequently, $[\alpha_3,\alpha_1^\phi] = [\alpha_2,\alpha_2^\phi] =1$ in $G \wedge G$.
Also, by Lemma \ref{O}, we get
$$[\alpha_3,\alpha^\phi]^p=[\alpha_3^p,\alpha^\phi]=1=[\alpha_2^p,\alpha_1^\phi]=[\alpha_2,\alpha_1^\phi]^p.$$
and
$$[\alpha_2,\alpha^\phi]^p = [\alpha_2^p,\alpha^\phi]=1= [\alpha_1^p,\alpha^\phi]=[\alpha_1,\alpha^\phi]^p .$$
Hence, in view of   Lemma \ref{R}(vi), 
$$G \wedge G \cong \langle[\alpha_1,\alpha^\phi], [\alpha_2,\alpha^\phi], [\alpha_2,\alpha_1^\phi],[\alpha_3,\alpha^\phi]\rangle \cong \mathbb{Z}_p^{(4)}.$$
The remaining cases go on the same lines.
\hfill$\Box$

\end{proof}

We remark that the $G \wedge G$ and $G \otimes G$ for all groups $G$ of order $p^4$ have been computed in  \cite{PNi}. We recompute $G \wedge G$,  for all groups $G$ of order $p^4$, on the way to computing the minimal generating set for $G \wedge G$. 


\section{Classes $\Phi_4, \Phi_5$: Groups of class 2 with $G/G'$ elementary abelian}

Throughout the article, we make calculations in the subgroup $[G, G^\phi]$ of $\nu(G)$ modulo $\nabla(G)$, i.e., we work in $G \wedge G$. For commutator and power calculations, we use Lemma \ref{O} without any further reference.

First we consider extra-special $p$-groups of order $p^5$, for which  we have the following result.
\begin{lemma}
If $G$ is one of the groups $\Phi_5(2111)$ or $\Phi_5(1^5)$, then $\M(G)$ is isomorphic to $\mathbb{Z}_p^{(5)}$ and $G \wedge G$ is isomorphic to $\mathbb{Z}_p^{(6)}$.
\end{lemma}
\begin{proof}
The groups $G$ are extra-special groups. So it follows from  \cite[Theorem 3.3.6(i)]{GK} that $\M(G)$ is an elementary abelian $p$-group of order $p^5$.
By \cite[Corollary 2.3]{NR}, we have 
$$G \otimes G \cong \mathbb{Z}_p^{(16)}.$$ 
Now by Theorem \ref{E},
$$G \wedge G \cong \mathbb{Z}_p^{(6)},$$
which completes the proof.    \hfill $\Box$

\end{proof}

Now let $G$ be a finite $p$-group of class 2 such that  $G/G'$ and $G'$ are elementary abelian  of order $p^3$ and $p^2$ respectively.  We consider $G/G'$ and $G'$ as vector spaces over $\mathbb{F}_p$, which we denote by $V, W$ respectively. The bilinear map $(-,-) : V \times V \to W$ is defined by 
$$(v_1,v_2)=[g_1,g_2]$$
for  $v_1,v_2 \in V$ such that $v_i=g_iG', i \in \{1,2\}$.

The following construction is from \cite{BE}. Let $X_1$ be the subspace of $V \otimes W$ spanned by all 
$$v_1 \otimes (v_2,v_3) + v_2 \otimes (v_3,v_1) + v_3 \otimes (v_1,v_2)$$
for $v_1, v_2, v_3 \in V$.

Consider a map $f : V \rightarrow W$ given by 
$$f(gG')=g^p$$
 for $g \in G$.
We denote by $X_2$, the subspace spanned by all $v \otimes f(v)$, $v \in V$,  and take $$X:=X_1+X_2.$$

Now consider a homomorpism $\sigma: V\wedge V \rightarrow (V \otimes W)/X$ given by
$$\sigma(v_1 \wedge v_2)= \big(v_1 \otimes f(v_2)+ (_2^p)  v_2 \otimes (v_1,v_2)\big)+X.$$
Then there exists an abelian group $M^*$  admitting a subgroup $N$, isomorphic to  $(V \otimes W)/X$, such that 
$$1 \rightarrow N \rightarrow M^* \xrightarrow{\xi} V \wedge V \rightarrow 1$$
is exact.
Now we consider a homomorphism $\rho:V \wedge V \rightarrow W$ given by
$$\rho(v_1 \wedge v_2)=(v_1,v_2)$$ 
for all $v_1,v_2 \in V.$  Notice that $\rho$ is an epimorphism.
Denote by $M$, the subgroup of $M^*$ containing $N$ such that $M/N \cong \Ker \rho$. Then it follows that  $|M/N|=|V\wedge V|/|W|$, which will be used for calculating the order of $M$  without any further reference.

With the above setting, we have 
\begin{thm}$($\cite[Theorem 3.1]{BE}$)$ \label{thmBE}
 $\M(G) \cong M$.
\end{thm}

The following information will be used throughout the rest of this section without further reference.

\begin{remark}
Let $G$ be any group in the  isoclinism class $\Phi_4$.  Consider the natural epimorphism 
$$[G, G^\phi]\rightarrow [G/\Z(G), (G/\Z(G))^\phi].$$ 
Since $G/\Z(G) \wedge G/\Z(G)$ is elementary abelian of order $p^3$, it follows that the elements
$[\alpha_1,\alpha^\phi], [\alpha_2,\alpha^\phi], [\alpha_2,\alpha_1^\phi]$ are non-trivial and independent in $G \wedge G$. Furthermore, by Lemma \ref{S}, $G \otimes G$ is abelian.
\end{remark}

\begin{lemma}
If $G$ is one of the groups $\Phi_4(221)a$, $\Phi_4(221)b$, $\Phi_4\left(221\right)c$,  
$\Phi_4(221)d_{\frac{1}{2}(p-1)}$, $\Phi_4(221)d_r \;\big(r \ne \frac{1}{2}(p-1)\big)$,  
$\Phi_4(221)e$,
$\Phi_4(221)f_0$ or $\Phi_4(221)f_r$, then $\M(G)$ is isomorphic to $\mathbb{Z}_p, \mathbb{Z}_p \times \mathbb{Z}_p, \mathbb{Z}_p$,   $\mathbb{Z}_{p^2}$,   $\mathbb{Z}_p$,
$\mathbb{Z}_p,    \mathbb{Z}_{p^2}$ or $\mathbb{Z}_p$ respectively, and $G \wedge G$ is isomorphic to $\mathbb{Z}_p^{(3)}, \mathbb{Z}_{p^2} \times \mathbb{Z}_p \times \mathbb{Z}_p, \mathbb{Z}_p^{(3)}$,  $\mathbb{Z}_{p^2} \times \mathbb{Z}_p \times \mathbb{Z}_p$,   
$\mathbb{Z}_p^{(3)}$,   $\mathbb{Z}_p^{(3)}, \mathbb{Z}_{p^2} \times \mathbb{Z}_p \times \mathbb{Z}_p$ or $\mathbb{Z}_p^{(3)}$ respectively.
\end{lemma}
\begin{proof}
For the group $G=\Phi_4(221)a$, notice that $X_1$ is spanned by 
$$(\alpha_1G' \otimes \alpha^p - \alpha_2G' \otimes \alpha_1^p)$$ and $\dim X_1=1$. Observe that $\alpha G' \otimes \alpha^p, \alpha_1 G' \otimes \alpha_1^p \in X_2$ and $(\alpha G'+uG')\otimes \alpha^p, (\alpha_1 G'+uG') \otimes \alpha_1^p \in X_2$ for $uG' \in \Ker f$. So, we have $uG' \otimes \alpha^p, uG' \otimes \alpha_1^p \in X_2$ for all $uG' \in \Ker f$. This implies $\alpha_2 G' \otimes \alpha^p, \alpha_2 G' \otimes \alpha_1^p \in X_2$.
Now a general element of $X_2$ is of the form
\begin{eqnarray*}
(p_1\alpha G'+p_2\alpha_1G' + p_3\alpha_2G')\otimes (p_1\alpha^p+p_2\alpha_1^p) &=& p_1^2\alpha G' \otimes \alpha^p + p_2^2\alpha_1 G' \otimes \alpha_1^p\\
&& +p_1p_2(\alpha G' \otimes \alpha_1^p + \alpha_1G' \otimes \alpha^p)\\
&& +p_3p_1(\alpha_2 G' \otimes \alpha^p)\\
&& +p_3p_2(\alpha_2 G' \otimes \alpha_1^p).
\end{eqnarray*}
This shows that, $X_2$ is spanned by the set
$$\{ \alpha G' \otimes \alpha^p, \alpha_1 G' \otimes \alpha_1^p, \alpha_2G' \otimes \alpha^p, \alpha_2G' \otimes \alpha_1^p, (\alpha G' \otimes \alpha_1^p+ \alpha_1G' \otimes \alpha^p)\}.$$ 
Hence, $\dim X_2=5$.
Observe that $(\alpha_1G' \otimes \alpha^p - \alpha_2G' \otimes \alpha_1^p)$ is not contained in $X_2$. Thus,
$\dim X=6$, and consequently,   $|N|=1, |M|=p$. Now by Theorem \ref{thmBE}, we have
$$\M(G) \cong \mathbb{Z}_p,$$ which gives
 $$|G \wedge G|=p^3.$$
Hence 
$$G \wedge G=\langle [\alpha_1,\alpha^\phi], [\alpha_2,\alpha^\phi], [\alpha_1,\alpha_2^\phi] \rangle \cong \mathbb{Z}_p^{(3)}.$$

\par

For the group $G=\Phi_4(221)b$, we see that $X_1$ is spanned by $$(\alpha_1G' \otimes \alpha^p - \alpha_2G' \otimes \alpha_2^p)$$ and $\dim X_1=1$.
As described in the preceding case, the subspace $X_2$ is spanned by the set
$$\{\alpha G' \otimes \alpha^p, \alpha_2 G' \otimes \alpha_2^p, \alpha_1G' \otimes \alpha^p, \alpha_1G' \otimes \alpha_2^p, (\alpha G' \otimes \alpha_2^p+ \alpha_2G' \otimes \alpha^p)\},$$ 
and $\dim X_2=5$.
Observe that $X_1 \subset X_2$, so  $\dim X=5$ and  $|N|=p$. By Theorem \ref{thmBE} $|M|=|\M(G)|=p^2$. Hence, $|G \wedge G|=p^4$.

By Lemma \ref{R}(viii), $$[\beta_1,\beta_2^\phi]=[\alpha_1,\alpha,\beta_2^\phi]=1.$$
For $i \in \{1,2\}, x \in \{\alpha,\alpha_1,\alpha_2\}$, we have (by Lemma \ref{O}),
\begin{eqnarray*}
&& [\beta_i, x^{\phi}]^p=[\beta_i^p,x]=1,\hspace{2cm}\\
&&[\beta_2,x^\phi]= [\alpha^p,x^\phi]=[\alpha,x^\phi]^p,\\
&&[\beta_1,x^\phi]=[\alpha_2^p,x^\phi]=[\alpha_2,x^\phi]^p, \\
&& [\alpha,\alpha_1^\phi]^p=[\alpha,{(\alpha_1^p)^\phi}]=1=[\alpha_2,{(\alpha_1^p)^\phi}]=[\alpha_2,\alpha_1^\phi]^p,\\
&& [\alpha_2,\alpha^\phi]^{p^2}=[\alpha_2^{p^2},\alpha^\phi]=1.
\end{eqnarray*}
Hence, it follows that 
$$G \wedge G =\langle [\alpha_2,\alpha^\phi], [\alpha_1,\alpha^\phi], [\alpha_1,\alpha_2^\phi]\rangle \cong \mathbb{Z}_{p^2} \times \mathbb{Z}_p \times \mathbb{Z}_p$$
and
$$\M(G) \cong \langle [\alpha_2,\alpha^\phi]^p, [\alpha_1,\alpha_2^\phi]\rangle  \cong\mathbb{Z}_p \times \mathbb{Z}_p.$$

\par

For the group $G=\Phi_4(221)c$, we see that $X_1$ is spanned by 
$$(\alpha_1G' \otimes \alpha_2^p - \alpha_2G' \otimes \alpha_1^p),$$
and $\dim X_1=1$. 
It follows that $X_2$ is spanned by the set
$$\{\alpha_1G' \otimes \alpha_1^p, \alpha_2 G' \otimes \alpha_2^p, \alpha G' \otimes \alpha_1^p, \alpha G' \otimes \alpha_2^p, (\alpha_1 G' \otimes \alpha_2^p+ \alpha_2G' \otimes \alpha_1^p)\}$$ and $\dim X_2=5$.
So $\dim X=6$ and $|N|=1, |M|=p$. Hence 
$$\M(G) \cong \mathbb{Z}_p$$ and
$$G \wedge G\cong\langle [\alpha_1,\alpha^\phi], [\alpha_2,\alpha^\phi], [\alpha_1,\alpha_2^\phi] \rangle \cong \mathbb{Z}_p^{(3)}.$$

\par

For the group $G=\Phi_4(221)d_r$, $X_1$ is spanned by 
$$(\alpha_1G' \otimes \alpha_2^p - \alpha_2G' \otimes \beta_1)$$ and $\dim X_1=1$.
The space $X_2$ is spanned by the set
$$\{\alpha_1G' \otimes \beta_1^k, \alpha_2 G' \otimes \alpha_2^p, \alpha G' \otimes \beta_1^k, \alpha G' \otimes \alpha_2^p, (\alpha_1 G' \otimes \alpha_2^p+ \alpha_2G' \otimes \beta_1^k)\}$$
and $\dim X_2=5$. 

For $r=\frac{1}{2}(p-1)$, it follows that in the presentation of $G$ $k \equiv  -1$ (mod $p$).
So we have $X_1 \subset X_2$. Hence $\dim X=5$, and therefore $|N|=p,$ $|\M(G)|=p^2$ and $|G \wedge G|=p^4$.
By Lemma \ref{R}(viii), $$[\beta_1,\beta_2^\phi]=[\alpha_1,\alpha,\beta_2^\phi]=1.$$
For $i \in \{1,2\}, x \in \{\alpha,\alpha_1,\alpha_2\}$, we have
\begin{eqnarray*}
&& [\beta_i, x^{\phi}]^p=[\beta_i^p,x]=1,\hspace{2cm}\\
&&[\beta_2,x^\phi]= [\alpha_2^p,x^\phi]=[\alpha_2,x^\phi]^p,\\
&&[\beta_1,x^\phi]=[\alpha_1^{-p},x^\phi]=[\alpha_1^{-1},x^\phi]^p=[\alpha_1,x^\phi]^{-p}, \\
&& [\alpha,\alpha_1^\phi]^p=[\alpha^p,\alpha_1^\phi]=1=[\alpha_2,{(\alpha^p)^\phi}]=[\alpha_2,\alpha^\phi]^p,\\
&& [\alpha_1,\alpha_2^\phi]^{p^2}=[\alpha_1^{p^2},\alpha_2^\phi]=1.
\end{eqnarray*}
Hence, it follows that  $$G \wedge G= \langle [\alpha_1,\alpha_2^\phi] , [\alpha_1,\alpha^\phi], [\alpha_2,\alpha^\phi]\rangle \cong \mathbb{Z}_{p^2} \times \mathbb{Z}_p \times \mathbb{Z}_p$$ 
and, since $[\alpha_1, \alpha_2] = 1$, 
 $$\M(G) \cong \langle [\alpha_1,\alpha_2^\phi]\rangle \cong  \mathbb{Z}_{p^2}.$$ 

For $r \neq \frac{1}{2}(p-1)$, $X_1 \cap X_2= \emptyset$, so  $\dim X=6$ and therefore $|N|=1$. Hence 
$$\M(G) \cong \mathbb{Z}_p$$ and
$$G \wedge G=\langle [\alpha_1,\alpha^\phi], [\alpha_2,\alpha^\phi], [\alpha_1,\alpha_2^\phi] \rangle \cong \mathbb{Z}_p^{(3)} .$$

\par

For the group $G=\Phi_4(221)e$, $X_1$ is spanned by $$(\alpha_1G' \otimes \beta_2 - \alpha_2G' \otimes \beta_1)$$ and $\dim X_1=1$.
The subspace $X_2$ is spanned by the set
\begin{eqnarray*}
&&\{\alpha_1 G' \otimes \beta_2^{-\frac{1}{4}}, (\alpha_2 G' \otimes \beta_1+\alpha_2G' \otimes \beta_2), \alpha G' \otimes \beta_2^{-\frac{1}{4}},(\alpha G' \otimes \beta_1+ \alpha G' \otimes \beta_2), \\
&&\hspace{4 cm} (\alpha_1 G' \otimes \beta_1+\alpha_1 G' \otimes \beta_2+ \alpha_2G' \otimes \beta_2^{-\frac{1}{4}})\}
\end{eqnarray*} 
and $\dim X_2=5$. So $\dim X=6$. Therefore $|N|=1$ and $|M|=|\M(G)|=p$. Hence
$$\M(G) \cong \mathbb{Z}_p$$
and
$$G \wedge G=\langle [\alpha_1,\alpha^\phi], [\alpha_2,\alpha^\phi], [\alpha_1,\alpha_2^\phi] \rangle \cong \mathbb{Z}_p^{(3)}.$$

\par

For the group $G=\Phi_4(221)f_0$, $X_1$ is spanned by $$(\alpha_1G' \otimes \beta_2 - \alpha_2G' \otimes \beta_1)$$ and $\dim X_1=1$.
The subspace $X_2$ is spanned by $$\{\alpha_1 G' \otimes \beta_2, \alpha_2 G' \otimes \beta_1^\nu, \alpha G' \otimes \beta_2, \alpha G' \otimes \beta_1^\nu, (\alpha_1 G' \otimes \beta_1^\nu+ \alpha_2G' \otimes \beta_2)\}$$ and $\dim X_2=5$.
Observe that  $X_1 \subset X_2$, so $\dim X=5$. Therefore $|N|=p, |M|=|\M(G)|=p^2$. 
By Lemma \ref{R}(viii), $$[\beta_1,\beta_2^\phi]=[\alpha_1,\alpha,\beta_2^\phi]=1.$$
For $i \in \{1,2\}, x \in \{\alpha,\alpha_1,\alpha_2\}$, we have
\begin{eqnarray*}
&& [\beta_i, x^{\phi}]^p=[\beta_i^p,x^\phi]=1,\hspace{2cm}\\
&&[\beta_2,x^\phi]= [\alpha_1^p,x^\phi]=[\alpha_1,x^\phi]^p,\\
&&[\beta_1,x^\phi]=[\alpha_2^{p\nu^{-1}},x^\phi]=[\alpha_2,x^\phi]^{p\nu^{-1}}, \\
&& [\alpha,\alpha_1^\phi]^p=[\alpha^p,{\alpha_1^\phi}]=1=[\alpha_2,{(\alpha^p)^\phi}]=[\alpha_2,\alpha^\phi]^p.
\end{eqnarray*}
Hence 
$$G \wedge G=\langle [\alpha_1,\alpha_2^\phi], [\alpha_2,\alpha^\phi], [\alpha_1,\alpha^\phi] \rangle \cong \mathbb{Z}_{p^2} \times \mathbb{Z}_p \times \mathbb{Z}_p$$
and, since $[\alpha_1, \alpha_2] = 1$, 
$$\M(G) \cong \langle [\alpha_1,\alpha_2^\phi] \rangle \cong \mathbb{Z}_{p^2}.$$

\par

On the same lines, we can prove that  $$\M(\Phi_4(221)f_r) \cong \mathbb{Z}_p$$ and
$$\Phi_4(221)f_r \wedge \Phi_4(221)f_r=\langle [\alpha_1,\alpha^\phi], [\alpha_2,\alpha^\phi], [\alpha_1,\alpha_2^\phi] \rangle \cong \mathbb{Z}_p^{(3)}.$$
This completes the proof of the lemma.
\hfill$\Box$

\end{proof}

\begin{lemma}
If $G$ is one of the groups $\Phi_4(2111)a, \Phi_4(2111)b, \Phi_4(2111)c$ or $ \Phi_4(1^5)$, then $\M(G)$ is isomorphic to $\mathbb{Z}_p^{(3)},\mathbb{Z}_p^{(3)},\mathbb{Z}_p^{(3)}$ or $\mathbb{Z}_p^{(6)}$ respectively, and $G \wedge G$ is isomorphic to $\mathbb{Z}_p^{(5)},\mathbb{Z}_p^{(5)},\mathbb{Z}_p^{(5)}$ or $\mathbb{Z}_p^{(8)}$respectively.
\end{lemma}
\begin{proof}
For the group $G=\Phi_4(2111)a$, $X_1$ is spanned by the element
$$(\alpha_1G' \otimes \alpha^p - \alpha_2G' \otimes \beta_1)$$ and $\dim X_1=1$. The space $X_2$ is spanned by the set
$$\{\alpha G' \otimes \alpha^p, \alpha_1 G' \otimes \alpha^p, \alpha_2 G' \otimes \alpha^p\}$$ and $\dim X_2=3$. So $\dim X=4$. Therefore  $|N|=p^2$, $|\M(G)|=|M|=p^3$ and $|G \wedge G|=p^5$.
By Lemma \ref{R}(viii), 
$$[\beta_1,\beta_2^\phi]=[\alpha_1,\alpha,\beta_2^\phi]=1.$$

For $i \in \{1,2\}, x \in \{\alpha,\alpha_1,\alpha_2\}$, we have 
\begin{eqnarray*}
&& [\beta_i, x^{\phi}]^p=[\beta_i^p,x]=1, \\
&&[\alpha_2,x^\phi]^p=[\alpha_2^p,x^\phi]=1=[\alpha_1^p,\alpha^\phi]=[\alpha_1,\alpha^\phi]^p.
\end{eqnarray*}
Therefore every generator of $G \wedge G$ is of order at most $p$.
Hence 
$$G\wedge G  \cong \mathbb{Z}_p^{(5)}$$
and
$$\M(G)\cong \mathbb{Z}_p^{(3)}.$$

\par

In the remaining three cases, the proof  goes on the same lines.
\hfill $\Box$

\end{proof}

\section{Groups of maximal class}
\begin{lemma}
If $G$ is one of the groups $\Phi_9(2111)a,\Phi_9(2111)b_r, \Phi_{10}(2111)a_r$ or $\Phi_{10}(2111)b_r$, then $\M(G)$ is isomorphic to $\mathbb{Z}_p$ and $G \wedge G$ is isomorphic to $ \mathbb{Z}_p^{(4)}$.  
\end{lemma}
\begin{proof}
For the groups $G$ under consideration, taking $K=\Z(G)$, in Theorem \ref{2}(i), we get $p \le \M(G)$.  Thus, $|G \wedge G| \geq p^4$. 
\par

If $G$ is either $\Phi_9(2111)a$ or $\Phi_9(2111)b_r$, then by Lemma \ref{R}(viii),
 $$[\alpha_2,\alpha_4^\phi]=[\alpha_1,\alpha,\alpha_4^\phi]=1=[\alpha_2,\alpha,\alpha_4^\phi]=[\alpha_3,\alpha_4^\phi], $$ as $\alpha_4 \in \Z(G)$.
Now we have
\begin{eqnarray*}
[\alpha^{-1},\alpha_1^{-1},\alpha_2^\phi]^{\alpha_1} &=& [\alpha_1\alpha \alpha_2^{-1} \alpha^{-1}\alpha_1^{-1},\alpha_2^\phi]^{\alpha_1}=[\alpha \alpha_2^{-1}\alpha^{-1},\alpha_2^{\phi}]\\
&=& [\alpha\alpha_3\alpha^{-1} \alpha_2^{-1},\alpha_2^{\phi}]=[\alpha_3\alpha_4^{-1}\alpha_2^{-1},\alpha_2^{\phi}]\\
&=& [\alpha_3,\alpha_2^\phi][\alpha_4,\alpha_2^\phi]^{-1}=[\alpha_3,\alpha_2^\phi],
\end{eqnarray*}
\begin{eqnarray*}
[\alpha^{-1},\alpha_1^{-1},\alpha_3^\phi]^{\alpha_1} &=& [\alpha^{-1},\alpha_1^{-1},\alpha_3^\phi]^{\alpha_1} = [\alpha_1\alpha \alpha_2^{-1} \alpha^{-1}\alpha_1^{-1},\alpha_3^\phi]^{\alpha_1}\\
&=& [\alpha \alpha_2^{-1}\alpha^{-1},\alpha_3^{\phi}]=[\alpha_3\alpha_4^{-1}\alpha_2^{-1},\alpha_3^{\phi}]\\
&=&[\alpha_4,\alpha_3^{\phi}]^{-1}[\alpha_2,\alpha_3^\phi]^{-1}=[\alpha_2,\alpha_3^\phi]^{-1}.
\end{eqnarray*}
By Hall-Witt identity, we have
\begin{eqnarray*}
1&=&[\alpha_2,\alpha,\alpha_1^\phi]^{\alpha^{-1}}[\alpha^{-1},\alpha_1^{-1},\alpha_2^\phi]^{\alpha_1}[\alpha_1,\alpha_2^{-1},{(\alpha^{-1})}^\phi]^{\alpha_2}\\
&=&[\alpha_3,\alpha_1^\phi][\alpha_3,\alpha_2^\phi],\\
1&=&[\alpha_3,\alpha,\alpha_1^\phi]^{\alpha^{-1}}[\alpha^{-1},\alpha_1^{-1},\alpha_3^\phi]^{\alpha_1}[\alpha_1,\alpha_3^{-1},{(\alpha^{-1})}^\phi]^{\alpha_3}\\
&=&[\alpha_4,\alpha_1^\phi][\alpha_2,\alpha_3^\phi]^{-1}.
\end{eqnarray*}
This implies that $[\alpha_3,\alpha_1^\phi]=[\alpha_2,\alpha_3^\phi]=[\alpha_4,\alpha_1^\phi]$.

Now consider the group $G\cong \Phi_9(2111)b_r$.
We have the following identities:
\begin{eqnarray*}
&&[\alpha_4,\alpha^\phi]^p=[\alpha_4^p,\alpha^\phi]=1=[\alpha_4^p,\alpha_1^\phi]=[\alpha_4,\alpha_1^\phi]^p,\\
&&[\alpha_3,\alpha^\phi]^p=[\alpha_3^p,\alpha^\phi]=1=[\alpha_3^p,\alpha_1^\phi]=[\alpha_3,\alpha_1^\phi]^p,\\
&&[\alpha_2,\alpha^\phi]^p=[\alpha_2^p,\alpha^\phi]=1=[\alpha_2^p,\alpha_1^\phi]=[\alpha_2,\alpha_1^\phi]^p, \\
&&[\alpha_4,\alpha^\phi]=[\alpha_1^{pk^{-1}},\alpha^\phi]=[\alpha_1,\alpha^\phi]^{pk^{-1}}=[\alpha_1,(\alpha^{pk^{-1}})^\phi]=1,\\
&&[\alpha_4,\alpha_1^\phi]=[\alpha_1^{pk^{-1}},\alpha_1^\phi]=[\alpha_1,\alpha_1^\phi]^{pk^{-1}}=1,\\
&& [\alpha_1,\alpha^\phi]^p=[\alpha_1,(\alpha^p)^\phi]=1. \\
\end{eqnarray*}
Thus, by Proposition \ref{B}, $G \wedge G$ is generated by 
$$\{[\alpha_1,\alpha^\phi], [\alpha_2,\alpha^\phi], [\alpha_3,\alpha^\phi], [\alpha_2,\alpha_1^\phi]\}.$$ 
Since, by Lemma \ref{R}(vi), $[[\alpha_2,\alpha^\phi],[\alpha_1,\alpha^\phi]]=[\alpha_3,\alpha_2^\phi]=1$, it follows that $G \wedge G$ is elementary abelian of order $p^4$.
Hence 
$$G \wedge G \cong  \mathbb{Z}_p^{(4)},$$
and consequently
$$\M(G) \cong \langle [\alpha_2,\alpha_1^\phi]\rangle\cong \mathbb{Z}_p.$$

\par

Similarly, we can compute
$$\Phi_9(2111)a \wedge \Phi_9(2111)a \cong \mathbb{Z}_p^{(4)}
~ \text{and}~
\M(\Phi_9(2111)a) \cong \mathbb{Z}_p.$$

\par

If $G$ is either $\Phi_{10}(2111)a_r$ or $\Phi_{10}(2111)b_r$, then by Lemma \ref{R}(viii),
 $$[\alpha_2,\alpha_4^\phi]=[\alpha_1,\alpha,\alpha_4^\phi]=1=[\alpha_2,\alpha,\alpha_4^\phi]=[\alpha_3,\alpha_4^\phi], $$
\begin{eqnarray*}
[\alpha^{-1},\alpha_1^{-1},\alpha_2^\phi]^{\alpha_1} &=& [\alpha_1\alpha \alpha_2^{-1} \alpha^{-1}\alpha_1^{-1},\alpha_2^\phi]^{\alpha_1}=[\alpha \alpha_2^{-1}\alpha^{-1},(\alpha_1^{-1}\alpha_2\alpha_1)^{\phi}]\\
&=& [\alpha\alpha_3\alpha^{-1} \alpha_2^{-1},(\alpha_2\alpha_4^{-1})^{\phi}]=[\alpha_3\alpha_4^{-1}\alpha_2^{-1},(\alpha_2\alpha_4^{-1})^{\phi}]\\
&=& [\alpha_3,\alpha_2^\phi][\alpha_4,\alpha_2^\phi]^{-1}=[\alpha_3,\alpha_2^\phi]
\end{eqnarray*}
and 
\begin{eqnarray*}
[\alpha^{-1},\alpha_1^{-1},\alpha_3^\phi]^{\alpha_1} &=& [\alpha^{-1},\alpha_1^{-1},\alpha_3^\phi]^{\alpha_1} = [\alpha_1\alpha \alpha_2^{-1} \alpha^{-1}\alpha_1^{-1},\alpha_3^\phi]^{\alpha_1}\hspace{1.1cm}\\
&=& [\alpha \alpha_2^{-1}\alpha^{-1},\alpha_3^{\phi}]=[\alpha\alpha_3\alpha^{-1} \alpha_2^{-1},\alpha_3^{\phi}]\\
&=& [\alpha_3\alpha_4\alpha_2^{-1},\alpha_3^{\phi}]=[\alpha_4,\alpha_3^\phi][\alpha_2,\alpha_3^\phi]^{-1}\\
&=& [\alpha_2,\alpha_3^\phi]^{-1}.
\end{eqnarray*}
By Hall-Witt identity, we have
\begin{eqnarray*}
1&=&[\alpha_2,\alpha,\alpha_1^\phi]^{\alpha^{-1}}[\alpha^{-1},\alpha_1^{-1},\alpha_2^\phi]^{\alpha_1}[\alpha_1,\alpha_2^{-1},{(\alpha^{-1})}^\phi]^{\alpha_2}\\
&=&[\alpha_3,\alpha_1^\phi][\alpha_3,\alpha_2^\phi][\alpha_4^{-1},(\alpha^{-1})^\phi]^{\alpha_2}\\
&=&[\alpha_3,\alpha_1^\phi][\alpha_3,\alpha_2^\phi][\alpha_4,\alpha^\phi],\\
1&=&[\alpha_3,\alpha,\alpha_1^\phi]^{\alpha^{-1}}[\alpha^{-1},\alpha_1^{-1},\alpha_3^\phi]^{\alpha_1}[\alpha_1,\alpha_3^{-1},{(\alpha^{-1})}^\phi]^{\alpha_3}\\
&=&[\alpha_4,\alpha_1^\phi][\alpha_2,\alpha_3^\phi]^{-1}.
\end{eqnarray*}
This implies that $[\alpha_2,\alpha_3^\phi]=[\alpha_4,\alpha_1^\phi]=[\alpha_4,\alpha^\phi][\alpha_3,\alpha_1^\phi]$.

Consider the group $G\cong \Phi_{10}(2111)a_r$.
We have the following identities:
\begin{eqnarray*}
&&[\alpha_4,\alpha^\phi]^p=[\alpha_4^p,\alpha^\phi]=1=[\alpha_4^p,\alpha_1^\phi]=[\alpha_4,\alpha_1^\phi]^p,\\
&&[\alpha_3,\alpha^\phi]^p=[\alpha_3^p,\alpha^\phi]=1=[\alpha_3^p,\alpha_1^\phi]=[\alpha_3,\alpha_1^\phi]^p,\\
&&[\alpha_2,\alpha^\phi]^p=[\alpha_2^p,\alpha^\phi]=1=[\alpha_2^p,\alpha_1^\phi]=[\alpha_2,\alpha_1^\phi]^p, \\
&&[\alpha_4,\alpha^\phi]=[\alpha^{pk^{-1}},\alpha^\phi]=[\alpha,\alpha^\phi]^{pk^{-1}}=1,\\
&&[\alpha_4,\alpha_1^\phi]=[\alpha^{pk^{-1}},\alpha_1^\phi]=[\alpha,\alpha_1^\phi]^{pk^{-1}}=[\alpha,(\alpha_1^{pk^{-1}})^\phi]=1,\\
&& [\alpha_1,\alpha^\phi]^p=[\alpha_1^p,\alpha^\phi]=1.
\end{eqnarray*}
Thus $G \wedge G$ is generated by 
$$\{[\alpha_1,\alpha^\phi], [\alpha_2,\alpha^\phi], [\alpha_3,\alpha^\phi], [\alpha_2,\alpha_1^\phi]\}.$$ 
By Lemma \ref{R}(vi), $G \wedge G$ is elementary abelian of order $p^4$.
Hence 
$$G \wedge G \cong  \mathbb{Z}_p^{(4)},$$
and consequently
$$\M(G) \cong \langle [\alpha_2,\alpha_1^\phi][\alpha_3,\alpha^\phi] \rangle \cong \mathbb{Z}_p.$$

\par

Similarly,  we can compute
$$\Phi_{10}(2111)b_r \wedge \Phi_{10}(2111)b_r \cong \mathbb{Z}_p^{(4)}
~ \text{and}~
\M(\Phi_{10}(2111)b_r) \cong \mathbb{Z}_p.$$
This completes the proof.
\hfill $\Box$

\end{proof}
\begin{lemma}
If $G$ is one of the groups $\Phi_9(1^5)$ or $\Phi_{10}(1^5)$, then $\M(G)$ is isomorphic to $\mathbb{Z}_p^{(3)}$, and $G \wedge G$ is isomorphic to $\Phi_2(111) \times \mathbb{Z}_p^{(3)}$. 
\end{lemma}
\begin{proof}
Consider the group $G \cong \Phi_9(1^5)$. 
Let $F$ be the free group generated by $\{\alpha,\alpha_1\}$. 
Define $\alpha_{i+1}=[\alpha_i,\alpha], \; 1 \leq i\leq 3$. Set $\beta_1=[\alpha_1,\alpha_2]$, $\beta_2=[\alpha_4,\alpha]$ and $\beta_3=[\alpha_1,\alpha_4]$. Then, modulo $\gamma_6(F)$, we have 

\begin{eqnarray*}
[\alpha^{-1},\alpha_1^{-1},\alpha_3]^{\alpha_1} &=& [\alpha^{-1},\alpha_1^{-1},\alpha_3]^{\alpha_1} =[\alpha_1\alpha \alpha_2^{-1} \alpha^{-1}\alpha_1^{-1},\alpha_3]^{\alpha_1} \\
&=& [\alpha \alpha_2^{-1}\alpha^{-1},\alpha_3[\alpha_3,\alpha_1]]=[\alpha\alpha_3\alpha^{-1} \alpha_2^{-1},\alpha_3[\alpha_3,\alpha_1]]  \\
&=& [\alpha_3\alpha\alpha_4^{-1}\alpha^{-1}\alpha_2^{-1},\alpha_3[\alpha_3,\alpha_1]]=[\alpha_3\alpha\beta_2\alpha^{-1}\alpha_4^{-1}\alpha_2^{-1},\alpha_3[\alpha_3,\alpha_1]] \\
&=& [\alpha_3\beta_2\alpha_4^{-1}\alpha_2^{-1},\alpha_3[\alpha_3,\alpha_1]]=[\alpha_2^{-1},\alpha_3]\\
&=&[\alpha_2,\alpha_3]^{-1},
\end{eqnarray*}
\begin{eqnarray*}
[\alpha^{-1},\alpha_1^{-1},\alpha_2]^{\alpha_1}&=& [\alpha^{-1},\alpha_1^{-1},\alpha_2]^{\alpha_1}=[\alpha_1\alpha \alpha_2^{-1} \alpha^{-1}\alpha_1^{-1},\alpha_2]^{\alpha_1}\hspace{2.8cm}\\
&=& [\alpha_3\beta_2\alpha_4^{-1}\alpha_2^{-1},\alpha_2\beta_1^{-1}]=[\alpha_3,\alpha_2\beta_1^{-1}][\alpha_2^{-1},\alpha_2\beta_1^{-1}]\\
&=&[\alpha_3,\alpha_2][\alpha_2,\beta_1].
\end{eqnarray*}
By Hall-Witt identity, we have the following identities modulo $\gamma_6(F)$: 
\begin{eqnarray}
1&=&[\alpha_3,\alpha,\alpha_1]^{\alpha^{-1}}[\alpha^{-1},\alpha_1^{-1},\alpha_3]^{\alpha_1}[\alpha_1,\alpha_3^{-1},\alpha^{-1}]^{\alpha_3} \nonumber \\
&=&[\alpha_4,\alpha_1]^{\alpha^{-1}}[\alpha_2,\alpha_3]^{-1}[\alpha_3,\alpha_1,\alpha_3^{-1}\alpha^{-1}\alpha_3] \nonumber \\
&=&[\alpha_4,\alpha_1][\alpha_2,\alpha_3]^{-1}[\alpha_3,\alpha_1, \alpha_4\alpha^{-1}] \nonumber \\
&=&[\alpha_4,\alpha_1][\alpha_2,\alpha_3]^{-1}[\alpha_3,\alpha_1,\alpha^{-1}], \label{x}
\end{eqnarray}
\begin{eqnarray}
1&=&[\alpha_2,\alpha,\alpha_1]^{\alpha^{-1}}[\alpha^{-1},\alpha_1^{-1},\alpha_2]^{\alpha_1}[\alpha_1,\alpha_2^{-1},\alpha^{-1}]^{\alpha_2} \nonumber \\
&=&[\alpha_3,\alpha_1]^{\alpha^{-1}}[\alpha_3,\alpha_2][\alpha_2,\beta_1][\beta_1^{-1},\alpha_2^{-1}\alpha^{-1}\alpha_2] \nonumber \\
&=&[\alpha_3,\alpha_1][\alpha_3,\alpha_1, \alpha^{-1}][\alpha_3,\alpha_2][\alpha_2,\beta_1][\beta_1^{-1}, \alpha_3\alpha^{-1}] \nonumber \\
&=&[\alpha_3,\alpha_1][\alpha_3,\alpha_1, \alpha]^{-1}[\alpha_3,\alpha_2][\alpha_2,\beta_1][\beta_1^{-1},\alpha^{-1}] \nonumber \\
&=&[\alpha_3,\alpha_1][\alpha_3,\alpha_1, \alpha]^{-1}[\alpha_3,\alpha_2][\alpha_2,\beta_1][\beta_1,\alpha].\label{y}
\end{eqnarray}

Now consider
$$H_1 = F/\langle\gamma_6(F), F^p,  \beta_1, \beta_3, [\alpha_3,\alpha_1,\alpha], [\alpha_3,\alpha_1,\alpha_1]\rangle.$$ 
Using (\ref{x}) and (\ref{y}) we have  $[\alpha_4,\alpha_1]=[\alpha_2,\alpha_3]=[\alpha_3,\alpha_1]=1$ in $H_1$.
So 
$$H_1 \cong  \langle \alpha,\alpha_1,\alpha_2,\alpha_3,\alpha_4,\beta_2 \mid [\alpha_i,\alpha]=\alpha_{i+1}, [\alpha_4,\alpha]=\beta_2,\alpha^p=\alpha_1^p=\alpha_{i+1}^p=\beta_2^p=1 \;(i=1,2,3)\rangle,$$
which is the  group $\Phi_{35}(1^6)$ of order $p^6$  in \cite{RJ}.  
Now consider 
$$H_2=F/\langle\gamma_6(F), F^p, \beta_1,[\alpha_3,\alpha_1,\alpha], [\alpha_3,\alpha_1,\alpha_1]\rangle.$$ 
Using (\ref{x}) and (\ref{y}), we have $[\alpha_4,\alpha_1]=[\alpha_2,\alpha_3]=[\alpha_3,\alpha_1]$ in $H_2$. 
Observe that 
\begin{eqnarray*}
H_2 &\cong& \langle \alpha,\alpha_i,\alpha_4,\beta_2, \beta_3 \mid [\alpha_i,\alpha]=\alpha_{i+1}, [\alpha_4,\alpha]=\beta_2,[\alpha_1,\alpha_4]=[\alpha_1,\alpha_3] =[\alpha_3,\alpha_2]=\beta_3,\\
&& \hspace{5 cm} \alpha^p=\alpha_i^{p} = \alpha_4^p=\beta_2^p=\beta_3^p=1, \;1 \leq i \leq 3  \rangle.
\end{eqnarray*}
Since $H_2/\gen{\beta_3} \cong H_1$, it follows that $H_2$ is of order $p^7$.
Now consider 
$$H=F/\langle\gamma_6(F), F^p, [\beta_1,\alpha], [\beta_1,\alpha_1],[\alpha_3,\alpha_1,\alpha], [\alpha_3,\alpha_1,\alpha_1]\rangle.$$
Observe that 
\begin{eqnarray*}
H &\cong& \langle \alpha,\alpha_j,\beta_i \mid [\alpha_i,\alpha]=\alpha_{i+1},[\alpha_1,\alpha_2]=\beta_1,[\alpha_4,\alpha]=\beta_2,[\alpha_1,\alpha_4]=[\alpha_1,\alpha_3] =[\alpha_3,\alpha_2]=\beta_3,\\
&& \hspace{6.5 cm} \alpha^p=\alpha_j^{p}=\beta_i^p=1, \;1 \leq i \leq 3, \;1 \leq j\leq 4  \rangle.
\end{eqnarray*}
Then $H$ is a group of order $p^8$ and
 $$H/\langle \beta_1,\beta_2,\beta_3\rangle \cong G.$$ 
 
 Take $Z=\langle \beta_1,\beta_2,\beta_3\rangle$. Now the image of the transgression homomorphism $tra: \Hom(Z,\mathbb{C}^*) \rightarrow \M(H/Z)$ is $$H' \cap Z \cong Z \cong \mathbb{Z}_p^{(3)}.$$ Hence $\mathbb{Z}_p^{(3)}$ is contained in $\M(G)$. By Theorem \ref{J}, $|\M(G)| \leq p^3$. 
Hence $$\M(G) \cong \mathbb{Z}_p^{(3)}.$$
As $|G'|=p^3$, we get $$|G \wedge G|=p^6.$$
Since $\alpha_4 \in \Z(G)$,   by Lemma \ref{R}(viii)
$$[\alpha_2,\alpha_4^\phi]=[\alpha_3,\alpha_4^\phi]=1.$$ 
Thus $G \wedge G$ is generated by the set
$$\{[\alpha_1,\alpha^\phi], [\alpha_2,\alpha^\phi], [\alpha_3,\alpha^\phi],[\alpha_4,\alpha^\phi],[\alpha_2,\alpha_1^\phi],[\alpha_3,\alpha_1^\phi], [\alpha_4,\alpha_1^\phi],[\alpha_3,\alpha_2^\phi]\}.$$ 
Now we have
\begin{eqnarray*}
[\alpha^{-1},\alpha_1^{-1},\alpha_2^\phi]^{\alpha_1} &=& [\alpha_1\alpha \alpha_2^{-1} \alpha^{-1}\alpha_1^{-1},\alpha_2^\phi]^{\alpha_1}=[\alpha \alpha_2^{-1}\alpha^{-1},\alpha_2^{\phi}]\\
&=& [\alpha\alpha_3\alpha^{-1} \alpha_2^{-1},\alpha_2^{\phi}]=[\alpha_3\alpha_4^{-1}\alpha_2^{-1},\alpha_2^{\phi}]\\
&=& [\alpha_3,\alpha_2^\phi][\alpha_4,\alpha_2^\phi]^{-1}=[\alpha_3,\alpha_2^\phi],
\end{eqnarray*}
\begin{eqnarray*}
[\alpha^{-1},\alpha_1^{-1},\alpha_3^\phi]^{\alpha_1} &=& [\alpha^{-1},\alpha_1^{-1},\alpha_3^\phi]^{\alpha_1} = [\alpha_1\alpha \alpha_2^{-1} \alpha^{-1}\alpha_1^{-1},\alpha_3^\phi]^{\alpha_1}\\
&=& [\alpha \alpha_2^{-1}\alpha^{-1},\alpha_3^{\phi}]=[\alpha\alpha_3\alpha^{-1} \alpha_2^{-1},\alpha_3^{\phi}]\\
&=& [\alpha_3\alpha_4^{-1}\alpha_2^{-1},\alpha_3^{\phi}]=[\alpha_2,\alpha_3^\phi]^{-1}.
\end{eqnarray*}
By Hall-Witt identity, we have
\begin{eqnarray*}
1&=&[\alpha_2,\alpha,\alpha_1^\phi]^{\alpha^{-1}}[\alpha^{-1},\alpha_1^{-1},\alpha_2^\phi]^{\alpha_1}[\alpha_1,\alpha_2^{-1},{(\alpha^{-1})}^\phi]^{\alpha_2}\\
&=&[\alpha_3,\alpha_1^\phi][\alpha_3,\alpha_2^\phi],\\
1&=&[\alpha_3,\alpha,\alpha_1^\phi]^{\alpha^{-1}}[\alpha^{-1},\alpha_1^{-1},\alpha_3^\phi]^{\alpha_1}[\alpha_1,\alpha_3^{-1},{(\alpha^{-1})}^\phi]^{\alpha_3}\\
&=&[\alpha_4,\alpha_1^\phi][\alpha_2,\alpha_3^\phi]^{-1}.
\end{eqnarray*}
This implies that $[\alpha_4,\alpha_1^\phi]= [\alpha_2,\alpha_3^\phi]=[\alpha_3,\alpha_1^\phi]$. 
By Lemma \ref{R}(vi), $[[\alpha_1,\alpha^\phi], [\alpha_2,\alpha^\phi]]=[\alpha_2,\alpha_3^\phi]$.
By Lemma \ref{O}, we have
\begin{eqnarray*}
&&[\alpha_4,\alpha^\phi]^p=[\alpha_4^p,\alpha^\phi]=1=[\alpha_3^p,\alpha^\phi]=[\alpha_3,\alpha^\phi]^p,\\
&&[\alpha_3,\alpha_2^\phi]^p=[\alpha_3^p,\alpha_2^\phi]=1=[\alpha_2^p,\alpha_1^\phi]=[\alpha_2,\alpha_1^\phi]^p,\\
&&[\alpha_2,\alpha^\phi]^p=[\alpha_2^p,\alpha^\phi]=1=[\alpha_1^p,\alpha^\phi]=[\alpha_1,\alpha^\phi]^p.
\end{eqnarray*}
Hence $$G \wedge G \cong \langle [\alpha_1,\alpha^\phi], [\alpha_2,\alpha^\phi], [\alpha_2,\alpha_3^\phi]  \rangle \times \langle [\alpha_2,\alpha_1^\phi] \rangle \times  \langle [\alpha_3,\alpha^\phi] \rangle \times \langle [\alpha_4,\alpha^\phi] \rangle \cong \Phi_2(111)\times \mathbb{Z}_p^{(3)}.$$

\par 

Now consider the group $G \cong \Phi_{10}(1^5)$.  Let $F$ be the free group generated by $\{\alpha,\alpha_1\}$. 
Define $\alpha_{i+1}=[\alpha_i,\alpha], \; i=1,2,3$. Set $\beta=[\alpha_1,\alpha_2]$,
$\beta_1=[\alpha_4,\alpha]$, $\beta_2=[\alpha_1,\alpha_4]$ and $\beta_3=\alpha_4^{-1}[\alpha_1,\alpha_2]$.
Reading Modulo $\gamma_6(F)$, we have
\begin{eqnarray*}
[\alpha^{-1},\alpha_1^{-1},\alpha_3]^{\alpha_1} &=& [\alpha^{-1},\alpha_1^{-1},\alpha_3]^{\alpha_1} =[\alpha_1\alpha \alpha_2^{-1} \alpha^{-1}\alpha_1^{-1},\alpha_3]^{\alpha_1}\\
&=& [\alpha \alpha_2^{-1}\alpha^{-1},\alpha_3[\alpha_3,\alpha_1]]=[\alpha\alpha_3\alpha^{-1} \alpha_2^{-1},\alpha_3[\alpha_3,\alpha_1]]\\
&=& [\alpha_3\alpha\alpha_4^{-1}\alpha^{-1}\alpha_2^{-1},\alpha_3[\alpha_3,\alpha_1]]=[\alpha_3\alpha\beta_1\alpha^{-1}\alpha_4^{-1}\alpha_2^{-1},\alpha_3[\alpha_3,\alpha_1]]\\
&=& [\alpha_3\beta_1\alpha_4^{-1}\alpha_2^{-1},\alpha_3[\alpha_3,\alpha_1]]=[\alpha_2^{-1},\alpha_3]\\
&=& [\alpha_2,\alpha_3]^{-1},
\end{eqnarray*}
\begin{eqnarray*}
[\alpha^{-1},\alpha_1^{-1},\alpha_2]^{\alpha_1}&=& [\alpha^{-1},\alpha_1^{-1},\alpha_2]^{\alpha_1}=[\alpha_1\alpha \alpha_2^{-1} \alpha^{-1}\alpha_1^{-1},\alpha_2]^{\alpha_1}\hspace{2.8cm}\\
&=& [\alpha_3\beta_1\alpha_4^{-1}\alpha_2^{-1},\alpha_2\beta_3^{-1}\alpha_4^{-1}]=[\alpha_3,\alpha_2\beta_3^{-1}\alpha_4^{-1}][\alpha_2^{-1},\alpha_2\beta_3^{-1}\alpha_4^{-1}]\\
&=&[\alpha_3,\alpha_2][\alpha_2,\beta_3].
\end{eqnarray*}
By Hall-Witt identity, we have the following identities in $F$  modulo $\gamma_6(F)$: 
\begin{eqnarray}
1&=&[\alpha_3,\alpha,\alpha_1]^{\alpha^{-1}}[\alpha^{-1},\alpha_1^{-1},\alpha_3]^{\alpha_1}[\alpha_1,\alpha_3^{-1},\alpha^{-1}]^{\alpha_3}\hspace{.8cm}\nonumber\\
&=&[\alpha_4,\alpha_1]^{\alpha^{-1}}[\alpha_2,\alpha_3]^{-1}[\alpha_3,\alpha_1,\alpha_3^{-1}\alpha^{-1}\alpha_3]\nonumber\\
&=&[\alpha_4,\alpha_1][\alpha_2,\alpha_3]^{-1}[\alpha_3,\alpha_1, \alpha_4\alpha^{-1}]\nonumber\\
&=&[\alpha_4,\alpha_1][\alpha_2,\alpha_3]^{-1}[\alpha_3,\alpha_1,\alpha^{-1}], \label{x1}
\end{eqnarray}
\begin{eqnarray}
1&=&[\alpha_2,\alpha,\alpha_1]^{\alpha^{-1}}[\alpha^{-1},\alpha_1^{-1},\alpha_2]^{\alpha_1}[\alpha_1,\alpha_2^{-1},\alpha^{-1}]^{\alpha_2} \nonumber\\
&=&[\alpha_3,\alpha_1]^{\alpha^{-1}}[\alpha_3,\alpha_2][\alpha_2,\beta_3][\beta_3^{-1}\alpha_4^{-1},\alpha_2^{-1}\alpha^{-1}\alpha_2] \nonumber\\
&=&[\alpha_3,\alpha_1][\alpha_3,\alpha_1, \alpha^{-1}][\alpha_3,\alpha_2][\alpha_2,\beta_3][\beta_3^{-1}\alpha_4^{-1}, \alpha_3\alpha^{-1}] \nonumber\\
&=&[\alpha_3,\alpha_1][\alpha_3,\alpha_1, \alpha]^{-1}[\alpha_3,\alpha_2][\alpha_2,\beta_3][\beta_3,\alpha][\alpha_4,\alpha]. \label{y1}
\end{eqnarray} 

Consider 
$$H_1= F/\langle\gamma_6(F), F^p, \beta_1,\beta_3,[\alpha_3,\alpha_1,\alpha], [\alpha_3,\alpha_1,\alpha_1]\rangle.$$
By (\ref{x1}) and (\ref{y1}), respectively,  we have $[\alpha_4,\alpha_1]=[\alpha_2,\alpha_3]$ and  
$[\alpha_1,\alpha_3]=[\alpha_4,\alpha][\alpha_3,\alpha_2]$ in $H_1$. Observe that
\begin{eqnarray*}
H_1 &\cong & \langle \alpha,\alpha_1,\alpha_2,\alpha_3,\alpha_4, \beta_2 \mid [\alpha_i,\alpha]=\alpha_{i+1}, [\alpha_1,\alpha_2]=\alpha_4, [\alpha_1,\alpha_4]=[\alpha_3,\alpha_2]=[\alpha_1,\alpha_3]=\beta_2,\\
&&\hspace{6 cm}\alpha^p=\alpha_1^p=\alpha_{i+1}^p=\beta_2^p=1 \;(i=1,2,3)\rangle.
\end{eqnarray*}
Then $H_1$ is the group $\Phi_{39}(1^6)$ of order $p^6$ in \cite{RJ}. Now define
$$H_2= F/\langle\gamma_6(F), F^p, \beta_3,[\alpha_3,\alpha_1,\alpha], [\alpha_3,\alpha_1,\alpha_1]\rangle.$$
Again, by (\ref{x1}) and (\ref{y1}), we have $[\alpha_4,\alpha_1]=[\alpha_2,\alpha_3]$ and  
$[\alpha_1,\alpha_3]=[\alpha_4,\alpha][\alpha_3,\alpha_2]$ in $H_2$; hence
\begin{eqnarray*}
H_2 &\cong & \langle \alpha,\alpha_i, \alpha_4, \beta_1, \beta_2 \mid [\alpha_i,\alpha]=\alpha_{i+1},[\alpha_1,\alpha_2]=\alpha_4,[\alpha_4,\alpha]=\beta_1,[\alpha_1,\alpha_4]=[\alpha_3,\alpha_2]=\beta_2,\\
&&\hspace{4 cm}[\alpha_1,\alpha_3]=\beta_1\beta_2,\alpha^p=\alpha_{i+1}^{p}=\beta_1^p=\beta_2^p=1,~1 \leq i \leq 3 \rangle.
\end{eqnarray*}
Notice that $|H_2| = p^7$.

Finally consider
$$H=F/\langle\gamma_6(F),F^p, [\alpha_3,\alpha_1,\alpha], [\alpha_3,\alpha_1,\alpha_1],[\beta_3,\alpha],[\beta_3,\alpha_1]\rangle.$$ It follows that
\begin{eqnarray*}
H &\cong & \langle \alpha,\alpha_j,\beta_i \mid [\alpha_i,\alpha]=\alpha_{i+1},[\alpha_1,\alpha_2]=\alpha_4 \beta_3,[\alpha_4,\alpha]=\beta_1,[\alpha_1,\alpha_4]=[\alpha_3,\alpha_2]=\beta_2,\\
&&\hspace{3 cm}[\alpha_1,\alpha_3]=\beta_1\beta_2,\alpha^p=\alpha_{i+1}^{p}=\beta_i^p=1, ~1 \leq i \leq 3,~1 \leq j \leq 4 \rangle.
\end{eqnarray*} 
Notice that $|H|=p^8$ and $H/\langle \beta_1,\beta_2,\beta_3 \rangle \cong G$.
Then, as in the preceding case, we have 
$$\M(G) \cong \mathbb{Z}_p^{(3)}.$$
As $|G'|=p^3$, we get $$|G \wedge G|=p^6.$$
Since $\alpha_4 \in \Z(G)$, by Lemma \ref{R}(viii),
 $$[\alpha_2,\alpha_4^\phi]=[\alpha_1,\alpha,\alpha_4^\phi]=1=[\alpha_2,\alpha,\alpha_4^\phi]=[\alpha_3,\alpha_4^\phi]. $$ 
Hence $G \wedge G$ is generated by the set
$$\{[\alpha_1,\alpha^\phi], [\alpha_2,\alpha^\phi], [\alpha_3,\alpha^\phi],[\alpha_4,\alpha^\phi],[\alpha_2,\alpha_1^\phi],[\alpha_3,\alpha_1^\phi], [\alpha_4,\alpha_1^\phi],[\alpha_3,\alpha_2^\phi]\}.$$

We have
\begin{eqnarray*}
[\alpha^{-1},\alpha_1^{-1},\alpha_2^\phi]^{\alpha_1} &=& [\alpha_1\alpha \alpha_2^{-1} \alpha^{-1}\alpha_1^{-1},\alpha_2^\phi]^{\alpha_1}=[\alpha \alpha_2^{-1}\alpha^{-1},(\alpha_1^{-1}\alpha_2\alpha_1)^{\phi}]\\
&=& [\alpha\alpha_3\alpha^{-1} \alpha_2^{-1},(\alpha_2\alpha_4)^{\phi}]=[\alpha_3\alpha_4^{-1}\alpha_2^{-1},(\alpha_2\alpha_4)^{\phi}]\\
&=& [\alpha_4,\alpha_2^\phi]^{-1}[\alpha_3,\alpha_2^\phi]=[\alpha_3,\alpha_2^\phi],
\end{eqnarray*}
\begin{eqnarray*}
[\alpha^{-1},\alpha_1^{-1},\alpha_3^\phi]^{\alpha_1} &=& [\alpha^{-1},\alpha_1^{-1},\alpha_3^\phi]^{\alpha_1} = [\alpha_1\alpha \alpha_2^{-1} \alpha^{-1}\alpha_1^{-1},\alpha_3^\phi]^{\alpha_1}\hspace{1.2cm}\\
&=& [\alpha \alpha_2^{-1}\alpha^{-1},\alpha_3^{\phi}]=[\alpha\alpha_3\alpha^{-1} \alpha_2^{-1},\alpha_3^{\phi}]\\
&=& [\alpha_3\alpha_4\alpha_2^{-1},\alpha_3^{\phi}]=[\alpha_2,\alpha_3^\phi]^{-1}.
\end{eqnarray*}
By Hall-Witt identity, we have
\begin{eqnarray*}
1&=&[\alpha_2,\alpha,\alpha_1^\phi]^{\alpha^{-1}}[\alpha^{-1},\alpha_1^{-1},\alpha_2^\phi]^{\alpha_1}[\alpha_1,\alpha_2^{-1},{(\alpha^{-1})}^\phi]^{\alpha_2}\\
&=&[\alpha_3,\alpha_1^\phi][\alpha_3,\alpha_2^\phi][\alpha_4^{-1},(\alpha^{-1})^\phi]^{\alpha_2}\\
&=&[\alpha_3,\alpha_1^\phi][\alpha_3,\alpha_2^\phi][\alpha_4,\alpha^\phi],\\
1&=&[\alpha_3,\alpha,\alpha_1^\phi]^{\alpha^{-1}}[\alpha^{-1},\alpha_1^{-1},\alpha_3^\phi]^{\alpha_1}[\alpha_1,\alpha_3^{-1},{(\alpha^{-1})}^\phi]^{\alpha_3}\\
&=&[\alpha_4,\alpha_1^\phi][\alpha_2,\alpha_3^\phi]^{-1}.
\end{eqnarray*}
This implies that $[\alpha_4,\alpha_1^\phi]= [\alpha_2,\alpha_3^\phi]=[\alpha_4,\alpha^\phi][\alpha_3,\alpha_1^\phi]$.
By Lemma \ref{R}(vi), $[[\alpha_1,\alpha^\phi], [\alpha_2,\alpha^\phi]]=[\alpha_2,\alpha_3^\phi]$. 
We have the following identities
\begin{eqnarray*}
&&[\alpha_4,\alpha^\phi]^p=[\alpha_4^p,\alpha^\phi]=1=[\alpha_3^p,\alpha^\phi]=[\alpha_3,\alpha^\phi]^p,\\
&&[\alpha_3,\alpha_2^\phi]^p=[\alpha_3^p,\alpha_2^\phi]=1=[\alpha_2^p,\alpha_1^\phi]=[\alpha_2,\alpha_1^\phi]^p,\\
&&[\alpha_2,\alpha^\phi]^p=[\alpha_2^p,\alpha^\phi]=1=[\alpha_1^p,\alpha^\phi]=[\alpha_1,\alpha^\phi]^p.
\end{eqnarray*}
Hence 
$$G \wedge G=\langle [\alpha_1,\alpha^\phi], [\alpha_2,\alpha^\phi], [\alpha_2,\alpha_3^\phi]  \rangle \times \langle [\alpha_2,\alpha_1^\phi] \rangle \times  \langle [\alpha_3,\alpha^\phi] \rangle \times \langle [\alpha_4,\alpha^\phi] \rangle \cong \Phi_2(111)\times \mathbb{Z}_p^{(3)}.$$
The proof is now complete.
\hfill $\Box$

\end{proof}


\section{Remaining Groups of order $p^5$}
We start with the groups which occur as direct product of groups of smaller orders. For such groups we compute the Schur multiplier  in the following lemma, whose proof follows  using  Theorem \ref{D} and Theorem \ref{SHHH}.

\begin{lemma}
The following assertions hold:
\begin{enumerate}[label=(\roman*)]
\item $\M(\Phi_2(311)a) \cong \M(\Phi_2(221)b)  \cong  \mathbb{Z}_p \times \mathbb{Z}_p$, 

\item $\M(\Phi_2(221)a)  \cong  \mathbb{Z}_p^{(3)}$, 

\item $\M(\Phi_2(2111)a) \cong \M(\Phi_2(2111)b) \cong  \mathbb{Z}_p^{(5)}$,  

\item $\M(\Phi_2(2111)c) \cong \M(\Phi_2(2111)d) \cong  \mathbb{Z}_p^{(4)}$,

\item $\M(\Phi_2(1^5)) \cong  \mathbb{Z}_p^{(7)}$,

\item $\M(\Phi_3(2111)a) \cong \M(\Phi_3(2111)b_r) \cong  \mathbb{Z}_p^{(3)}$, 

\item $\M(\Phi_3(1^5)) \cong  \mathbb{Z}_p^{(4)}$.
\end{enumerate}
\end{lemma}

The proof of the following lemma is a direct consequence of Proposition \ref{D1} and Theorem \ref{SHHH}. 

\begin{lemma}
The following assertions hold:
\begin{enumerate}[label=(\roman*)]
\item $\Phi_2(311)a \otimes \Phi_2(311)a  \cong \Phi_2(221)b \otimes \Phi_2(221)b  \cong  \mathbb{Z}_{p^2} \times \mathbb{Z}_p^{(8)}$, 

\item $\Phi_2(221)a \otimes \Phi_2(221)a \cong  \mathbb{Z}_{p^2}^{(2)} \times \mathbb{Z}_p^{(7)}$, 

\item $\Phi_2(2111)a \otimes \Phi_2(2111)a \cong \Phi_2(2111)b \otimes \Phi_2(2111)b \cong  \mathbb{Z}_p^{(16)}$,  

\item $\Phi_2(2111)c \otimes \Phi_2(2111)c \cong \Phi_2(2111)d \otimes \Phi_2(2111)d \cong   \mathbb{Z}_{p^2} \times \mathbb{Z}_p^{(10)}$,

\item $\Phi_2(1^5) \otimes \Phi_2(1^5) \cong  \mathbb{Z}_p^{(18)}$,

\item $\Phi_3(2111)a \otimes \Phi_3(2111)a \cong \Phi_3(2111)b_r \otimes \Phi_3(2111)b_r \cong  \mathbb{Z}_p^{(11)}$, 

\item$\Phi_3(1^5) \otimes \Phi_3(1^5)\cong  \mathbb{Z}_p^{(12)}$.
\end{enumerate}

\end{lemma}
\begin{lemma}
If $G$ is one of the groups   $\Phi_2(41),  \Phi_2(32)a_1$, $ \Phi_2(32)a_2$ or $\Phi_8(32)$, then $\M(G)$ is isomorphic to $\{1\}, \mathbb{Z}_p$, $\mathbb{Z}_p$ or $\{1\}$ respectively,  and 
$G \wedge G$ is isomorphic to $\mathbb{Z}_{p}, \mathbb{Z}_{p^2}, \mathbb{Z}_{p^2}$ or $\mathbb{Z}_{p^2}$ respectively.
\end{lemma}
\begin{proof}
Since these groups are metacyclic, the assertion about the Schur multipliers  follows from \cite[Theorem 2.11.3]{GK}.  
Now we compute the exterior square. 
Since the Schur multiplier of the groups $G$ isomorphic to $\Phi_2\left(41\right)$ or $\Phi_8\left(32\right)$ is trivial, we have $G \wedge G \cong G'$.
Hence 
$$\Phi_2\left(41\right) \wedge \Phi_2\left(41\right) \cong \mathbb{Z}_p$$
and
$$\Phi_8\left(32\right) \wedge  \Phi_8\left(32\right) \cong \mathbb{Z}_{p^2}.$$ 

Now we consider $G = \Phi_2\left(32\right)a_1$.  We have the following identities
\begin{eqnarray*}
&&[\alpha_2,\alpha^\phi]=[\alpha^{p^2},\alpha^\phi]=[\alpha,\alpha^\phi]^{p^2} = 1=[\alpha_2^p,\alpha_1^\phi]=[\alpha_2,\alpha_1^\phi]^p,\\
&& [\alpha_2,\alpha_1^\phi]=[\alpha^{p^2},\alpha_1^\phi]=[\alpha,\alpha_1^\phi]^{p^2}=[\alpha,{(\alpha_1^{p^2})}^\phi]=1.
\end{eqnarray*}
Hence, by Proposition \ref{B}, $G \wedge G$ is a cyclic group generated by $[\alpha_1,\alpha^\phi]$. Since both $\M(G)$ and $\gamma_2(G)$ are of order $p$, it follows that  
$$G \wedge G \cong\langle [\alpha_1,\alpha^\phi] \rangle\cong \mathbb{Z}_{p^2}.$$

Similarly for $G=\Phi_2\left(32\right)a_2$,  we have
\begin{eqnarray*}
&&[\alpha_2,\alpha_1^\phi]=[\alpha_1^p,\alpha_1^\phi]=[\alpha_1,\alpha_1^\phi]^p =1,\\
&& [\alpha_2,\alpha^\phi]=[\alpha_1^p,\alpha^\phi]=[\alpha_1,\alpha^\phi]^p.
\end{eqnarray*}
Hence $G \wedge G$ is generated by $[\alpha_1,\alpha^\phi]$. As above, 
$$G \wedge G \cong  \langle [\alpha_1,\alpha^\phi] \rangle\cong\mathbb{Z}_{p^2},$$
and the proof is complete.
\hfill $\Box$

\end{proof}

\begin{lemma}
If $G$ is one of the groups $\Phi_2(311)b$ or $ \Phi_2(311)c$, then $\M(G)$ is isomorphic to $ \mathbb{Z}_p \times \mathbb{Z}_p$,  and $G \wedge G$ is isomorphic to $\mathbb{Z}_p^{(3)}$.
\end{lemma}

\begin{proof} 
For the  group $G=\Phi_2(311)b$,  taking $K=\Z(G)$ in  Theorem \ref{J}, we have $|\M(G)| \leq p^2$. On the other hand,   taking $K=\Phi(G)$ in Theorem \ref{2}(ii), we get $d(\M(G)) \geq 2$. 
Hence, 
$$\M(G) \cong \mathbb{Z}_p \times \mathbb{Z}_p.$$
The group $G \wedge G$ is generated by the set$$\{[\alpha_1,\alpha^\phi],[\alpha_2,\alpha^\phi],[\alpha_2,\alpha_1^\phi],[\gamma, \alpha^\phi], [\gamma,\alpha_1^\phi],[\gamma,\alpha_2^\phi]\}.$$
By Lemma \ref{R}(viii),
$$[\alpha_2,\gamma^\phi]=[\alpha_1,\alpha,\gamma^\phi]=1.$$
We get the following set of identities: 
\begin{eqnarray*} 
&&[\gamma,\alpha^\phi]^p=[\gamma,{(\alpha^p)}^\phi]=1=[\gamma,{(\alpha_1^p)}^\phi]=[\gamma,\alpha_1^\phi]^p, \\
&&[\alpha_2,\alpha^\phi]=[\gamma^{p^2},\alpha^\phi]=[\gamma,\alpha^\phi]^{p^2}=1, \\
&&[\alpha_2,\alpha_1^\phi]=[\gamma^{p^2},\alpha_1^\phi]=[\gamma,\alpha_1^\phi]^{p^2}=1,\\
&&[\alpha,\alpha_1^\phi]^p=[\alpha^p,\alpha_1^\phi]=1.
\end{eqnarray*}
Hence $G \wedge G$ is generated by $\{[\alpha_1,\alpha^\phi],[\gamma, \alpha^\phi], [\gamma,\alpha_1^\phi]\}$. Since the nilpotency class of $G$ is $2$, by Lemma \ref{S}, $G \wedge G$ is abelian.
Hence,
  $$G \wedge G \cong \mathbb{Z}_p^{(3)}.$$

Now consider the group $G=\Phi_2(311)c$. Again using Theorem \ref{J} with $K=G'$, we get $|\M(G)| \leq p^2$.  By Theorem \ref{2}(ii) with $K=\langle \alpha^p \rangle$, we have $d(\M(G)) \geq 2$. Hence 
$$\M(G) \cong \mathbb{Z}_p \times \mathbb{Z}_p.$$ Since $|G'| = p$, it follows that $G \wedge G$ is of order $p^3$.

Notice that
\begin{align*}
[\alpha_2, \alpha_1^\phi]^p &= [\alpha_2^p, \alpha_1^\phi]=1, \\
[\alpha_2, \alpha^\phi]^p & =[\alpha_2^p, \alpha^\phi] =1, \\
[\alpha_1, \alpha^\phi]^p  &=[\alpha_1^p, \alpha^\phi] =1.
\end{align*}
Consequently, $G \wedge G$ is generated by $\{[\alpha_1,\alpha^\phi],[\alpha_2, \alpha^\phi], [\alpha_2,\alpha_1^\phi]\}$, which shows that
$$G \wedge G \cong \mathbb{Z}_p^{(3)}.$$
The proof is now complete.
\hfill$\Box$

\end{proof}

\begin{lemma}
If $G$ is one of the groups $\Phi_3(311)a, \Phi_3(311)b_r$ or $ \Phi_2(221)d$, then $\M(G)$ is isomorphic to $\mathbb{Z}_p , \mathbb{Z}_p$ or $\mathbb{Z}_p^{(3)}$ respectively, and $G \wedge G$ is isomorphic to $\mathbb{Z}_p^{(3)}, \mathbb{Z}_p^{(3)}$ or $\mathbb{Z}_{p^2} \times \mathbb{Z}_p \times \mathbb{Z}_p$ respectively.
\end{lemma}
\begin{proof}
Let  $G$ be one of the groups $\Phi_3(311)a$ or $\Phi_3(311)b_r$. Then taking $K=\Z(G)$
in Theorem \ref{2}(i),  $p$ divides $|\M(G)|$. Since $|G'| = p^2$, it follows that $|G \wedge G| \geq p^3$.

For the group $G =\Phi_3\left(311\right)a$, by Lemma \ref{R}(viii), $[\alpha_2,\alpha_3^\phi]=[\alpha_1,\alpha,\alpha_3^\phi]=1$.  
It now follows by Proposition \ref{B} that  $G \wedge G$ is generated by the set 
$$\{[\alpha_1,\alpha^\phi],[\alpha_2,\alpha^\phi],[\alpha_2,\alpha_1^\phi],[\alpha_3, \alpha^\phi], [\alpha_3,\alpha_1^\phi]\}.$$
 The following identities hold:
\begin{eqnarray*}
&& [\alpha_3,\alpha^\phi]=[\alpha^{p^2},\alpha^\phi]=[\alpha,\alpha^\phi]^{p^2}=1, \\
&&[\alpha_3,\alpha_1^\phi]^p=[\alpha_3^p,\alpha_1^\phi]=1,\\
&& [\alpha_2,\alpha^\phi]^p= [\alpha_2^p,\alpha^\phi]=1= [\alpha_2^p,\alpha_1^\phi]=[\alpha_2,\alpha_1^\phi]^p, \\
&& [\alpha_1,\alpha^\phi]^p = [\alpha_1^p,\alpha^\phi]=1, \\
&&[\alpha_3,\alpha_1^\phi]=[\alpha^{p^2},\alpha_1^\phi]=[\alpha,\alpha_1^\phi]^{p^2}=1.
\end{eqnarray*}
Hence $G \wedge G$ is generated by $\{[\alpha_1,\alpha^\phi], [\alpha_2,\alpha^{\phi}], [\alpha_2,\alpha_1^{\phi}]\}$. Since, in view of Lemma \ref{R}(vi), $G \wedge G$ is abelian,  we have
$$G \wedge G \cong \mathbb{Z}_p^{(3)}.$$ 
Consequently,
$$\M(G) \cong [\alpha_1,\alpha_2^{\phi}] \rangle \cong \mathbb{Z}_p.$$

The computation for the group $G=\Phi_3(311)b_r$ goes on the same lines for $r = 1, \nu$.

Finally we consider the group $G=\Phi_2(221)d$. The group
$G \wedge G$ is generated by the set 
$$\{[\alpha_1,\alpha^{\phi}], [\alpha_2,\alpha^{\phi}] , [\alpha_1,\alpha_2^{\phi}]\}.$$
By Lemma \ref{S}, $G \wedge G$ is abelian. We have the identities:
\begin{eqnarray*}
&& [\alpha_2,\alpha_1^\phi]^p=[\alpha_2^p,\alpha_1^\phi]=1= [\alpha_2^p,\alpha^\phi]=[\alpha_2,\alpha^\phi]^p,\\
&&[\alpha_1,\alpha^\phi]^{p^2}=[\alpha_1^{p^2},\alpha^\phi]=1.
\end{eqnarray*}

Now the natural epimorphism $[G,  G^\phi] \rightarrow [G^{ab},  (G^{ab})^\phi]$ implies that  the order of $[\alpha_1,\alpha^\phi]$ is $p^2$ in $G \wedge G$.
Notice that $G/\gen{\alpha_1^p} \cong \Phi_2(211)c$. Consider the natural epimorphism 
$$[G, G^\phi] \rightarrow [G/\gen{\alpha_1^p}, 
(G/\gen{\alpha_1^p})^\phi] \cong [\Phi_2(211)c,   (\Phi_2(211)c)^\phi],$$ which induces an epimorphism $G \wedge G \to (\Phi_2(211)c) \wedge (\Phi_2(211)c)$. By Theorem \ref{SHHH}, we know that $ \Phi_2(211)c \wedge \Phi_2(211)c \cong   \mathbb{Z}_p ^{(3)}$. 
Hence the generators $[\alpha_1,\alpha^{\phi}], [\alpha_2,\alpha^{\phi}] , [\alpha_1,\alpha_2^{\phi}]$ of $G \wedge G$ are non-trivial, independent and $[\alpha_2,\alpha^{\phi}],[\alpha_1,\alpha_2^{\phi}]$ have order $p$ in $G \wedge G$. 
As a consequence, we get
$$G \wedge G \cong \mathbb{Z}_{p^2} \times \mathbb{Z}_p \times \mathbb{Z}_p,$$
 which, by observing the fact that $G' = \gen{[\alpha_1,\alpha]}$, gives
$$\M(G) \cong \mathbb{Z}_p \times \mathbb{Z}_p \times \mathbb{Z}_p.$$
This completes  the proof.
\hfill$\Box$

\end{proof}

\begin{lemma}
If $G$ is one of the groups $\Phi_3(221)a$ or $  \Phi_2(221)c$, then $\M(G)$ is isomorphic to $\mathbb{Z}_p$ or $ \mathbb{Z}_{p^2} \times \mathbb{Z}_p$ respectively, and $G \wedge G$ is isomorphic to $\mathbb{Z}_p^{(3)}$ or $\mathbb{Z}_{p^2} \times \mathbb{Z}_p \times \mathbb{Z}_p$ respectively.
\end{lemma}
\begin{proof}
The group $G =\Phi_3(221)a $ is a semidirect product of its  subgroups
$$N=\gen{\alpha_2,\alpha, \alpha_3} \cong \Phi_2(21)$$ and 
$$T=\gen{\alpha_1} \cong \mathbb{Z}_{p^2},$$
 where $N$ is normal.

By Theorem \ref{SHHH}, we know that $\M(N)={1}$.  Now using Theorem \ref{1}, we get the following exact sequence
$$1 \rightarrow \Ho^1(T, \Hom(N,\mathbb{C}^*)) \rightarrow \Ho^2(G,\mathbb{C}^*) \rightarrow 1.$$
Hence, 
$$\M(G) \cong \Ho^1(T, \Hom(N,\mathbb{C}^*)).$$ 
Notice that  $\Hom(N,\mathbb{C}^*) \cong N/N' \cong \gen{\alpha N', \alpha_2 N'}$.
Let $\zeta$ be a primitive $p$-th root of unity and $\Hom(N,\mathbb{C}^*) \cong \langle \phi_1,\phi_2 \rangle$, where $\phi_i:N \rightarrow \mathbb{C}^*$, $i = 1, 2$, are defined by setting
$$\phi_1(\alpha)=\zeta, \;\;\phi_1(\alpha_2)=1$$ 
and 
$$\phi_2(\alpha)=1, \;\;\phi_2(\alpha_2)=\zeta^{-1}.$$ 
Recall that  $T$  acts on $\Hom(N,\mathbb{C}^*)$ as follows.
For $\phi_1 \in \Hom(N, \mathbb{C}^*$), we set 
$$^{\alpha_1}{\phi_1}(\alpha)=\phi_1(\alpha_1^{-1} \alpha \alpha_1)=\phi_1(\alpha)$$ 
and 
$$^{\alpha_1}{\phi_1}(\alpha_2)=\phi_1(\alpha_1^{-1} \alpha_2 \alpha_1)=\phi_1(\alpha_2).$$ 

So, $^{\alpha_1}{\phi_1}=\phi_1$. Similarly the action of $\alpha_1$ on $\phi_2$ is given by $^{\alpha_1}{\phi_2}=\phi_1 \phi_2$. 

Define the map $Norm:\Hom(N,\mathbb{C}^*)\rightarrow \Hom(N,\mathbb{C}^*)$  given 
by
$$Norm(\phi)= \;^{(1+ \alpha_1 + \alpha_1^2 +\cdots+\alpha_1^{p-1})}{\phi}.$$
It is easy to check that $\Ker(Norm)=\Hom(N, \mathbb{C}^*)$.
Define $\beta:\Hom(N,\mathbb{C}^*)\rightarrow \Hom(N,\mathbb{C}^*)$ by 
$$\beta(\phi)=\;^{(\alpha_1-1)}{\phi}$$
Then $\im(\beta)=\langle \phi_1 \rangle$. It is a general fact  (see Step 3 in the proof of Theorem 5.4 of \cite{H}) that
$$\Ho^1(T, \Hom(N,\mathbb{C}^*)) \cong \frac{\Ker(Norm)}{\im (\beta)} \cong \mathbb{Z}_p.$$
Hence, $$\M(G) \cong \mathbb{Z}_p.$$
 
We have the following identities:
\begin{eqnarray*}
&&[\alpha_2,\alpha_3^\phi]=[\alpha_1,\alpha,\alpha_3^\phi]=1, ~\text{by Lemma \ref{R}(viii)}.\hspace{2cm} \\
&&[\alpha_3,\alpha^\phi]=[\alpha^p, \alpha^\phi]=[\alpha,\alpha^\phi]^p=1,\\
&&[\alpha_2,\alpha^\phi]^p=[\alpha_2^p,\alpha^\phi]=1=[\alpha_2^p,\alpha_1^\phi]=[\alpha_2,\alpha_1^\phi]^p,\\
&&[\alpha_3, \alpha_1^\phi]=[\alpha^p, \alpha_1^\phi]=[\alpha, \alpha_1^\phi]^p.
\end{eqnarray*}
Now we get
\begin{eqnarray*}
[\alpha^{-1},\alpha_1^{-1},\alpha_2^\phi]^{\alpha_1} &=& [\alpha_1\alpha\alpha_2^{-1}\alpha^{-1}\alpha_1^{-1},\alpha_2^{\phi}]^{\alpha_1}=[\alpha\alpha_2^{-1}\alpha^{-1},\alpha_2^{\phi}]\\
&=& [\alpha_3\alpha_2^{-1},\alpha_2^{\phi}]=[\alpha_3,\alpha_2^\phi][\alpha_2,\alpha_2^\phi]^{-1}\\
&=& 1.
\end{eqnarray*}
By Hall-Witt identity,   
\begin{eqnarray*}
1 &=& [\alpha_2,\alpha,\alpha_1^\phi]^{\alpha^{-1}}[\alpha^{-1},\alpha_1^{-1},\alpha_2^\phi]^{\alpha_1}\hspace{4cm}\\
&=&[\alpha_3,\alpha_1^\phi].
\end{eqnarray*}
Hence, $G \wedge G$ is generated by the set
$$\{[\alpha_1,\alpha^\phi],[\alpha_2,\alpha^\phi],[\alpha_2,\alpha_1^\phi]\}.$$ 
By Lemma \ref{R}(vi) $G \wedge G$ is abelian, and therefore 
$$G \wedge G \cong \mathbb{Z}_p^{(3)}.$$

The group $G=\Phi_2(221)c$ is a semidirect product of its normal subgroup $N=\langle \alpha_1, \gamma \rangle$ and $T=\langle \alpha \rangle$. 
Now using Theorem \ref{1}, we get the following exact sequence
$$1 \rightarrow \Ho^1(T, \Hom(N,\mathbb{C}^*)) \rightarrow \Ho^2(G,\mathbb{C}^*) \rightarrow \Ho^2(N,\mathbb{C}^*)^T.$$
 As above, $\Ho^1(T, \Hom(N,\mathbb{C}^*)) \cong \mathbb{Z}_{p^2}$, which embeds in $\M(G)$. Now by Theorem \ref{2}(ii), taking $K=\langle \gamma^p \rangle$, we have $2 \leq d(\M(G))$. By Theorem \ref{J}, taking $K=\langle \gamma \rangle$, we have $|\M(G)| \leq p ^3$. 
Hence 
$$\M(G) \cong \mathbb{Z}_{p^2} \times \mathbb{Z}_p.$$ 

We have the following identities:
\begin{eqnarray*}
&&[\alpha_2,\gamma^\phi]=[\alpha_1,\alpha,\gamma^\phi]=1, ~\text{by Lemma \ref{R}(viii)},\hspace{2cm}\\
&& [\alpha_2, \alpha_1^\phi]=[\gamma^p, \alpha_1^\phi]=[\gamma, \alpha_1^\phi]^p=[\gamma, {(\alpha_1^p)}^\phi]=1,\\
&& [\alpha_2, \alpha^\phi]=[\gamma^p, \alpha^\phi]=[\gamma, \alpha^\phi]^p,\\ 
&&[\alpha_1,\alpha^\phi]^p=[\alpha_1^p,\alpha^\phi]=1.
\end{eqnarray*}
Hence $G \wedge G$ is generated by the set
$$\{[\alpha_1,\alpha^\phi],[\gamma,\alpha^\phi],[\gamma,\alpha_1^\phi]\}.$$ Since the nilpotency class of $G$ is $2$, by Lemma \ref{S}, $G \wedge G$ is abelian; hence
$$G \wedge G \cong \langle[\gamma,\alpha^\phi] \rangle \times \langle [\gamma,\alpha_1^\phi]\rangle \times \langle [\alpha_1,\alpha^\phi]\rangle \cong \mathbb{Z}_{p^2} \times \mathbb{Z}_p \times \mathbb{Z}_p.$$ The proof is now complete.
\hfill$\Box$

\end{proof}

\begin{lemma}
If $G$ is one of the groups $\Phi_3(2111)c, \Phi_3(2111)d$ or $\Phi_3(2111)e$, then $\M(G)$ is isomorphic to $\mathbb{Z}_p^{(3)}$, $\mathbb{Z}_p \times \mathbb{Z}_p$ or $\mathbb{Z}_p \times \mathbb{Z}_p$ respectively, and $G \wedge G$ is isomorphic to $\mathbb{Z}_p^{(5)}, \mathbb{Z}_p^{(4)}$ or $\mathbb{Z}_p^{(4)}$ respectively.
\end{lemma}
\begin{proof}
For the group $G=\Phi_3(2111)c$, by Theorem \ref{J}, taking $K=\Z(G)$, we get $|\M(G)| \leq p^3$, and by Theorem \ref{2}(i), taking $K=\langle \gamma^p \rangle$, it follows that $p^3$ divides $|\M(G)|$; hence $|\M(G)|=p^3$.

We have the following identities:
\begin{eqnarray*}
&&[\alpha_2,\gamma^\phi]=[\alpha_1,\alpha, \gamma^\phi]=1=[\alpha_2,\alpha, \gamma^\phi]=[\alpha_3,\gamma^\phi], ~\text{by Lemma \ref{R}(viii)}.\\
&&[\alpha_3,\alpha^\phi]=[\gamma^p,\alpha^\phi]=[\gamma,\alpha^\phi]^{p}=[\gamma,(\alpha^p)^\phi]=1, \\
&&[\alpha_3,\alpha_i^\phi]=[\gamma^p,\alpha_i^\phi]=[\gamma,\alpha_i^\phi]^{p}=[\gamma,(\alpha_i^p)^\phi]=1~\text{for}~ i =1,2,\\
&& [\alpha_2,\alpha_1^\phi]^p=[\alpha_2^p,\alpha_1^\phi]=1=[\alpha_2^p,\alpha^\phi]=[\alpha_2,\alpha^\phi]^p, \\
&& [\alpha_1,\alpha^\phi]^p = [\alpha_1^p,\alpha^\phi]=1. 
\end{eqnarray*}
The group $G \wedge G$ is generated by 
$$\{[\alpha_1,\alpha^\phi],[\alpha_2,\alpha^\phi],[\alpha_2,\alpha_1^\phi],[\gamma, \alpha^\phi], [\gamma,\alpha_1^\phi]\},$$
and every generator has order at most $p$.
By Lemma \ref{R}(vi), $G \wedge G$ is abelian. Hence
$$\M(G) \cong \mathbb{Z}_p^{(3)}.$$ 
Since $|G'| = p^2$, we get $$G \wedge G \cong \mathbb{Z}_p^{(5)}.$$

Consider the group $G=\Phi_3(2111)d$.
By Theorem \ref{2}(i), taking $K=\gen{\alpha^p}$,  $p^2\leq \M(G)$. Since $|G'| = p^2$, we  have 
$$|G \wedge G|\geq p^4.$$

By Lemma \ref{R}(viii),
$$[\alpha_2,\alpha_3^\phi]=[\alpha_1,\alpha,\alpha_3^\phi]=1.$$
We have 
\begin{eqnarray*}
[\alpha^{-1},\alpha_1^{-1},\alpha_2^\phi]^{\alpha_1}&=&[\alpha_1\alpha\alpha_2^{-1}\alpha^{-1}\alpha_1^{-1},\alpha_2^{\phi}]^{\alpha_1}=[\alpha\alpha_2^{-1}\alpha^{-1},\alpha_2^{\phi}]\\
&=& [\alpha_3\alpha_2^{-1},\alpha_2^{\phi}]= [\alpha_3,\alpha_2^\phi][\alpha_2,\alpha_2^\phi]^{-1}=[\alpha_2,\alpha_2^\phi]^{-1}\\
&=& 1.
\end{eqnarray*}
By Hall-Witt identity,
\begin{eqnarray*}
1 &=& [\alpha_2,\alpha,\alpha_1^\phi]^{\alpha^{-1}}[\alpha^{-1},\alpha_1^{-1},\alpha_2^\phi]^{\alpha_1}\hspace{4cm}\\
&=&[\alpha_3,\alpha_1^\phi].
\end{eqnarray*}
We have the following identities:
\begin{eqnarray*}
&& [\alpha_3,\alpha^\phi]^p=[\alpha_3^p,\alpha^\phi]=1=[\alpha_2^p,\alpha_1^\phi]=[\alpha_2,\alpha_1^\phi]^p,\hspace{2cm}\\
&&[\alpha_2,\alpha^\phi]^p=[\alpha_2^p,\alpha^\phi]=1=[\alpha_1^p,\alpha^\phi]=[\alpha_1,\alpha^\phi]^p. 
\end{eqnarray*}
So $G \wedge G$ is generated by the set
$$\{[\alpha_1,\alpha^\phi],[\alpha_2,\alpha^\phi],[\alpha_2,\alpha_1^\phi], [\alpha_3,\alpha^{\phi}]\}.$$
By Lemma \ref{R}(vi), it follows that $G \wedge G$ is elementary abelian $p$-group, and hence 
 $|G \wedge G|\leq p^4$. Thus,
$$G \wedge G \cong \mathbb{Z}_p^{(4)}$$ and 
$$\M(G) \cong \mathbb{Z}_p \times \mathbb{Z}_p.$$

For the group $G=\Phi_3(2111)e$, by Theorem \ref{2}(i), taking $K=\langle \alpha_1^p\rangle$, we get $p^2\leq |\M(G)|$. Since $|G'| = p^2$, we  have 
$|G \wedge G|\geq p^4.$
As in the preceding case, observe that $G \wedge G$ is elementary abelian $p$-group of order $p^4$. 
Hence, 
$$\M(G) \cong \mathbb{Z}_p \times \mathbb{Z}_p$$ and 
$$G \wedge G \cong \mathbb{Z}_p^{(4)}.$$
The proof is now complete.
\hfill$\Box$

\end{proof}

\begin{lemma} 
If $G$ is one of the groups $\Phi_3(221)b_r$, $\Phi_6(221)d_0$ or $\Phi_6(221)b_{\frac{1}{2}(p-1)}$, then $\M(G)$ is isomorphic to  $ \mathbb{Z}_p \times \mathbb{Z}_p, \mathbb{Z}_p$ or $\mathbb{Z}_p$ respectively, and $G \wedge G$ is isomorphic to $\mathbb{Z}_{p^2} \times \mathbb{Z}_p \times \mathbb{Z}_p$ in all the cases.
\end{lemma}
\begin{proof}
From \cite[Table 4.1]{RJ}, it follows that
for the groups $E$ of order $p^6$ in the isoclinism classes $\Phi_{25}$, $\Phi_{26}$, $\Phi_{42}$ or $\Phi_{43}$,  $E/Z(E)$ is isomorphic to  $\Phi_3\left(221\right)b_1, \Phi_3\left(221\right)b_{\nu}, \Phi_6(221)b_{\frac{1}{2}(p-1)}$ or $\Phi_6(221)d_0$ respectively.
Thus,  all  the groups $G$, under consideration, are capable; hence $Z^* (G)=1$.
 
For the group $G=\Phi_3\left(221\right)b_r$, by Theorem \ref{G}, taking $Z=\langle \alpha^p\rangle$, we have the following exact sequence 
$$ G/G' \otimes Z \xrightarrow{\lambda} \M(G) \rightarrow \mathbb{Z}_p \rightarrow 1.$$
By Theorem \ref{G1}, it follows that  $\alpha_1 \otimes \alpha^p \notin \Ker \lambda$, as $\alpha \otimes \alpha^p \in \Ker \lambda$.
So $|\im(\lambda)|=p$, and $|\M(G)|=p^2$. Since $|G'| = p^2$,  $|G \wedge G|=p^4$. 
By Lemma \ref{R}(viii),
$$[\alpha_2,\alpha_3^\phi]=[\alpha_1,\alpha,\alpha_3^\phi]=1.$$
The following identities hold:
\begin{eqnarray*}
&&[\alpha_3,\alpha^\phi]^p= [\alpha_3^p,\alpha^\phi]=1=[\alpha_2^p,\alpha_1^\phi]=[\alpha_2,\alpha_1^\phi]^p, \hspace{2cm}
 \\
&&[\alpha_2,\alpha^\phi]^p=[\alpha_2^p,\alpha^\phi]=1,\\
&& [\alpha_3,\alpha_1^\phi]=[\alpha_1^{pr^{-1}},\alpha_1^\phi]=[\alpha_1,\alpha_1^\phi]^{pr^{-1}}=1, \\
&&[\alpha_3,\alpha^\phi]=[\alpha_1^{pr^{-1}},\alpha^\phi]=[\alpha_1,\alpha^\phi]^{pr^{-1}},\\
&&[\alpha_1,\alpha^\phi]^{p^2}=[\alpha_1^{p^2},\alpha^\phi]=1.
\end{eqnarray*}
Observe that $G \wedge G$ is generated by the set  $\{[\alpha_1,\alpha^{\phi}], [\alpha_2,\alpha^{\phi}], [\alpha_1,\alpha_2^{\phi}]\}$.
By Lemma \ref{R}(vi), $G \wedge G$ is abelian, and therefore
$$G \wedge G\cong \langle[\alpha_1,\alpha^\phi] \rangle \times \langle[\alpha_2,\alpha^\phi] \rangle\times \langle[\alpha_2,\alpha_1^\phi] \rangle \cong \mathbb{Z}_{p^2} \times \mathbb{Z}_p \times \mathbb{Z}_p.$$
Hence $$\M(G) \cong \langle [\alpha_1,\alpha^{\phi}]^p \rangle \times \langle [\alpha_2,\alpha_1^{\phi}] \rangle \cong \mathbb{Z}_p \times \mathbb{Z}_p.$$

For the group $G \cong \Phi_6(221)d_0$, taking $Z= \gen{\beta_2}$, notice that $G/Z \cong \Phi_3(211)b_\nu$. By  Theorem \ref{SHHH}, $\M(G/Z) \cong \mathbb{Z}_p$. Then by Theorem \ref{G}, we have the exact sequence 

\centerline{$ G/G' \otimes Z \xrightarrow{\lambda} M(G) \xrightarrow{\mu}\mathbb{Z}_p \rightarrow \mathbb{Z}_p \rightarrow 1$.}

By Theorem \ref{G1}, it follows that $\alpha_2 G' \otimes \alpha_1^p \notin \Ker \lambda$, as $\alpha_1 G' \otimes \alpha_1^p \in \Ker \lambda$. Hence
$|\im(\lambda)| = p$, and therefore 
$$\M(G) \cong \mathbb{Z}_p.$$ 
Since $|G'| = p^3$, we have 
$$|G \wedge G| = p^4.$$
By Lemma \ref{R}(viii), $[\beta,\beta_1^\phi]=[\alpha_1,\alpha_2,\beta_1^\phi]=1$. Similarly,   $[\beta,\beta_2^\phi]=[\beta_2,\beta_1^\phi]=1$.

We have the following identities:
\begin{eqnarray*}
&&[x,y^\phi]^p=[x^p,y^\phi]=1 ~for~ x \in \{\beta,\beta_1,\beta_2\}, y \in \{\alpha_1,\alpha_2\}, \\
&&[\beta_1,\alpha_1^\phi]=[\alpha_2^{p\nu^{-1}},\alpha_1^\phi]=[\alpha_2,\alpha_1^\phi]^{p\nu^{-1}},\\
&& [\beta_1,\alpha_2^\phi]=[\alpha_2^{p\nu^{-1}},\alpha_2^\phi]=[\alpha_2,\alpha_2^\phi]^{p\nu^{-1}}=1,\\
&& [\beta_2,\alpha_1^\phi]=[\alpha_1^p,\alpha_1^\phi]=[\alpha_1,\alpha_1^\phi]^p=1, \\
&&[\beta_2,\alpha_2^\phi]=[\alpha_1^p,\alpha_2^\phi]=[\alpha_1,\alpha_2^\phi]^p,\\
&&[\alpha_1,\alpha_2^\phi]^{p^2}=[\beta_2,\alpha_2^\phi]^p=1.
\end{eqnarray*}
Hence $G \wedge G$ is generated by $\{[\alpha_1,\alpha_2^\phi], [\beta,\alpha_1^\phi], [\beta,\alpha_2^\phi]\}$ and 
$$G \wedge G \cong \mathbb{Z}_{p^2} \times \mathbb{Z}_p \times \mathbb{Z}_p.$$

Finally we consider the group $G = \Phi_6(221)b_{\frac{1}{2}(p-1)}$. As in the preceding case, taking $Z = \gen{\beta_2}$, it follows that 
$$\M(G) \cong \mathbb{Z}_p.$$
Recall that $\zeta$ is the smallest positive integer which is a primitive root (mod $p$).
In this case, $\zeta^{\frac{1}{2}(p-1)} \equiv -1$ (mod $p$).
By Lemma \ref{R}(viii), $[\beta,\beta_1^\phi]=[\alpha_1,\alpha_2,\beta_1^\phi]=1$. Similarly $[\beta,\beta_2^\phi]=[\beta_2,\beta_1^\phi]=1$.

We have the following identities:
\begin{eqnarray*}
&&[x,y^\phi]^p= [x^p,y^\phi]=1 ~\text{for}~ x \in \{\beta,\beta_1,\beta_2\}, y \in \{\alpha_1,\alpha_2\},\\
&&[\beta_1,\alpha_1^\phi]=[\alpha_1^{-p},\alpha_1^\phi]=[\alpha_1^{-1},\alpha_1^\phi]^{p}=1,\\
&& [\beta_1,\alpha_2^\phi]=[\alpha_1^{-p},\alpha_2^\phi]=[\alpha_1^{-1},\alpha_2^\phi]^p=[\alpha_1,\alpha_2^\phi]^{-p} ,\\
&& [\beta_2,\alpha_1^\phi]=[\alpha_2^p,\alpha_1^\phi]=[\alpha_2,\alpha_1^\phi]^p,\\
&& [\beta_2,\alpha_2^\phi]=[\alpha_2^p,\alpha_2^\phi]=[\alpha_2,\alpha_2^\phi]^p=1\\
&&[\alpha_2,\alpha_1^\phi]^{p^2}=[\beta_2,\alpha_1^\phi]^p=1.
\end{eqnarray*}

Hence, $G \wedge G$ is generated by the set $\{[\alpha_1,\alpha_2^\phi], [\beta,\alpha_1^\phi], [\beta,\alpha_2^\phi]\}$, and therefore
$$G \wedge G \cong \mathbb{Z}_{p^2} \times \mathbb{Z}_p \times \mathbb{Z}_p,$$
which completes the proof.
\hfill $\Box$

\end{proof}

\begin{lemma}
If $G$ is one of the groups $\Phi_6(221)a, \Phi_6(221)b_r  \;\big(r \ne \frac{1}{2}(p-1) \big)$, $\Phi_6(221)c_r$ or  $\Phi_6(221)d_r$, then $\M(G)$ is trivial and $G \wedge G \cong \mathbb{Z}_p^{(3)}$.
\end{lemma}
\begin{proof}
Follows from \cite[Theorem 2.2]{MS}.
\hfill $\Box$

\end{proof}

\begin{lemma}
If $G$ is one of the groups $\Phi_6(2111)a, \Phi_6(2111)b_r$ or $\Phi_6(1^5)$, then $\M(G)$ is isomorphic to  $\mathbb{Z}_p, \mathbb{Z}_p,$ or $\mathbb{Z}_p^{(3)}$ respectively, and $G \wedge G \cong \mathbb{Z}_p^{(4)}, \mathbb{Z}_p^{(4)}$ or $ \mathbb{Z}_p^{(6)}$ respectively.
\end{lemma}
\begin{proof}
For the group $G \cong \Phi_6(2111)a$, by Theorem \ref{2}(i), taking $K=\langle\beta_1 \rangle$, we have $p \leq |\M(G)|$, and  so, since $|G'| = p^3$,  it follows that $p^4 \leq |G \wedge G|$.
By Lemma \ref{R}(viii), 
$$[\beta,\beta_1^\phi] =[\beta,\beta_2^\phi]=[\beta_2,\beta_1^\phi]=1.$$

Also we have
\begin{eqnarray*}
&&[x,y^\phi]^p= [x^p,y^\phi]=1 ~\text{for}~ x \in \{\beta,\beta_1,\beta_2\}, y \in \{\alpha_1,\alpha_2\},\\
&&[\beta_1,\alpha_1^\phi]=[\alpha_1^p,\alpha_1^\phi]=[\alpha_1,\alpha_1^\phi]^p=1,\\ 
&&[\beta_1,\alpha_2^\phi]=[\alpha_1^p,\alpha_2^\phi]=[\alpha_1,\alpha_2^\phi]^p=[\alpha_1,{(\alpha_2^p)}^\phi]=1.
\end{eqnarray*}
Now we have 
\begin{eqnarray*}
[\alpha_2^{-1},\alpha_1^{-1},\beta^\phi]^{\alpha_1} &=& [\alpha_2\alpha_1\beta^{-1}\alpha_1^{-1}\alpha_2^{-1},\beta^\phi]^{\alpha_1}=[\alpha_2\alpha_1\beta_1\alpha_1^{-1}\beta^{-1}\alpha_2^{-1},\beta^\phi]^{\alpha_1}\\
&=&[\alpha_2\beta_1\beta^{-1}\alpha_2^{-1},\beta^\phi]^{\alpha_1}=[\alpha_2\beta_1\beta_2\alpha_2^{-1}\beta^{-1},\beta^\phi]^{\alpha_1}\\
&=& [\beta_1\beta_2\beta^{-1},\beta^\phi]^{\alpha_1}=[\beta,\beta^\phi]^{-1}=1.
\end{eqnarray*}
By Hall-Witt identity, we get
\begin{eqnarray*}
1 &=& [\beta,\alpha_2,\alpha_1^\phi]^{\alpha_2^{-1}}[\alpha_2^{-1},\alpha_1^{-1},\beta^\phi]^{\alpha_1}[\alpha_1,\beta^{-1},\alpha_2^{-\phi}]^{\beta}\hspace{2cm} \\
&=&[\beta_2,\alpha_1^\phi]^{\alpha_2^{-1}}[\beta\beta_1\beta^{-1},\alpha_2^{-\phi}]^{\beta}\\
&=& [\beta_2,\alpha_1^\phi][\beta_1,\alpha_2^{\phi}]^{-1}.
\end{eqnarray*}
Hence, $[\beta_2,\alpha_1^\phi]=[\beta_1,\alpha_2^\phi]=1$. 
Consequently,  $G \wedge G$ is generated by the set
$$\{[\alpha_1,\alpha_2^{\phi}],[\beta,\alpha_1^{\phi}], [\beta,\alpha_2^{\phi}], [\beta_2,\alpha_2^{\phi}]\},$$
and is  elementary abelian by  Lemma \ref{R}(vi).
Hence
$$G \wedge G \cong \mathbb{Z}_p^{(4)},$$
and therefore 
$$\M(G) \cong \mathbb{Z}_p.$$

For the group $G \cong \Phi_6(2111)b_r$,  taking  $K=\langle \beta_1\rangle$, the proof goes on the lines of the preceding case.

Now we consider the group $G=\Phi_6(1^5)$.
By Theorem \ref{J}, taking $K=\langle \beta_1 \rangle$, we have $|\M(G)| \leq p^3$. Since $|G'| = p^3$,  we get $|G \wedge G| \leq p^6$.
As described above,
$$[\beta,\beta_1^\phi]=[\beta,\beta_2^\phi]=[\beta_2,\beta_1^\phi]=1$$
 and, by Hall-Witt identity, we get
$$[\beta_2,\alpha_1^\phi]=[\beta_1,\alpha_2^\phi].$$
Thus $G \wedge G$ is generated by the set
\[ \{ [\alpha_1,\alpha_2^{\phi}],[\beta,\alpha_1^{\phi}], [\beta,\alpha_2^{\phi}], [\beta_1,\alpha_1^\phi],   [\beta_1,\alpha_2^\phi], [\beta_2,\alpha_2^{\phi}]\}. \] 
It follows from Lemma \ref{R}(vi) that $G \wedge G$ is  abelian. A straightforward calculation (as above) shows that each generator of $G \wedge G$ is of order at most $p$.

Consider the natural epimorphism $$[G, G^\phi] \rightarrow [G/\langle \beta_1 \rangle, (G/\langle \beta_1 \rangle)^\phi] \cong [\Phi_3(1^4), \Phi_3(1^4)^\phi],$$ 
which shows that the elements $[\alpha_1,\alpha_2^{\phi}],[\beta,\alpha_1^{\phi}], [\beta,\alpha_2^\phi], [\beta_2,\alpha_2^{\phi}]$ are non-trivial and independent in $G \wedge G$. 
Now, consider the natural epimorphism 
$$[G, G^\phi] \rightarrow [G/\langle \beta_2 \rangle, (G/\langle \beta_2 \rangle)^\phi] \cong [\Phi_3(1^4), \Phi_3(1^4)^\phi],$$ 
which shows that the elements $[\alpha_1,\alpha_2^{\phi}],[\beta,\alpha_1^{\phi}], [\beta,\alpha_2^\phi], [\beta_1,\alpha_1^{\phi}]$ are non-trivial and independent in $G \wedge G$.
We take the quotient group
\begin{eqnarray*}
G/\gen{\beta_1\beta_2} &=&\langle \alpha_1,\alpha_2,\beta,\beta_1 \mid [\alpha_1,\alpha_2]=\beta,[\beta,\alpha_1]=[\alpha_2,\beta]=\beta_1, \alpha_1^p=\alpha_2^p=\beta^p=\beta_1^p=1\rangle\\
& \cong & \langle \alpha_1\alpha_2, \alpha_1, \beta,\beta_1 \mid [\alpha_1,\alpha_1\alpha_2]=\beta, [\beta,\alpha_1]=\beta_1,
[\beta,\alpha_1\alpha_2]=1, \alpha_1^p=(\alpha_1\alpha_2)^p=\\
&&\hspace{9.5cm}\beta^p=\beta_1^p=1\rangle\\
& \cong & \Phi_3(1^4).
\end{eqnarray*}
In $[G/\langle \beta_1\beta_2 \rangle, (G/\langle \beta_1\beta_2 \rangle)^\phi]$, by Hall-Witt identity we have
\begin{eqnarray*}
1 &=& [\beta,\alpha_1,\alpha_2^\phi]^{\alpha_1^{-1}}[\alpha_1^{-1},\alpha_2^{-1},\beta^\phi]^{\alpha_2}[\alpha_2,\beta^{-1},\alpha_1^{-\phi}]^{\beta}\hspace{2cm}\\
&=&[\beta_1,\alpha_2^\phi]^{\alpha_1^{-1}}[\alpha_1\alpha_2\beta\alpha_2^{-1}\alpha_1^{-1},\beta^\phi]^{\alpha_2}[\beta\beta_1^{-1}\beta^{-1},\alpha_1^{-\phi}]^{\beta}\\
&=& [\beta_1,\alpha_2^\phi][\alpha_1\beta\beta_1\alpha_1^{-1},\beta^\phi]^{\alpha_2}[\beta_1^{-1},\alpha_1^{-\phi}]^{\beta}\\
&=&[\beta_1,\alpha_2^\phi][\alpha_1\beta\alpha_1^{-1}\beta_1,\beta^\phi]^{\alpha_2}[\beta_1,\alpha_1^{\phi}]\\
&=& [\beta_1,\alpha_2^\phi][\beta\alpha_1\beta_1^{-1}\alpha_1^{-1}\beta_1,\beta^\phi]^{\alpha_2}[\beta_1,\alpha_1^{\phi}]\\
&=& [\beta_1,\alpha_2^\phi][\beta,\beta^\phi][\beta_1,\alpha_1^{\phi}]\\
&=& [\beta_1,\alpha_2^\phi][\beta_1,\alpha_1^{\phi}].
\end{eqnarray*}
It follows from Theorem \ref{SHHH} that 
$$G/\langle \beta_1\beta_2 \rangle \wedge G/\langle \beta_1\beta_2\rangle \cong \langle[\alpha_1\alpha_2,\alpha_1^\phi], [\beta^{-1}, \alpha_1^\phi],[\beta^{-1},(\alpha_1\alpha_2)^\phi], [\beta_1^{-1},\alpha_1^\phi]\rangle.$$ 
A straightforward computation now shows that
 $$G/\langle \beta_1\beta_2 \rangle \wedge G/\langle \beta_1\beta_2\rangle \cong \langle[\alpha_1,\alpha_2^\phi], [\beta,\alpha_1^\phi],[\beta,\alpha_2^\phi], [\beta_1,\alpha_2^\phi]\rangle \cong \mathbb{Z}_p^{(4)}.$$

Now consider the natural epimorphism
$$[G, G^\phi] \rightarrow [G/\langle \beta_1\beta_2 \rangle, (G/\langle \beta_1\beta_2 \rangle)^\phi] \cong [\Phi_3(1^4), \Phi_3(1^4)^\phi],$$ 
which shows that the generators  $ [\alpha_1,\alpha_2^{\phi}],[\beta,\alpha_1^{\phi}], [\beta,\alpha_2^{\phi}],  [\beta_1,\alpha_2^\phi]$ of $G \wedge G$ are independent.
Thus, it follows that all the six generators  of $G\wedge G$ are non-trivial and independent, and so, we have 
$|G \wedge G| \geq p^6$. 
Hence,
 $$G \wedge G \cong \mathbb{Z}_p^{(6)}$$ 
and 
$$\M(G) \cong \mathbb{Z}_p^{(3)}.$$
\hfill $\Box$

\end{proof}

\begin{lemma}
If $G$ is one of the groups $\Phi_7(2111)a, \Phi_7(2111)b_r, \Phi_7(2111)c$ or $\Phi_7(1^5)$, then $\M(G)$ is isomorphic to $\mathbb{Z}_p^{(3)},  \mathbb{Z}_p^{(3)}, \mathbb{Z}_p^{(3)}$ or $\mathbb{Z}_p^{(4)}$ respectively, and $G \wedge G$ is isomorphic to $\mathbb{Z}_p^{(5)},  \mathbb{Z}_p^{(5)}, \mathbb{Z}_p^{(5)}$ or $\mathbb{Z}_p^{(6)}$ respectively.
\end{lemma}
\begin{proof}
For the groups $G$ belonging  to the isoclinism class $\Phi_7$, it follows from Theorem \ref{2}(ii), taking $K=\Z(G)$, that $d(\M(G)) \geq 3$.

For the group $G=\Phi_7(2111)a$, consider the normal subgroup
$$N=\langle\alpha,\alpha_1,\alpha_2,\alpha_3\rangle \cong \Phi_3(211)a$$
of $G$ of order $p^4$. By Theorem \ref{SHHH}, $\M(N) \cong \mathbb{Z}_p$. Then, by Theorem \ref{J'}, $|\M(G)|$ divides $|\M(N)||N/N'|$. Since $|N/N'| = p^2$, it follows that  $|\M(G)|=p^3$. Hence $G \wedge G$ is of order $p^5$.
By Proposition \ref{B}, $G \wedge G$ is generated by 
\[\{[\alpha_1,\alpha^\phi], [\alpha_2,\alpha^\phi], [\alpha_3,\alpha^\phi], [\beta,\alpha^\phi], [\alpha_2,\alpha_1^\phi],[\alpha_3,\alpha_1^\phi],[\beta,\alpha_1^\phi], [\alpha_3,\alpha_2^\phi], [\beta,\alpha_2^\phi],[\beta,\alpha_3^\phi]\}.\] 
By Lemma \ref{R}(viii), we have $$[\alpha_2,\alpha_3^\phi] =[\alpha_1,\alpha,\alpha_3^\phi]=1.$$
For $x \in \{\alpha,\alpha_1,\alpha_2,\alpha_3,\beta\}$, the following identities hold:
\begin{eqnarray*}
&&[\alpha_3,x^\phi]^p=[\alpha_3^p,x^\phi]=1=[\beta^p,x^\phi]=[\beta,x^\phi]^p,\hspace{4cm}\\
&&[\alpha_2,x^\phi]^p=[\alpha_2^p,x^\phi]=1=[\alpha_1^p,\alpha^\phi]=[\alpha_1,\alpha^\phi]^p.
\end{eqnarray*}
By Lemma \ref{R}(vi), it follows that $G \wedge G$ is abelian, and consequently 
$$G \wedge G\cong \mathbb{Z}_p^{(5)}$$
and
$$\M(G) \cong  \mathbb{Z}_p^{(3)}.$$

For the group $G=\Phi_7(2111)b_r, (r=1$ or $\nu)$, considering the normal subgroup 
$$N=\langle \alpha,\alpha_1,\alpha_2,\alpha_3\rangle \cong \Phi_3(211)b_r$$
of $G$ of order $p^4$, the proof is similar to the preceding case.

For the group $G=\Phi_7(2111)c$, considering the normal subgroup 
$$N=\langle\alpha_1, \alpha\beta,\alpha_2,\alpha_3\rangle \cong \Phi_3(211)a$$ 
of $G$ of order $p^4$, the proof follows as above.

Finally consider the group $G = \Phi_7(1^5)$. It follows from \cite[Main Theorem]{SX} that  
$$\M(G)\cong \mathbb{Z}_p^{(4)}.$$ 
That  $G \wedge G \cong \mathbb{Z}_p^{(6)}$, follows as in the above cases. 
\hfill $\Box$

\end{proof}

 Theorem \ref{E} and all lemmas in  Sections 3, 4 and 5 yield  our main result, which we present in  the following table:

\newpage

\begin{table}[H]
\centering
 \begin{tabular}{||c c c c c c||} 
 \hline
 $G$ & $G^{ab}$ & $\Gamma (G^{ab})$ & $\M(G)$ & $G \wedge G$ &  $G \otimes G$ \\ [.5ex] 
 \hline\hline
$\Phi_2(311)a$ & $\mathbb{Z}_{p^2} \times \mathbb{Z}_p^{(2)}$ &  $\mathbb{Z}_{p^2} \times \mathbb{Z}_p^{(5)}$ & $\mathbb{Z}_p \times \mathbb{Z}_p$ & $\mathbb{Z}_p^{(3)}$ & $\mathbb{Z}_{p^2} \times \mathbb{Z}_p^{(8)}$ \\ 

$\Phi_2(221)a$ & $\mathbb{Z}_{p^2} \times \mathbb{Z}_p^{(2)}$ &  $\mathbb{Z}_{p^2} \times \mathbb{Z}_p^{(5)}$ & $\mathbb{Z}_p^{(3)}$ & $\mathbb{Z}_{p^2} \times \mathbb{Z}_p^{(2)}$ & $\mathbb{Z}_{p^2}^{(2)} \times \mathbb{Z}_p^{(7)}$ \\ 

$\Phi_2(221)b$ & $\mathbb{Z}_{p^2} \times \mathbb{Z}_p^{(2)}$ &  $\mathbb{Z}_{p^2} \times \mathbb{Z}_p^{(5)}$ & $\mathbb{Z}_p \times \mathbb{Z}_p$ & $\mathbb{Z}_p^{(3)}$ & $\mathbb{Z}_{p^2} \times \mathbb{Z}_p^{(8)}$ \\

$\Phi_2(2111)a$ & $\mathbb{Z}_p^{(4)}$ &  $ \mathbb{Z}_p^{(10)}$ & $\mathbb{Z}_p^{(5)}$ & $\mathbb{Z}_p^{(6)}$ & $\mathbb{Z}_p^{(16)}$\\
 
$\Phi_2(2111)b$ & $\mathbb{Z}_p^{(4)}$ &  $ \mathbb{Z}_p^{(10)}$ & $\mathbb{Z}_p^{(5)}$ & $\mathbb{Z}_p^{(6)}$ & $\mathbb{Z}_p^{(16)}$\\

$\Phi_2(2111)c$ & $\mathbb{Z}_{p^2} \times \mathbb{Z}_p^{(2)}$ &  $\mathbb{Z}_{p^2} \times \mathbb{Z}_p^{(5)}$& $\mathbb{Z}_p^{(4)}$ & $\mathbb{Z}_p^{(5)}$ & $\mathbb{Z}_{p^2} \times \mathbb{Z}_p^{(10)}$\\

$\Phi_2(2111)d$ & $\mathbb{Z}_{p^2} \times \mathbb{Z}_p^{(2)}$ &  $\mathbb{Z}_{p^2} \times \mathbb{Z}_p^{(5)}$& $\mathbb{Z}_p^{(4)}$ & $\mathbb{Z}_p^{(5)}$ & $\mathbb{Z}_{p^2} \times \mathbb{Z}_p^{(10)}$\\

$\Phi_2(1^5)$ & $\mathbb{Z}_p^{(4)}$ &  $ \mathbb{Z}_p^{(10)}$ & $\mathbb{Z}_p^{(7)}$ & $\mathbb{Z}_p^{(8)}$ & $\mathbb{Z}_p^{(18)}$\\

 $\Phi_2(41)$ & $\mathbb{Z}_{p^3} \times \mathbb{Z}_p$ & $\mathbb{Z}_{p^3} \times \mathbb{Z}_p^{(2)}$ & $\{1\}$ & $\mathbb{Z}_p$ & $\mathbb{Z}_{p^3} \times \mathbb{Z}_p^{(3)}$ \\ 
 
 $\Phi_2(32)a_1$ & $\mathbb{Z}_{p^2} \times \mathbb{Z}_{p^2}$ & $\mathbb{Z}_{p^2}^{(3)}$ & $\mathbb{Z}_p$ &$\mathbb{Z}_{p^2}$ & $\mathbb{Z}_{p^2}^{(4)}$ \\
 
  $\Phi_2(32)a_2$ & $\mathbb{Z}_{p^3} \times \mathbb{Z}_p$ &  $\mathbb{Z}_{p^3} \times \mathbb{Z}_p^{(2)}$ & $\mathbb{Z}_p$ & $\mathbb{Z}_{p^2}$ & $\mathbb{Z}_{p^3} \times \mathbb{Z}_{p^2}\times \mathbb{Z}_p^{(2)}$  \\
 
  $\Phi_2(311)b$ & $\mathbb{Z}_{p^2} \times \mathbb{Z}_p^{(2)}$ &  $\mathbb{Z}_{p^2} \times \mathbb{Z}_p^{(5)}$ & $\mathbb{Z}_p \times\mathbb{Z}_p$ & $\mathbb{Z}_p^{(3)}$ & $\mathbb{Z}_{p^2} \times \mathbb{Z}_p^{(8)}$  \\
  
 $\Phi_2(311)c$ & $\mathbb{Z}_{p^3} \times \mathbb{Z}_p$ &  $\mathbb{Z}_{p^3} \times \mathbb{Z}_p^{(2)}$ &$\mathbb{Z}_p \times\mathbb{Z}_p$ & $\mathbb{Z}_p^{(3)}$ & $\mathbb{Z}_{p^3} \times \mathbb{Z}_p^{(5)}$  \\ 
 
$\Phi_2(221)c$ & $\mathbb{Z}_{p^2} \times \mathbb{Z}_p^{(2)}$ &  $\mathbb{Z}_{p^2} \times \mathbb{Z}_p^{(5)}$ &$\mathbb{Z}_{p^2} \times\mathbb{Z}_p$ & $\mathbb{Z}_{p^2} \times \mathbb{Z}_p^{(2)}$ & $\mathbb{Z}_{p^2}^{(2)} \times \mathbb{Z}_p^{(7)}$  \\ 
 
$\Phi_2(221)d$ & $\mathbb{Z}_{p^2} \times \mathbb{Z}_{p^2}$ & $\mathbb{Z}_{p^2}^{(3)}$ & $\mathbb{Z}_p^{(3)}$ & $\mathbb{Z}_{p^2} \times \mathbb{Z}_p^{(2)}$ & $\mathbb{Z}_{p^2}^{(4)} \times \mathbb{Z}_p^{(2)}$ \\

$\Phi_3(2111)a$ & $\mathbb{Z}_p^{(3)}$ &  $\mathbb{Z}_p^{(6)}$ & $\mathbb{Z}_p^{(3)}$ & $\mathbb{Z}_p^{(5)}$ & $\mathbb{Z}_p^{(11)}$  \\ 

$\Phi_3(2111)b_r$ & $\mathbb{Z}_p^{(3)}$ &  $\mathbb{Z}_p^{(6)}$ & $\mathbb{Z}_p^{(3)}$ & $\mathbb{Z}_p^{(5)}$ & $\mathbb{Z}_p^{(11)}$  \\ 

$\Phi_3(1^5)$ & $\mathbb{Z}_p^{(3)}$ &  $\mathbb{Z}_p^{(6)}$ & $\mathbb{Z}_p^{(4)}$ & $\mathbb{Z}_p^{(6)}$ & $\mathbb{Z}_p^{(12)}$  \\ 

$\Phi_3(311)a$ & $\mathbb{Z}_{p^2} \times \mathbb{Z}_p$ & $\mathbb{Z}_{p^2} \times \mathbb{Z}_p^{(2)}$ & $\mathbb{Z}_p$ & $\mathbb{Z}_p^{(3)}$ & $\mathbb{Z}_{p^2} \times \mathbb{Z}_p^{(5)}$\\

$\Phi_3(311)b_r$ & $\mathbb{Z}_{p^2} \times \mathbb{Z}_p$ & $\mathbb{Z}_{p^2} \times \mathbb{Z}_p^{(2)}$ & $\mathbb{Z}_p$ & $\mathbb{Z}_p^{(3)}$ & $\mathbb{Z}_{p^2} \times \mathbb{Z}_p^{(5)}$\\

$\Phi_3(221)a$ & $\mathbb{Z}_{p^2} \times \mathbb{Z}_p$ & $\mathbb{Z}_{p^2} \times \mathbb{Z}_p^{(2)}$ & $\mathbb{Z}_p$ & $\mathbb{Z}_p^{(3)}$ & $\mathbb{Z}_{p^2} \times \mathbb{Z}_p^{(5)}$\\

$\Phi_3(221)b_r$ & $\mathbb{Z}_{p^2} \times \mathbb{Z}_p$ & $\mathbb{Z}_{p^2} \times \mathbb{Z}_p^{(2)}$ & $\mathbb{Z}_p \times \mathbb{Z}_p$ & $\mathbb{Z}_{p^2} \times \mathbb{Z}_p^{(2)}$ & $\mathbb{Z}_{p^2}^{(2)} \times \mathbb{Z}_p^{(4)}$\\

$\Phi_3(2111)c$ & $\mathbb{Z}_p^{(3)}$ & $ \mathbb{Z}_p^{(6)}$ & $\mathbb{Z}_p^{(3)}$ & $\mathbb{Z}_p^{(5)}$ & $\mathbb{Z}_p^{(11)}$\\

$\Phi_3(2111)d$ & $\mathbb{Z}_{p^2} \times \mathbb{Z}_p$ & $\mathbb{Z}_{p^2} \times \mathbb{Z}_p^{(2)}$ & $\mathbb{Z}_p^{(2)}$ & $\mathbb{Z}_p^{(4)}$ & $\mathbb{Z}_{p^2} \times \mathbb{Z}_p^{(6)}$\\

$\Phi_3(2111)e$ & $\mathbb{Z}_{p^2} \times \mathbb{Z}_p$ & $\mathbb{Z}_{p^2} \times \mathbb{Z}_p^{(2)}$ & $\mathbb{Z}_p^{(2)}$ & $\mathbb{Z}_p^{(4)}$ & $\mathbb{Z}_{p^2} \times \mathbb{Z}_p^{(6)}$\\

$\Phi_4(221)a$ & $\mathbb{Z}_p^{(3)}$ &  $\mathbb{Z}_p^{(6)}$ & $\mathbb{Z}_p$ & $\mathbb{Z}_p^{(3)}$ & $\mathbb{Z}_p^{(9)}$  \\ 

$\Phi_4(221)b$ & $\mathbb{Z}_p^{(3)}$ &  $\mathbb{Z}_p^{(6)}$ & $\mathbb{Z}_p \times \mathbb{Z}_p$ & $\mathbb{Z}_{p^2} \times \mathbb{Z}_p^{(2)}$ & $\mathbb{Z}_{p^2} \times \mathbb{Z}_p^{(8)}$  \\ 

$\Phi_4(221)c$ & $\mathbb{Z}_p^{(3)}$ &  $\mathbb{Z}_p^{(6)}$ & $\mathbb{Z}_p$ & $\mathbb{Z}_p^{(3)}$ & $\mathbb{Z}_p^{(9)}$  \\ 

$\Phi_4(221)d_r,\;r \ne \frac{1}{2}(p-1)$ & $\mathbb{Z}_p^{(3)}$ &  $\mathbb{Z}_p^{(6)}$ & $\mathbb{Z}_p$ & $\mathbb{Z}_p^{(3)}$ & $\mathbb{Z}_p^{(9)}$  \\ 

$\Phi_4(221)d_{\frac{1}{2}(p-1)}$ & $\mathbb{Z}_p^{(3)}$ &  $\mathbb{Z}_p^{(6)}$ & $\mathbb{Z}_{p^2}$ & $\mathbb{Z}_{p^2} \times\mathbb{Z}_p^{(2)}$ & $\mathbb{Z}_{p^2}\times \mathbb{Z}_p^{(8)}$  \\ 

$\Phi_4(221)e$ & $\mathbb{Z}_p^{(3)}$ &  $\mathbb{Z}_p^{(6)}$ & $\mathbb{Z}_p$ & $\mathbb{Z}_p^{(3)}$ & $\mathbb{Z}_p^{(9)}$  \\ 

$\Phi_4(221)f_0$ & $\mathbb{Z}_p^{(3)}$ &  $\mathbb{Z}_p^{(6)}$ & $\mathbb{Z}_{p^2}$ & $\mathbb{Z}_{p^2} \times \mathbb{Z}_p^{(2)}$ & $\mathbb{Z}_{p^2} \times \mathbb{Z}_p^{(8)}$  \\

$\Phi_4(221)f_r$ & $\mathbb{Z}_p^{(3)}$ &  $\mathbb{Z}_p^{(6)}$ & $\mathbb{Z}_p$ & $\mathbb{Z}_p^{(3)}$ & $\mathbb{Z}_p^{(9)}$  \\

$\Phi_4(2111)a$ & $\mathbb{Z}_p^{(3)}$ &  $\mathbb{Z}_p^{(6)}$ &  $  \mathbb{Z}_p^{(3)}$ & $\mathbb{Z}_p^{(5)}$ & $\mathbb{Z}_p^{(11)}$  \\

$\Phi_4(2111)b$ & $\mathbb{Z}_p^{(3)}$ &  $\mathbb{Z}_p^{(6)}$ &  $  \mathbb{Z}_p^{(3)}$ & $\mathbb{Z}_p^{(5)}$ & $\mathbb{Z}_p^{(11)}$  \\

$\Phi_4(2111)c$ & $\mathbb{Z}_p^{(3)}$ &  $\mathbb{Z}_p^{(6)}$ &  $  \mathbb{Z}_p^{(3)}$ & $\mathbb{Z}_p^{(5)}$ & $\mathbb{Z}_p^{(11)}$  \\

$\Phi_4(1^5)$ & $\mathbb{Z}_p^{(3)}$ &  $\mathbb{Z}_p^{(6)}$ &  $  \mathbb{Z}_p^{(6)}$ & $\mathbb{Z}_p^{(8)}$ & $\mathbb{Z}_p^{(14)}$  \\

$\Phi_5(2111)$ & $\mathbb{Z}_p^{(4)}$ &  $\mathbb{Z}_p^{(10)}$ &  $\mathbb{Z}_p^{(5)}$ & $\mathbb{Z}_p^{(6)}$ & $\mathbb{Z}_p^{(16)}$  \\

$\Phi_5(1^5)$ & $\mathbb{Z}_p^{(4)}$ &  $\mathbb{Z}_p^{(10)}$ &  $\mathbb{Z}_p^{(5)}$ & $\mathbb{Z}_p^{(6)}$ & $\mathbb{Z}_p^{(16)}$  \\
[.5ex] 
 \hline
\end{tabular}
\end{table}

\begin{table}[H]
\centering
 \begin{tabular}{||c c c c c c||} 
 \hline
 $G$ & $G^{ab}$ & $\Gamma (G^{ab})$ & $\M(G)$ & $G \wedge G$ &  $G \otimes G$ \\ [.5ex] 
 \hline\hline
$\Phi_6(221)a$ & $\mathbb{Z}_p^{(2)}$ &  $\mathbb{Z}_p^{(3)}$ &  $\{1\}$ & $\mathbb{Z}_p^{(3)}$ & $\mathbb{Z}_p^{(6)}$  \\

$\Phi_6(221)b_r, \;r \ne \frac{1}{2}(p-1)$ & $\mathbb{Z}_p^{(2)}$ &  $\mathbb{Z}_p^{(3)}$ &  $\{1\}$ & $\mathbb{Z}_p^{(3)}$ & $\mathbb{Z}_p^{(6)}$  \\

$\Phi_6(221)b_{\frac{1}{2}(p-1)}$ & $\mathbb{Z}_p^{(2)}$ &  $\mathbb{Z}_p^{(3)}$ &  $\mathbb{Z}_p$ & $\mathbb{Z}_{p^2} \times \mathbb{Z}_p^{(2)}$ & $\mathbb{Z}_{p^2} \times \mathbb{Z}_p^{(5)}$  \\

$\Phi_6(221)c_r$ & $\mathbb{Z}_p^{(2)}$ &  $\mathbb{Z}_p^{(3)}$ &  $\{1\}$ & $\mathbb{Z}_p^{(3)}$ & $\mathbb{Z}_p^{(6)}$  \\

$\Phi_6(221)d_0$ & $\mathbb{Z}_p^{(2)}$ &  $\mathbb{Z}_p^{(3)}$ &  $\mathbb{Z}_p$ & $\mathbb{Z}_{p^2} \times \mathbb{Z}_p^{(2)}$ & $\mathbb{Z}_{p^2} \times \mathbb{Z}_p^{(5)}$  \\

$\Phi_6(221)d_r$ & $\mathbb{Z}_p^{(2)}$ &  $\mathbb{Z}_p^{(3)}$ &  $\{1\}$ & $\mathbb{Z}_p^{(3)}$ & $\mathbb{Z}_p^{(6)}$  \\

$\Phi_6(2111)a$ & $\mathbb{Z}_p^{(2)}$ &  $\mathbb{Z}_p^{(3)}$ &  $\mathbb{Z}_p$ & $\mathbb{Z}_p^{(4)}$ & $\mathbb{Z}_p^{(7)}$  \\

$\Phi_6(2111)b_r$ & $\mathbb{Z}_p^{(2)}$ &  $\mathbb{Z}_p^{(3)}$ &  $\mathbb{Z}_p$ & $\mathbb{Z}_p^{(4)}$ & $\mathbb{Z}_p^{(7)}$  \\

$\Phi_6(1^5)$ & $\mathbb{Z}_p^{(2)}$ &  $\mathbb{Z}_p^{(3)}$ &  $\mathbb{Z}_p^{(3)}$ & $\mathbb{Z}_p^{(6)}$ & $\mathbb{Z}_p^{(9)}$  \\

$\Phi_7(2111)a$ & $\mathbb{Z}_p^{(3)}$ &  $\mathbb{Z}_p^{(6)}$ &  $\mathbb{Z}_p^{(3)}$ & $\mathbb{Z}_p^{(5)}$ & $\mathbb{Z}_p^{(11)}$  \\

$\Phi_7(2111)b_r$ & $\mathbb{Z}_p^{(3)}$ &  $\mathbb{Z}_p^{(6)}$ &  $\mathbb{Z}_p^{(3)}$ & $\mathbb{Z}_p^{(5)}$ & $\mathbb{Z}_p^{(11)}$  \\

$\Phi_7(2111)c$ & $\mathbb{Z}_p^{(3)}$ &  $\mathbb{Z}_p^{(6)}$ &  $\mathbb{Z}_p^{(3)}$ & $\mathbb{Z}_p^{(5)}$ & $\mathbb{Z}_p^{(11)}$  \\

$\Phi_7(1^5)$ & $\mathbb{Z}_p^{(3)}$ &  $\mathbb{Z}_p^{(6)}$ &  $\mathbb{Z}_p^{(4)}$ & $\mathbb{Z}_p^{(6)}$ & $\mathbb{Z}_p^{(12)}$  \\

$\Phi_8(32)$ & $\mathbb{Z}_{p^2} \times \mathbb{Z}_p$ &  $\mathbb{Z}_{p^2} \times \mathbb{Z}_p^{(2)}$ & $\{1\}$ & $\mathbb{Z}_{p^2}$ & $\mathbb{Z}_{p^2}^{(2)} \times \mathbb{Z}_p^{(2)}$  \\

$\Phi_9(2111)a$ & $\mathbb{Z}_p^{(2)}$ &  $\mathbb{Z}_p^{(3)}$ & $\mathbb{Z}_p$ & $\mathbb{Z}_p^{(4)}$ & $\mathbb{Z}_p^{(7)}$  \\

$\Phi_9(2111)b_r$ & $\mathbb{Z}_p^{(2)}$ &  $\mathbb{Z}_p^{(3)}$ & $\mathbb{Z}_p$ & $\mathbb{Z}_p^{(4)}$ & $\mathbb{Z}_p^{(7)}$  \\

$\Phi_9(1^5)$ & $\mathbb{Z}_p^{(2)}$ &  $\mathbb{Z}_p^{(3)}$ &  $  \mathbb{Z}_p^{(3)}$ & $\Phi_2(111) \times \mathbb{Z}_p^{(3)}$ & $\Phi_2(111)\times \mathbb{Z}_p^{(6)}$  \\

$\Phi_{10}(2111)a_r$ & $\mathbb{Z}_p^{(2)}$ &  $\mathbb{Z}_p^{(3)}$ & $\mathbb{Z}_p$ & $\mathbb{Z}_p^{(4)}$ & $\mathbb{Z}_p^{(7)}$  \\

$\Phi_{10}(2111)b_r$ & $\mathbb{Z}_p^{(2)}$ &  $\mathbb{Z}_p^{(3)}$ & $\mathbb{Z}_p$ & $\mathbb{Z}_p^{(4)}$ & $\mathbb{Z}_p^{(7)}$  \\

$\Phi_{10}(1^5)$ & $\mathbb{Z}_p^{(2)}$ &  $\mathbb{Z}_p^{(3)}$ &  $  \mathbb{Z}_p^{(3)}$ & $\Phi_2(111) \times \mathbb{Z}_p^{(3)}$ & $\Phi_2(111)\times \mathbb{Z}_p^{(6)}$  \\
 [.5ex] 
 \hline
\end{tabular}
\hspace{1cm}
\caption{Groups of order $p^5$, $p \ge 5$}\label{Table1}
\end{table}

In the following table  we determine the capability of non-abelian $p$-groups of order $p^5$, $p \ge 5$,  using Theorem \ref{G1}. The description of epicenter $Z^*(G)$ is also given for each group $G$ under consideration.
\begin{table}[H]
\centering
 \begin{tabular}{||c c c | c c c||} 
 \hline
 $G$ & Capability & Epicenter & $G$ & Capability & Epicenter\\ [.5ex] 
 \hline\hline
$\Phi_2(311)a$ & Not Capable & $\langle \alpha^p \rangle$ & $\Phi_4(221)e$ &   Not capable & $\Z(G)$\\
 
$\Phi_2(221)a$ & Capable & $\langle 1 \rangle$ & $\Phi_4(221)f_0$ &  Capable & $\langle 1 \rangle$\\
 
$\Phi_2(221)b$ & Not Capable & $\gamma_2(G) \times \langle \alpha_3^p \rangle$ &$\Phi_4(221)f_r$ &  Not capable & $\Z(G)$\\
 
$\Phi_2(2111)a$ & Not Capable & $\gamma_2(G)$  &$\Phi_4(2111)a$ & Not capable & $\langle \beta_2 \rangle$\\
  
$\Phi_2(2111)b$ & Not Capable & $\gamma_2(G)$  & $\Phi_4(2111)b$ &  Not capable & $\langle \beta_1 \rangle$\\
  
$\Phi_2(2111)c$ & Not Capable  & $\langle \alpha^p \rangle$ & $\Phi_4(2111)c$ &  Not Capable & $\langle \beta_1 \rangle$\\
  
$\Phi_2(2111)d$ & Not Capable & $\langle \alpha_3^p \rangle$ & $\Phi_4(1^5)$ &  Capable & $\langle 1 \rangle$\\
  
$\Phi_2(1^5)$ & Capable & $\langle 1 \rangle$ & $\Phi_5(2111)$ &  Not capable & $\Z(G)$ \\
 
$\Phi_2(41)$ & Not Capable  & $\Z(G$)  & $\Phi_5(1^5)$ &  Not capable & $\Z(G)$\\
 
$\Phi_2(32)a_1$ &  Not Capable & $\gamma_2(G)$ & $\Phi_6(221)a$ &  Not capable & $\Z(G)$\\
 
$\Phi_2(32)a_2$ &  Not capable & $\langle \alpha^{p^2} \rangle$  & $\Phi_6(221)b_r,\;r \ne \frac{1}{2}(p-1)$ &   Not capable & $\Z(G)$\\
 
$\Phi_2(311)b$ &  Not capable &  $\langle \gamma^p \rangle$  &$\Phi_6(221)b_{\frac{1}{2}(p-1)}$ &  Capable  & $\langle 1 \rangle$\\
  
$\Phi_2(311)c$ & Not capable &  $\langle \alpha^p \rangle$ & $\Phi_6(221)c_r$ &  Not capable & $\Z(G)$\\
 
$\Phi_2(221)c$ &  Capable & $\langle 1 \rangle$ & $\Phi_6(221)d_0$ & Capable & $\langle 1 \rangle$\\
 
$\Phi_2(221)d$ & Capable & $\langle 1 \rangle$  & $\Phi_6(221)d_r$ &  Not capable & $\Z(G)$\\

$\Phi_3(2111)a$ & Not capable & $\langle \alpha_3 \rangle$& $\Phi_6(2111)a$ &  Not capable & $\langle \beta_1 \rangle$\\

$\Phi_3(2111)b_r$ & Not capable & $\langle \alpha_3 \rangle$ & $\Phi_6(2111)b_r$ &  Not capable & $\langle \beta_1 \rangle$ \\

$\Phi_3(1^5)$ & Capable & $\langle 1 \rangle$ & $\Phi_6(1^5)$ & Capable & $\langle 1 \rangle$\\

$\Phi_3(311)a$ & Not capable &  $\Z(G)$  & $\Phi_7(2111)a$ &  Not capable & $\Z(G)$\\

$\Phi_3(311)b_r$  & Not capable & $\Z(G)$ &$\Phi_7(2111)b_r$ &  Not capable & $\Z(G)$\\

$\Phi_3(221)a$ & Not capable &  $\Z(G)$ & $\Phi_7(2111)c$ &  Not capable & $\Z(G)$\\

$\Phi_3(221)b_r$ & Capable & $\langle 1 \rangle$ & $\Phi_7(1^5)$ &  Capable & $\langle 1 \rangle$\\

$\Phi_3(2111)c$ &  Not capable&  $\langle \gamma^p \rangle$  & $\Phi_8(32)$ &  Not capable & $\Z(G)$\\

$\Phi_3(2111)d$ & Not capable & $\langle \alpha^p \rangle$ & $\Phi_9(2111)a$ &  Not capable & $\Z(G)$\\

$\Phi_3(2111)e$ & Not capable & $\langle \alpha_1^p \rangle$ &$\Phi_9(2111)b_r$ & Not capable & $\Z(G)$\\

$\Phi_4(221)a$ & Not capable & $\Z(G)$ & $\Phi_9(1^5)$ &  Capable & $\langle 1 \rangle$\\ 

$\Phi_4(221)b$ &  Capable & $\langle 1 \rangle$ & $\Phi_{10}(2111)a_r$ &  Not capable & $\Z(G)$\\ 

$\Phi_4(221)c$ &   Not capable & $\Z(G)$ & $\Phi_{10}(2111)b_r$ & Not capable & $\Z(G)$\\ 

$\Phi_4(221)d_r, \;r \ne \frac{1}{2}(p-1)$ & Not capable & $\Z(G)$ & $\Phi_{10}(1^5)$ & Capable & $\langle 1 \rangle$ \\

$\Phi_4(221)d_{\frac{1}{2}(p-1)}$ & Capable & $\langle 1 \rangle$ &  &  &\\
 [.5ex] 
 \hline
 \end{tabular}
 \hspace{1cm}
\caption{Capability of groups of order $p^5$, $p \ge 5$}\label{Table2}
\end{table}

\section{Non-abelian groups of order $2^5$ and $3^5$}

The following table takes care of groups of order $2^5$.
\begin{table}[H]
\centering
 \begin{tabular}{||c c c c c c||} 
 \hline
 Group ID & $\M(G)$ & $G \wedge G$ & $G \otimes G$ & Capability & Epicenter \\ 
 [.5ex] 
 \hline\hline
$2$ & $\mathbb{Z}_2^{(3)}$ & $\mathbb{Z}_4 \times \mathbb{Z}_2^{(2)}$ &  $\mathbb{Z}_4^{(4)} \times \mathbb{Z}_2^{(2)}$ & Capable & $\{1\}$
\\

$4$ & $\mathbb{Z}_2$ & $\mathbb{Z}_4$ &  $\mathbb{Z}_8 \times \mathbb{Z}_4^{(3)}$ & Not Capable & $\mathbb{Z}_2$\\

$5$ & $\mathbb{Z}_2^{(2)}$ & $\mathbb{Z}_4 \times \mathbb{Z}_2$ &  $\mathbb{Z}_8 \times \mathbb{Z}_4 \times \mathbb{Z}_2^{(3)}$ & Not Capable & $\mathbb{Z}_2$\\

$6$ & $\mathbb{Z}_2^{(2)}$ & $\mathbb{Z}_4 \times \mathbb{Z}_2^{(2)}$ & $\mathbb{Z}_4^{(2)} \times \mathbb{Z}_2^{(4)}$  & Capable & $\{1\}$\\

$7$ & $\mathbb{Z}_2$ & $\mathbb{Z}_4 \times \mathbb{Z}_2$& $\mathbb{Z}_8 \times \mathbb{Z}_4 \times \mathbb{Z}_2^{(3)}$ & Not Capable  & $\mathbb{Z}_2$\\

$8$ & $\mathbb{Z}_2$ & $\mathbb{Z}_4 \times \mathbb{Z}_2$ & $\mathbb{Z}_8 \times \mathbb{Z}_4 \times \mathbb{Z}_2^{(3)}$  & Not Capable & $\mathbb{Z}_2$\\

$9$ & $\mathbb{Z}_2^{(2)}$ & $\mathbb{Z}_8 \times \mathbb{Z}_2$ &  $\mathbb{Z}_8 \times \mathbb{Z}_4 \times \mathbb{Z}_2^{(3)}$ & Capable & $\{1\}$ \\

$10$ & $\mathbb{Z}_2$ & $\mathbb{Z}_4 \times \mathbb{Z}_2$ &  $\mathbb{Z}_8 \times \mathbb{Z}_4 \times \mathbb{Z}_2^{(3)}$ & Not Capable & $\mathbb{Z}_2$\\

$11$ & $\mathbb{Z}_2$ & $\mathbb{Z}_4 \times \mathbb{Z}_2$ & $\mathbb{Z}_4^{(2)} \times \mathbb{Z}_2^{(3)}$ & Not Capable  & $\mathbb{Z}_2$\\

$12$ & $\mathbb{Z}_2$ & $\mathbb{Z}_4$ & $\mathbb{Z}_8 \times \mathbb{Z}_4^{(2)} \times \mathbb{Z}_2$ & Not Capable & $\mathbb{Z}_4$ \\

$13$ & $\mathbb{Z}_2$ & $\mathbb{Z}_8$ &  $\mathbb{Z}_8 \times \mathbb{Z}_4^{(2)} \times \mathbb{Z}_2$  & Capable & $\{1\}$ \\

$14$ & $\mathbb{Z}_2$ & $\mathbb{Z}_8$ &  $\mathbb{Z}_8 \times \mathbb{Z}_4^{(2)} \times \mathbb{Z}_2$  & Not Capable & $\mathbb{Z}_2$\\

$15$ & $\{1\}$ & $\mathbb{Z}_4$ &  $\mathbb{Z}_8 \times \mathbb{Z}_4^{(2)} \times \mathbb{Z}_2$  & Not Capable & $\mathbb{Z}_4$\\

$17$ & $\{1\}$ & $\mathbb{Z}_2$ &   $\mathbb{Z}_{16} \times \mathbb{Z}_2^{(3)}$ & Not Capable & $\mathbb{Z}_8$\\

$18$ & $\mathbb{Z}_2$ & $\mathbb{Z}_{16}$ &  $\mathbb{Z}_{16} \times \mathbb{Z}_2^{(3)}$ & Capable & $\{1\}$\\

$19$ & $\{1\}$ & $\mathbb{Z}_8$ & $\mathbb{Z}_{16} \times \mathbb{Z}_2^{(3)}$ & Not Capable  & $\mathbb{Z}_2$\\

$20$ & $\{1\}$ & $\mathbb{Z}_8$ &  $\mathbb{Z}_{16} \times \mathbb{Z}_2^{(3)}$ & Not Capable  & $\mathbb{Z}_2$\\

$22$ & $\mathbb{Z}_2^{(4)}$ & $\mathbb{Z}_4 \times \mathbb{Z}_2^{(3)}$ &  $\mathbb{Z}_{4}^{(2)} \times \mathbb{Z}_2^{(8)}$ & Capable & $\{1\}$\\

$23$ & $\mathbb{Z}_2^{(3)}$ & $\mathbb{Z}_4 \times \mathbb{Z}_2^{(2)}$ & $\mathbb{Z}_{4}^{(3)} \times \mathbb{Z}_2^{(6)}$  & Not Capable & $\mathbb{Z}_2$\\

$24$ & $\mathbb{Z}_2 \times \mathbb{Z}_4$ & $\mathbb{Z}_4 \times \mathbb{Z}_2^{(2)}$ & $\mathbb{Z}_{4}^{(2)} \times \mathbb{Z}_2^{(7)}$ & Capable & $\{1\}$\\

$25$ & $\mathbb{Z}_2^{(3)}$ & $\mathbb{Z}_4 \times \mathbb{Z}_2^{(2)}$ & $\mathbb{Z}_{4}^{(2)} \times \mathbb{Z}_2^{(7)}$ & Not Capable  & $\mathbb{Z}_2$\\

$26$ & $\mathbb{Z}_2^{(2)}$ & $\mathbb{Z}_2^{(3)}$&  $\mathbb{Z}_{4}^{(3)} \times \mathbb{Z}_2^{(6)}$ & Not Capable & $\mathbb{Z}_2 \times\mathbb{Z}_2$\\

$27$ & $\mathbb{Z}_2^{(4)}$ & $\mathbb{Z}_4^{(2)} \times \mathbb{Z}_2^{(2)}$&  $\mathbb{Z}_{4}^{(2)} \times \mathbb{Z}_2^{(8)}$ & Capable & $\{1\}$\\

$28$ & $\mathbb{Z}_2^{(3)}$ & $\mathbb{Z}_4^{(2)} \times \mathbb{Z}_2$&  $\mathbb{Z}_{4}^{(2)} \times \mathbb{Z}_2^{(7)}$ & Capable  &$\{1\}$ \\

$29$ & $\mathbb{Z}_2^{(2)}$ & $\mathbb{Z}_4 \times \mathbb{Z}_2^{(2)}$ &  $\mathbb{Z}_{4}^{(3)} \times \mathbb{Z}_2^{(6)}$ & Not Capable & $\mathbb{Z}_2$\\

$30$ & $\mathbb{Z}_2^{(2)}$ & $\mathbb{Z}_4 \times \mathbb{Z}_2^{(2)}$ & $\mathbb{Z}_{4}^{(2)} \times \mathbb{Z}_2^{(7)}$ & Not Capable  & $\mathbb{Z}_2$\\

$31$ & $\mathbb{Z}_2 \times \mathbb{Z}_4$ & $\mathbb{Z}_4^{(2)} \times \mathbb{Z}_2$ & $\mathbb{Z}_{4}^{(3)} \times \mathbb{Z}_2^{(6)}$ & Capable  &$\{1\}$\\

$32$ & $\mathbb{Z}_2$ & $\mathbb{Z}_2^{(3)}$ &  $\mathbb{Z}_{4}^{(2)} \times \mathbb{Z}_2^{(7)}$ & Not Capable & $\mathbb{Z}_2 \times\mathbb{Z}_2$ \\

$33$ & $\mathbb{Z}_2$ & $\mathbb{Z}_2^{(3)}$ & $\mathbb{Z}_{4}^{(2)} \times \mathbb{Z}_2^{(7)}$ &  Not Capable  & $\mathbb{Z}_2 \times\mathbb{Z}_2$\\

$34$ & $\mathbb{Z}_2^{(2)} \times \mathbb{Z}_4$ & $\mathbb{Z}_4^{(3)}$ &  $\mathbb{Z}_{4}^{(3)} \times \mathbb{Z}_2^{(6)}$ &  Capable & $\{1\}$\\

$35$ & $\mathbb{Z}_2^{(2)}$ & $\mathbb{Z}_4 \times \mathbb{Z}_2^{(2)}$ &  $\mathbb{Z}_{4}^{(4)} \times \mathbb{Z}_2^{(5)}$ & Not Capable & $\mathbb{Z}_2$ \\

$37$ & $\mathbb{Z}_2^{(2)}$ &  $\mathbb{Z}_2^{(3)}$ & $\mathbb{Z}_8 \times \mathbb{Z}_2^{(8)}$ & Not Capable  & $\mathbb{Z}_4$\\

$38$ & $\mathbb{Z}_2^{(2)}$ &  $\mathbb{Z}_2^{(3)}$ &  $\mathbb{Z}_4 \times \mathbb{Z}_2^{(8)}$ & Not Capable & $\mathbb{Z}_4$\\

$39$ & $\mathbb{Z}_2^{(3)}$ & $\mathbb{Z}_8 \times \mathbb{Z}_2^{(2)}$&  $\mathbb{Z}_8 \times \mathbb{Z}_2^{(8)}$ & Capable & $\{1\}$\\

$40$ & $\mathbb{Z}_2^{(2)}$ & $\mathbb{Z}_4 \times \mathbb{Z}_2^{(2)}$ &  $\mathbb{Z}_8 \times \mathbb{Z}_2^{(8)}$ & Not Capable & $\mathbb{Z}_2$\\

$41$ & $\mathbb{Z}_2^{(2)}$ & $\mathbb{Z}_4 \times \mathbb{Z}_2^{(2)}$ & $\mathbb{Z}_8 \times \mathbb{Z}_2^{(8)}$ & Not Capable & $\mathbb{Z}_2$\\

$42$ & $\mathbb{Z}_2^{(2)}$ & $\mathbb{Z}_4 \times \mathbb{Z}_2^{(2)}$ &   $\mathbb{Z}_4 \times \mathbb{Z}_2^{(8)}$ & Not Capable & $\mathbb{Z}_2$\\

$43$ & $\mathbb{Z}_2^{(2)}$ & $\mathbb{Z}_4 \times \mathbb{Z}_2^{(2)}$ &   $\mathbb{Z}_4 \times \mathbb{Z}_2^{(8)}$ & Not Capable & $\mathbb{Z}_2$\\

$44$ & $\mathbb{Z}_2^{(2)}$ & $\mathbb{Z}_4 \times \mathbb{Z}_2^{(2)}$ &  $\mathbb{Z}_4 \times \mathbb{Z}_2^{(8)}$ & Not Capable & $\mathbb{Z}_2$\\

$46$ & $\mathbb{Z}_2^{(6)}$ & $\mathbb{Z}_4 \times \mathbb{Z}_2^{(5)}$ &  $\mathbb{Z}_4 \times \mathbb{Z}_2^{(15)}$ & Capable & $\{1\}$\\

$47$ & $\mathbb{Z}_2^{(5)}$ & $\mathbb{Z}_2^{(6)}$  &  $\mathbb{Z}_4^{(2)} \times \mathbb{Z}_2^{(14)}$ &  Not Capable & $\mathbb{Z}_2$\\

$48$ & $\mathbb{Z}_2^{(5)}$ & $\mathbb{Z}_2^{(6)}$ &  $\mathbb{Z}_2^{(16)}$ & Not Capable & $\mathbb{Z}_2$\\

$49$ & $\mathbb{Z}_2^{(5)}$ & $\mathbb{Z}_2^{(6)}$ &  $\mathbb{Z}_2^{(16)}$ & Not Capable & $\mathbb{Z}_2$\\

$50$ & $\mathbb{Z}_2^{(5)}$ & $\mathbb{Z}_2^{(6)}$ &   $\mathbb{Z}_2^{(16)}$ & Not Capable & $\mathbb{Z}_2$\\
[.5ex] 
 \hline
 \end{tabular}
 \hspace{1cm}
\caption{Groups of order $2^5$}\label{Table3}
\end{table}

The following table takes care of groups of order $3^5$.

\begin{table}[H]
\centering
 \begin{tabular}{||c c c c c c||} 
 \hline
 Group ID & $\M(G)$ & $G \wedge G$ & $G \otimes G$ & Capability & Epicenter \\ 
 [.5ex] 
 \hline\hline
$2$ & $\mathbb{Z}_3^{(3)}$ & $\mathbb{Z}_9 \times \mathbb{Z}_3^{(2)}$  & $\mathbb{Z}_9^{(4)} \times \mathbb{Z}_3^{(2)}$   & Capable  &$\{1\}$ \\

$3$ & $\mathbb{Z}_3^{(2)}$ & $\mathbb{Z}_9 \times \mathbb{Z}_3^{(3)}$ & $\mathbb{Z}_9 \times \mathbb{Z}_3^{(6)}$  &  Capable  &$\{1\}$ \\

$4$ & $\mathbb{Z}_3$ & $\mathbb{Z}_9 \times \mathbb{Z}_3^{(2)}$ &   $\mathbb{Z}_9 \times \mathbb{Z}_3^{(5)}$   & Capable  & $\{1\}$\\

$5$ & $\{1\}$ & $\mathbb{Z}_3^{(3)}$ & $\mathbb{Z}_3^{(6)}$  & Not Capable  & $\mathbb{Z}_3^{(2)}$\\

$6$ & $\mathbb{Z}_3$ & $\mathbb{Z}_9 \times \mathbb{Z}_3^{(2)}$ &   $\mathbb{Z}_9 \times \mathbb{Z}_3^{(5)}$ & Not Capable & $\mathbb{Z}_3$ \\

$7$ & $\{1\}$ & $\mathbb{Z}_3^{(3)}$  & $\mathbb{Z}_3^{(6)}$  & Not Capable   & $\mathbb{Z}_3^{(2)}$ \\

$8$ & $\mathbb{Z}_3$ & $\mathbb{Z}_9 \times \mathbb{Z}_3^{(2)}$ &   $\mathbb{Z}_9 \times \mathbb{Z}_3^{(5)}$ & Not Capable & $\mathbb{Z}_3$\\

$9$ & $\mathbb{Z}_3$ & $\mathbb{Z}_3^{(4)}$ & $\mathbb{Z}_3^{(7)}$  & Capable  & $\{1\}$ \\

$11$ & $\mathbb{Z}_3$ & $\mathbb{Z}_9$  & $ \mathbb{Z}_9^{(4)}$  &  Not Capable & $\mathbb{Z}_3$ \\

$12$ & $\mathbb{Z}_3^{(2)}$ & $\mathbb{Z}_3^{(3)}$  & $\mathbb{Z}_{27} \times \mathbb{Z}_3^{(5)}$  &  Not Capable & $\mathbb{Z}_9$ \\

$13$ & $\mathbb{Z}_3^{(2)}$ & $\mathbb{Z}_3^{(4)}$ &  $\mathbb{Z}_9 \times \mathbb{Z}_3^{(6)}$ & Capable  & $\{1\}$ \\

$14$ & $\mathbb{Z}_3^{(2)}$ & $\mathbb{Z}_9 \times \mathbb{Z}_3^{(2)}$ &  $\mathbb{Z}_9^{(2)}\times \mathbb{Z}_3^{(4)}$ & Capable  & $\{1\}$ \\

$15$ & $\mathbb{Z}_3^{(2)}$ & $\mathbb{Z}_9 \times \mathbb{Z}_3^{(2)}$  & $\mathbb{Z}_9^{(2)}\times \mathbb{Z}_3^{(4)}$  & Not Capable & $\mathbb{Z}_3$ \\

$16$ & $\mathbb{Z}_3$ & $\mathbb{Z}_3^{(3)}$  & $\mathbb{Z}_9 \times \mathbb{Z}_3^{(5)}$  &  Not Capable & $\mathbb{Z}_9$ \\

$17$ & $\mathbb{Z}_3^{(2)}$ & $\mathbb{Z}_9 \times \mathbb{Z}_3^{(2)}$ & $\mathbb{Z}_9^{(2)}\times \mathbb{Z}_3^{(4)}$  & Not Capable & $\mathbb{Z}_3$ \\

$18$ & $\mathbb{Z}_3$ & $\mathbb{Z}_3^{(3)}$  & $\mathbb{Z}_9 \times \mathbb{Z}_3^{(5)}$  & Not Capable & $\mathbb{Z}_3 \times \mathbb{Z}_3$ \\

$19$ & $\mathbb{Z}_3$ & $\mathbb{Z}_3^{(3)}$ & $\mathbb{Z}_9 \times \mathbb{Z}_3^{(5)}$  &  Not Capable & $\mathbb{Z}_9$ \\

$20$ & $\mathbb{Z}_3$ & $\mathbb{Z}_3^{(3)}$ & $\mathbb{Z}_9 \times \mathbb{Z}_3^{(5)}$  &  Not Capable & $\mathbb{Z}_9$ \\

$21$ & $\mathbb{Z}_3$ & $\mathbb{Z}_9$ &  $\mathbb{Z}_{27} \times \mathbb{Z}_9 \times \mathbb{Z}_3^{(2)}$
  & Not Capable & $\mathbb{Z}_3$\\

$22$ & $\{1\}$ & $\mathbb{Z}_9$ & $\mathbb{Z}_9^{(2)}\times \mathbb{Z}_3^{(2)}$  & Not Capable & $\mathbb{Z}_3$\\

$24$ & $\{1\}$ & $\mathbb{Z}_3$ & $\mathbb{Z}_{27} \times \mathbb{Z}_3^{(3)}$  & Not Capable & $\mathbb{Z}_{27}$ \\

$25$ & $\mathbb{Z}_3$ & $\mathbb{Z}_9 \times \mathbb{Z}_3^{(2)}$ &  $\mathbb{Z}_9 \times \mathbb{Z}_3^{(5)}$ & Not Capable & $\mathbb{Z}_3$\\

$26$ & $\mathbb{Z}_9 \times \mathbb{Z}_3$ & $X$ &  $X \times \mathbb{Z}_3^{(3)}$ & Capable  & $\{1\}$ \\

$27$ & $\mathbb{Z}_3$ & $\mathbb{Z}_9 \times \mathbb{Z}_3^{(2)}$ &  $\mathbb{Z}_9 \times \mathbb{Z}_3^{(5)}$ & Not Capable & $\mathbb{Z}_3$ \\

$28$ & $\mathbb{Z}_9$ & $Y$ & $Y \times \mathbb{Z}_3^{(3)}$  & Capable  & $\{1\}$\\

$29$ & $\mathbb{Z}_3$ & $\mathbb{Z}_9 \times \mathbb{Z}_3^{(2)}$ &  $\mathbb{Z}_9 \times \mathbb{Z}_3^{(5)}$ & Not Capable & $\mathbb{Z}_3$ \\

$30$ & $\mathbb{Z}_3$ & $\mathbb{Z}_9 \times \mathbb{Z}_3^{(2)}$ &  $\mathbb{Z}_9 \times \mathbb{Z}_3^{(5)}$ & Not Capable & $\mathbb{Z}_3$ \\

$32$ & $\mathbb{Z}_3^{(4)}$ & $\mathbb{Z}_3^{(5)}$  & $\mathbb{Z}_9 \times \mathbb{Z}_3^{(10)}$   & Not Capable & $\mathbb{Z}_3$\\

$33$ & $\mathbb{Z}_3^{(3)}$ & $\mathbb{Z}_9 \times \mathbb{Z}_3^{(2)}$  & $\mathbb{Z}_9^{(2)} \times \mathbb{Z}_3^{(7)}$   &  Capable  & $\{1\}$\\

$34$ & $\mathbb{Z}_9 \times \mathbb{Z}_3$ & $\mathbb{Z}_9 \times \mathbb{Z}_3^{(2)}$  &  $\mathbb{Z}_9^{(2)} \times \mathbb{Z}_3^{(7)}$  &  Capable  & $\{1\}$ \\

$35$ & $\mathbb{Z}_3^{(4)}$ & $\mathbb{Z}_3^{(5)}$  & $\mathbb{Z}_9 \times \mathbb{Z}_3^{(10)}$   & Not Capable & $\mathbb{Z}_3$\\

$36$ & $\mathbb{Z}_3^{(2)}$ & $\mathbb{Z}_3^{(3)}$  &  $\mathbb{Z}_9 \times \mathbb{Z}_3^{(8)}$  & Not Capable & $\mathbb{Z}_3 \times \mathbb{Z}_3$\\

$37$ & $\mathbb{Z}_3^{(6)}$ &  $\mathbb{Z}_3^{(8)}$ & $\mathbb{Z}_3^{(14)}$   &  Capable  & $\{1\}$ \\

$38$ & $\mathbb{Z}_3^{(3)}$ & $\mathbb{Z}_3^{(5)}$  & $\mathbb{Z}_3^{(11)}$   &  Not Capable & $\mathbb{Z}_3$\\

$39$ & $\mathbb{Z}_3^{(3)}$ & $\mathbb{Z}_3^{(5)}$ & $\mathbb{Z}_3^{(11)}$   & Not Capable & $\mathbb{Z}_3$\\

$40$ & $\mathbb{Z}_3^{(3)}$ & $\mathbb{Z}_3^{(5)}$  & $\mathbb{Z}_3^{(11)}$   & Not Capable & $\mathbb{Z}_3$\\

$41$ & $\mathbb{Z}_3$ & $\mathbb{Z}_3^{(3)}$ & $\mathbb{Z}_3^{(9)}$  & Not Capable & $\mathbb{Z}_3 \times \mathbb{Z}_3$ \\

$42$ & $\mathbb{Z}_3^{(2)}$ & $\mathbb{Z}_9 \times \mathbb{Z}_3^{(2)}$ & $\mathbb{Z}_9 \times \mathbb{Z}_3^{(8)}$  &  Capable  & $\{1\}$ \\

$43$ & $\mathbb{Z}_9$ & $\mathbb{Z}_9 \times \mathbb{Z}_3^{(2)}$ &  $\mathbb{Z}_9 \times \mathbb{Z}_3^{(8)}$ & Capable  & $\{1\}$ \\

$44$ & $\mathbb{Z}_3$ & $\mathbb{Z}_3^{(3)}$ & $\mathbb{Z}_3^{(9)}$  &  Not Capable & $\mathbb{Z}_3 \times \mathbb{Z}_3$ \\

$45$ & $\mathbb{Z}_9$ & $\mathbb{Z}_9 \times \mathbb{Z}_3^{(2)}$ &  $\mathbb{Z}_9 \times \mathbb{Z}_3^{(8)}$ & Capable  & $\{1\}$ \\
[.5ex] 
 \hline
 \end{tabular}
 \end{table}

\begin{table}[H]
\centering
 \begin{tabular}{||c c c c c c||} 
 \hline
 Group ID & $\M(G)$ & $G \wedge G$ & $G \otimes G$ & Capability & Epicenter \\ 
 [.5ex] 
 \hline\hline
 $46$ & $\mathbb{Z}_3$ & $\mathbb{Z}_3^{(3)}$  & $\mathbb{Z}_3^{(9)}$  &  Not Capable & $\mathbb{Z}_3 \times \mathbb{Z}_3$ \\

$47$ & $\mathbb{Z}_3$ & $\mathbb{Z}_3^{(3)}$ & $\mathbb{Z}_3^{(9)}$  &  Not Capable & $\mathbb{Z}_3 \times \mathbb{Z}_3$ \\

$49$ & $\mathbb{Z}_3^{(2)}$ & $\mathbb{Z}_3^{(3)}$  & $\mathbb{Z}_9 \times \mathbb{Z}_3^{(8)}$   & Not Capable & $\mathbb{Z}_9$\\

$50$ & $\mathbb{Z}_3^{(2)}$ & $\mathbb{Z}_3^{(3)}$  & $\mathbb{Z}_9 \times \mathbb{Z}_3^{(8)}$   & Not Capable & $\mathbb{Z}_9$\\

$51$ & $\mathbb{Z}_3^{(3)}$ & $\mathbb{Z}_3^{(5)}$  & $\mathbb{Z}_3^{(11)}$   &  Not Capable & $\mathbb{Z}_3$\\

$52$ & $\mathbb{Z}_3^{(3)}$ & $\mathbb{Z}_3^{(5)}$  &  $\mathbb{Z}_3^{(11)}$   &  Not Capable & $\mathbb{Z}_3$\\

$53$ & $\mathbb{Z}_3^{(4)}$ & $\mathbb{Z}_9 \times \mathbb{Z}_3^{(4)}$  &  $\mathbb{Z}_9 \times \mathbb{Z}_3^{(10)}$  &  Capable  & $\{1\}$\\

$54$ & $\mathbb{Z}_3^{(3)}$ &  $\mathbb{Z}_3^{(5)}$ &  $\mathbb{Z}_3^{(11)}$   & Not Capable & $\mathbb{Z}_3$\\

$55$ & $\mathbb{Z}_3^{(3)}$ &  $\mathbb{Z}_3^{(5)}$ &  $\mathbb{Z}_3^{(11)}$   &  Not Capable & $\mathbb{Z}_3$\\

$56$ & $\mathbb{Z}_3^{(3)}$ & $\mathbb{Z}_3^{(5)}$  & $\mathbb{Z}_3^{(11)}$    &  Not Capable & $\mathbb{Z}_3$\\

$57$ & $\mathbb{Z}_3^{(3)}$ & $\mathbb{Z}_3^{(5)}$  & $\mathbb{Z}_3^{(11)}$    &  Not Capable & $\mathbb{Z}_3$\\

$58$ & $\mathbb{Z}_3^{(4)}$ & $\mathbb{Z}_9 \times \mathbb{Z}_3^{(4)}$   & $\mathbb{Z}_9 \times \mathbb{Z}_3^{(10)}$   &  Capable  & $\{1\}$\\

$59$ & $\mathbb{Z}_3^{(3)}$ & $\mathbb{Z}_3^{(5)}$  &  $\mathbb{Z}_3^{(11)}$   &  Not Capable & $\mathbb{Z}_3$\\

$60$ & $\mathbb{Z}_3^{(3)}$ & $\mathbb{Z}_3^{(5)}$  &  $\mathbb{Z}_3^{(11)}$   &  Not Capable & $\mathbb{Z}_3$\\

$62$ & $\mathbb{Z}_3^{(7)}$ & $\mathbb{Z}_3^{(8)}$  &  $\mathbb{Z}_3^{(18)}$   &  Capable  & $\{1\}$\\

$63$ & $\mathbb{Z}_3^{(5)}$ & $\mathbb{Z}_3^{(6)}$  & $\mathbb{Z}_3^{(16)}$  &  Not Capable & $\mathbb{Z}_3$\\

$64$ & $\mathbb{Z}_3^{(5)}$ & $\mathbb{Z}_3^{(6)}$  & $\mathbb{Z}_3^{(16)}$ &  Not Capable & $\mathbb{Z}_3$\\

$65$ & $\mathbb{Z}_3^{(5)}$ & $\mathbb{Z}_3^{(6)}$  & $\mathbb{Z}_3^{(16)}$ &  Not Capable & $\mathbb{Z}_3$\\

$66$ & $\mathbb{Z}_3^{(5)}$ & $\mathbb{Z}_3^{(6)}$  & $\mathbb{Z}_3^{(16)}$ &  Not Capable & $\mathbb{Z}_3$\\
[.5ex] 
 \hline
 \end{tabular}
 \hspace{1cm}
\caption{Groups of order $3^5$}\label{Table4}
\end{table}
The groups $X$ and $Y$ in Table \ref{Table4} are given by
$$X=\langle a,b,c \mid [b,a]=c^3,[a,c]=[b,c]=1,a^9=b^9=c^9=1\rangle$$ 
and
$$Y=\langle a,b,c \mid [b,a]=c^3,[a,c]=[b,c]=1,a^9=b^3=c^9=1\rangle.$$

\noindent{\bf Acknowledgement.} The research of the first author is partly supported by Infosys grant. The second author is thankful to Harish-Chandra Research Institute, Allahabad  for offering post doctoral position to him, during which a good part of this  work was done.


\pagebreak
\begin{center}
{Appendix A}
\end{center}

For the convenience of the reader, we include presentations of groups of order $p^n$, $3 \le n \le 5$, given by James \cite{RJ}.
 
\noindent {\bf A1.}  Groups of order $p^3$ 

The isoclinism family $\Phi_2$ consists of the following groups:
\begin{enumerate}[label=(\roman*)]
\item $\Phi_2(21)=\langle \alpha,\alpha_1,\alpha_2 \mid [\alpha_1,\alpha]=\alpha_2, \alpha^p=\alpha_2, \alpha_1^p=\alpha_2^p=1 \rangle$,
\item $\Phi_2(111)=\langle \alpha,\alpha_1,\alpha_2 \mid [\alpha_1,\alpha]=\alpha_2, \alpha^p=\alpha_1^p=\alpha_2^p=1 \rangle$.
\end{enumerate}

\noindent {\bf A2.} Groups of order $p^4$

The isoclinism family $\Phi_2$ consists of the following groups:
\begin{enumerate}[label=(\roman*)]
\item $\Phi_2(211)a=\Phi_2(21)\times \mathbb{Z}_p$,
\item $\Phi_2(1^4)=\Phi_2(111)  \times \mathbb{Z}_p$,
\item $\Phi_2(31) = \langle\alpha,\alpha_1,\alpha_2 \mid [\alpha_1,\alpha]=\alpha^{p^2}=\alpha_2, \alpha_1^p=\alpha_2^p=1 \rangle$,
\item $\Phi_2(22) = \langle \alpha,\alpha_1,\alpha_2 \mid [\alpha_1,\alpha]=\alpha^{p}=\alpha_2, \alpha_1^{p^2}=\alpha_2^p=1 \rangle$,
\item $\Phi_2(211)b = \langle \alpha,\alpha_1,\alpha_2,\gamma \mid [\alpha_1,\alpha]=\gamma^p=\alpha_2, \alpha^p=\alpha_1^p=\alpha_2^p=1 \rangle$,
\item $\Phi_2(211)c = \langle \alpha,\alpha_1,\alpha_2 \mid [\alpha_1,\alpha]=\alpha_2, \alpha^{p^2}=\alpha_1^p=\alpha_2^p=1 \rangle$.
\end{enumerate}

The isoclinism family $\Phi_3$ consists of the following groups:

\begin{enumerate}[label=(\roman*)]
\item $\Phi_3(211)a= \langle \alpha,\alpha_1,\alpha_2, \alpha_3 \mid [\alpha_1,\alpha]=\alpha_2, [\alpha_2,\alpha]=\alpha^p=\alpha_3,\alpha_1^{(p)}=\alpha_2^p=\alpha_3^p=1 \rangle$,

\item $\Phi_3(211)b_r= \langle \alpha,\alpha_1,\alpha_2, \alpha_3 \mid [\alpha_1,\alpha]=\alpha_2, [\alpha_2,\alpha]^r=\alpha_1^{(p)}=\alpha_3^r,\alpha^p=\alpha_2^p=\alpha_3^p=1 \rangle$ for $r=1$ or $\nu$. 

\item $\Phi_3(1^4)= \langle \alpha,\alpha_1,\alpha_2, \alpha_3 \mid [\alpha_i,\alpha]=\alpha_{i+1},\alpha^p=\alpha_i^{(p)}=\alpha_3^p=1, (i=1,2) \rangle$.
\end{enumerate}

\noindent {\bf A3.}  Groups of order $p^5$

The groups of nilpotency class  $2$ falls in the isoclinism families $\Phi_2, \Phi_4$ and $\Phi_5$. 

The isoclinism family $\Phi_2$ consists of the following groups, in which the cyclic direct factor is generated by $\alpha_3$:

\begin{enumerate}[label=(\roman*)]
\item $\Phi_2(311)a =\Phi_2(31) \times \mathbb{Z}_p$,  
\item $\Phi_2(221)a = \Phi_2(22) \times \mathbb{Z}_p$, 
\item $\Phi_2(221)b = \Phi_2(21) \times \mathbb{Z}_{p^2}$, 
\item $\Phi_2(2111)a = \Phi_2(211)a \times \mathbb{Z}_p$, 
\item $\Phi_2(2111)b = \Phi_2(211)b \times \mathbb{Z}_p$, 
\item $\Phi_2(2111)c = \Phi_2(211)c \times \mathbb{Z}_p$, 
\item $\Phi_2(2111)d = \Phi_2(111) \times \mathbb{Z}_{p^2}$, 
\item $\Phi_2(1^5) = \Phi_2(1^4) \times \mathbb{Z}_p$,
\item $\Phi_2\left(41\right)= \langle \alpha,\alpha_1,\alpha_2 \mid [\alpha_1,\alpha]=\alpha^{p^3}=\alpha_2, \alpha_1^p=\alpha_2^p=1 \rangle$, 
\item $\Phi_2\left(32\right)a_1= \langle \alpha,\alpha_1,\alpha_2 \mid [\alpha_1,\alpha]=\alpha^{p^2}=\alpha_2, \alpha_1^{p^2}=\alpha_2^p=1 \rangle$,
\item $\Phi_2\left(32\right)a_2=\langle \alpha, \alpha_1, \alpha_2 \mid [\alpha_1,\alpha]=\alpha_1^{p}=\alpha_2, \alpha^{p^3}=\alpha_2^p=1 \rangle$,
\item $\Phi_2\left(311\right)b= \langle \alpha, \alpha_1, \alpha_2, \gamma \mid [\alpha_1,\alpha]=\gamma^{p^2}=\alpha_2, \alpha^{p}=\alpha_1^p=\alpha_2^p=1 \rangle$, 
\item $\Phi_2\left(311\right)c = \langle \alpha,\alpha_1,\alpha_2 \mid [\alpha_1,\alpha]=\alpha_2, \alpha^{p^3}=\alpha_1^{p}=\alpha_2^p=1 \rangle$,
\item $\Phi_2\left(221\right)c = \langle \alpha, \alpha_1, \alpha_2, \gamma \mid [\alpha_1,\alpha]=\gamma^{p}=\alpha_2, \alpha^{p^2}=\alpha_1^p=\alpha_2^p=1 \rangle$,
\item $\Phi_2\left(221\right)d=\langle \alpha,\alpha_1,\alpha_2 \mid [\alpha_1,\alpha]=\alpha_2, \alpha^{p^2}=\alpha_1^{p^2}=\alpha_2^p=1 \rangle$.
\end{enumerate}

Isoclinism family $\Phi_4$ consists of the following groups:
\begin{enumerate}[label=(\roman*)]
\item $\Phi_4\left(221\right)a= \langle\alpha,\alpha_1,\alpha_2,\beta_1,\beta_2 \mid [\alpha_i,\alpha]=\beta_i, \alpha^p=\beta_2, \alpha_1^p=\beta_1, \alpha_2^p=\beta_i^p=1\; (i=1,2) \rangle$,

\item $\Phi_4\left(221\right)b= \langle\alpha,\alpha_1,\alpha_2,\beta_1,\beta_2 \mid [\alpha_i,\alpha]=\beta_i, \alpha^p=\beta_2, \alpha_2^p=\beta_1, \alpha_1^p=\beta_i^p=1 \;(i=1,2)\rangle$,

\item $\Phi_4\left(221\right)c= \langle\alpha,\alpha_1,\alpha_2,\beta_1,\beta_2 \mid [\alpha_i,\alpha]=\beta_i=\alpha_i^p, \alpha^p=\beta_i^p=1 \;(i=1,2) \rangle$,

\item $\Phi_4\left(221\right)d_r= \langle\alpha,\alpha_1,\alpha_2,\beta_1,\beta_2 \mid [\alpha_i,\alpha]=\beta_i,\alpha_1^p=\beta_1^k,\alpha_2^p=\beta_2, \alpha^p=\beta_i^p=1 \;(i=1,2) \rangle$, where $k=\zeta^r, r=1,2, \ldots ,\frac{1}{2}(p-1)$,

\item $\Phi_4\left(221\right)e= \langle\alpha,\alpha_1,\alpha_2,\beta_1,\beta_2 \mid [\alpha_i,\alpha]=\beta_i,\alpha_1^p=\beta_2^{-1/4}, \alpha_2^p=\beta_1\beta_2,\alpha^p=\beta_i^p=1 \;(i=1,2) \rangle$,

\item $\Phi_4\left(221\right)f_0= \langle\alpha,\alpha_1,\alpha_2,\beta_1,\beta_2 \mid [\alpha_i,\alpha]=\beta_i,\alpha_1^p=\beta_2, \alpha_2^p=\beta_1^{\nu},\alpha^p=\beta_i^p=1 \;(i=1,2) \rangle$,

\item $\Phi_4\left(221\right)f_r= \langle\alpha,\alpha_1,\alpha_2,\beta_1,\beta_2 \mid [\alpha_i,\alpha]=\beta_i,\alpha_1^p=\beta_2^k, \alpha_2^p=\beta_1\beta_2,\alpha^p=\beta_i^p=1 \;(i=1,2) \rangle$, where $4k=\zeta^{2r+1}-1$ for $r=1,2, \ldots , \frac{1}{2}(p-1)$,

\item $\Phi_4\left(2111\right)a= \langle\alpha,\alpha_1,\alpha_2,\beta_1,\beta_2\mid [\alpha_i,\alpha]=\beta_i, \alpha^p=\beta_2,\alpha_i^p=\beta_i^p=1 \;(i=1,2) \rangle$,

\item $\Phi_4\left(2111\right)b= \langle\alpha,\alpha_1,\alpha_2,\beta_1,\beta_2 \mid [\alpha_i,\alpha]=\beta_i, \alpha_1^p=\beta_1, \alpha^p=\alpha_2^p=\beta_i^p=1\; (i=1,2) \rangle$,

\item $\Phi_4\left(2111\right)c= \langle\alpha,\alpha_1,\alpha_2,\beta_1,\beta_2 \mid[\alpha_i,\alpha]=\beta_i, \alpha_2^p=\beta_1, \alpha^p=\alpha_1^p=\beta_i^p=1 \;(i=1,2) \rangle$,

\item $\Phi_4\left(1^5\right)= \langle\alpha,\alpha_1,\alpha_2,\beta_1,\beta_2 \mid [\alpha_i,\alpha]=\beta_i, \alpha^p=\alpha_i^p=\beta_i^p=1 \;(i=1,2) \rangle$.
\end{enumerate}

The isoclinism family $\Phi_5$ consists of the following two groups: 

\begin{enumerate}[label=(\roman*)]
\item $\Phi_5(2111)=\langle\alpha_i,\alpha_2,\alpha_3, \alpha_4, \beta \mid [\alpha_1,\alpha_2]=[\alpha_3,\alpha_4]=\alpha_1^{p}=\beta, \alpha_2^p=\alpha_3^p=\alpha_4^p=\beta^p=1 \rangle$,
\item $\Phi_5(1^5)=\langle\alpha_1,\alpha_2,\alpha_3, \alpha_4, \beta \mid [\alpha_1,\alpha_2]=[\alpha_3,\alpha_4]=\beta, \alpha_1^p=\alpha_2^p=\alpha_3^p=\alpha_4^p=\beta^p=1 \rangle$.
\end{enumerate}

The groups of nilpotency class $3$ fall in the isoclinism families $\Phi_3, \Phi_6, \Phi_7$ and $\Phi_8$.

The class $\Phi_3$ consists of the following groups:
\begin{enumerate}[label=(\roman*)]
\item $\Phi_3(2111)a = \Phi_3(211)a \times \mathbb{Z}_p$, 

\item $\Phi_3(2111)b_r = \Phi_3(211)b_r \times \mathbb{Z}_p$ for $r=1$ or $\nu$, 

\item $\Phi_3(1^5) = \Phi_3(1^4) \times \mathbb{Z}_p$, 
 
\item $\Phi_3\left(311\right)a=\langle\alpha,\alpha_1,\alpha_2,\alpha_3\mid[\alpha_1,\alpha]=\alpha_2,[\alpha_2,\alpha]=\alpha^{p^2}=\alpha_3,\alpha_1^{(p)}=\alpha_2^p=\alpha_3^p=1\rangle$,

\item $\Phi_3\left(311\right)b_r=\langle\alpha,\alpha_1,\alpha_2,\alpha_3\mid [\alpha_1,\alpha]=\alpha_2,[\alpha_2,\alpha]^r=\alpha_1^{p^2}=\alpha_3, \alpha^p=\alpha_2^p=\alpha_3^p=1\rangle$ for $r=1$ or $\nu$,

\item $\Phi_3\left(221\right)a=\langle\alpha,\alpha_1,\alpha_2,\alpha_3\mid [\alpha_1,\alpha]=\alpha_2,[\alpha_2,\alpha]=\alpha^p=\alpha_3, \alpha_1^{p^2}=\alpha_2^p=\alpha_3^p=1\rangle$,

\item $\Phi_3\left(221\right)b_r=\langle\alpha,\alpha_1,\alpha_2,\alpha_3\mid [\alpha_1,\alpha]=\alpha_2,[\alpha_2,\alpha]^r=\alpha_1^{(p)}=\alpha_3^r, \alpha^{p^2}=\alpha_2^p=\alpha_3^p=1\rangle$ for $r=1$ or $\nu$,

\item $\Phi_3\left(2111\right)c=\langle\alpha,\alpha_1,\alpha_2,\alpha_3,\gamma\mid [\alpha_1,\alpha]=\alpha_2,[\alpha_2,\alpha]=\gamma^p=\alpha_3,\alpha^p=\alpha_i^{(p)}=1  \;(i=1,2,3)\rangle$,

\item $\Phi_3\left(2111\right)d=\langle\alpha,\alpha_1,\alpha_2,\alpha_3\mid [\alpha_i,\alpha]=\alpha_{i+1},\alpha^{p^2}=\alpha_i^{(p)}=\alpha_3^p=1  \;(i=1,2)\rangle$,

\item $\Phi_3\left(2111\right)e=\langle\alpha,\alpha_1,\alpha_2,\alpha_3\mid [\alpha_i,\alpha]=\alpha_{i+1},\alpha^p=\alpha_1^{p^2}=\alpha_{i+1}^p=1 \;(i=1,2)\rangle$.
\end{enumerate}

The class $\Phi_6$ consists of the following groups:
\begin{enumerate}[label=(\roman*)]
\item $\Phi_6(221)a=\langle \alpha_1,\alpha_2,\beta,\beta_1,\beta_2\mid [\alpha_1,\alpha_2]=\beta,[\beta,\alpha_i]=\beta_i=\alpha_i^p,\beta^p=\beta_i^p=1 \; (i=1,2)\rangle$,

\item $\Phi_6(221)b_r=\langle \alpha_1,\alpha_2,\beta,\beta_1,\beta_2\mid [\alpha_1,\alpha_2]=\beta,[\beta,\alpha_i]=\beta_i,\alpha_1^p=\beta_1^k,\alpha_2^p=\beta_2,\beta^p=\beta_i^p=1 \; (i=1,2)\rangle$, where $k= \zeta^r, r=1,2, \ldots , \frac{1}{2}(p-1)$,

\item $\Phi_6(221)c_r=\langle \alpha_1,\alpha_2,\beta,\beta_1,\beta_2\mid [\alpha_1,\alpha_2]=\beta,[\beta,\alpha_i]=\beta_i,\alpha_1^p=\beta_2^{-\frac{1}{4}r},\alpha_2^p=\beta_1^r\beta_2^r,\beta^p=\beta_i^p=1\;(i=1,2)\rangle$, where $r=1$ or $\nu$,

\item $\Phi_6(221)d_0=\langle\alpha_1,\alpha_2,\beta,\beta_1,\beta_2\mid [\alpha_1,\alpha_2]=\beta,[\beta,\alpha_i]=\beta_i,\alpha_1^p=\beta_2,\alpha_2^p=\beta_1^{\nu},\beta^p=\beta_i^p=1\;(i=1,2)\rangle$,

\item $\Phi_6(221)d_r=\langle\alpha_1,\alpha_2,\beta,\beta_1,\beta_2\mid [\alpha_1,\alpha_2]=\beta,[\beta,\alpha_i]=\beta_i,\alpha_1^p=\beta_2^k,\alpha_2^p=\beta_1\beta_2,\beta^p=\beta_i^p=1\; (i=1,2)\rangle$ where $4k=\zeta^{2r+1}-1, r=1,2, \ldots, \frac{1}{2}(p-1)$,

\item $\Phi_6(2111)a=\langle \alpha_1,\alpha_2,\beta,\beta_1,\beta_2\mid [\alpha_1,\alpha_2]=\beta,[\beta,\alpha_i]=\beta_i, \alpha_1^p=\beta_1,\alpha_2^p=\beta^p=\beta_i^p=1 \; (i=1,2)\rangle$ for $p > 3$,

\item $\Phi_6(2111)b_r=\langle \alpha_1,\alpha_2,\beta,\beta_1,\beta_2\mid [\alpha_1,\alpha_2]=\beta,[\beta,\alpha_i]=\beta_i, \alpha_2^p=\beta_1^r,\alpha_1^p=\beta^p=\beta_i^p=1 \; (i=1,2)\rangle$ for $r=1$ or $\nu$ and $p >3$,

\item $\Phi_6(1^5)=\langle \alpha_1,\alpha_2,\beta,\beta_1,\beta_2\mid [\alpha_1,\alpha_2]=\beta,[\beta,\alpha_i]=\beta_i,\alpha_i^p=\beta^p=\beta_i^p=1\;(i=1,2)\rangle$.
\end{enumerate}

The class $\Phi_7$ consists of the following groups:

\begin{enumerate}[label=(\roman*)]
\item $\Phi_7\left(2111\right)a=\langle\alpha,\alpha_1,\alpha_2,\alpha_3,\beta \mid [\alpha_i,\alpha]=\alpha_{i+1},[\alpha_1,\beta]=\alpha_3=\alpha^p, \alpha_1^{(p)}=\alpha_{i+1}^p=\beta^p=1 ~(i=1,2) \rangle$,

\item $\Phi_7\left(2111\right)b_r=\langle \alpha,\alpha_1,\alpha_2,\alpha_3,\beta \mid [\alpha_i,\alpha]=\alpha_{i+1},[\alpha_1,\beta]^r=\alpha_3^r=\alpha_1^{(p)}, \alpha^{p}=\alpha_{i+1}^p=\beta^p=1~ (i=1,2) \rangle $ for $r=1$ or $\nu$,

\item $\Phi_7\left(2111\right)c=\langle \alpha,\alpha_1,\alpha_2,\alpha_3,\beta \mid [\alpha_i,\alpha]=\alpha_{i+1},[\alpha_1,\beta]=\alpha_3=\beta^p, \alpha^p=\alpha_1^{(p)}=\alpha_{i+1}^p=1 ~(i=1,2) \rangle$,

\item $\Phi_7\left(1^5\right)=\langle \alpha,\alpha_1,\alpha_2,\alpha_3,\beta \mid [\alpha_i,\alpha]=\alpha_{i+1},[\alpha_1,\beta]=\alpha_3, \alpha^p=\alpha_1^{(p)}=\alpha_{i+1}^p=\beta^p=1 ~(i=1,2)\rangle$.
\end{enumerate}

The class $\Phi_8$ consists of only one group 

\begin{enumerate}[label=(\roman*)]
\item $\Phi_8(32)=\langle\alpha_1, \alpha_2, \beta \mid [\alpha_1,\alpha_2]=\beta=\alpha_1^p, {\beta}^{p^2}=\alpha_2^{p^2}=1\rangle$\\
\end{enumerate}

Groups of nilpotency class $4$ fall in the  isoclinism classes $\Phi_9$ and $\Phi_{10}$.  

The class $\Phi_9$ consists of the following groups:

\begin{enumerate}[label=(\roman*)]
\item $\Phi_9\left(2111\right)a=\langle \alpha,\alpha_1,\alpha_2,\alpha_3,\alpha_4\mid [\alpha_i,\alpha]=\alpha_{i+1},\alpha_4=\alpha^p, \alpha_1^{(p)}=\alpha_{i+1}^{(p)}=1 ~(i=1,2,3)\rangle$,

\item $\Phi_9\left(2111\right)b_r=\langle \alpha,\alpha_1,\alpha_2,\alpha_3,\alpha_4\mid [\alpha_i,\alpha]=\alpha_{i+1},\alpha_4^k=\alpha_1^{(p)}, \alpha^p=\alpha_{i+1}^{(p)}=1 \;(i=1,2,3)\rangle$, where $k=\zeta^r$ for $r+1=1,2,\ldots,(p-1,3)$

\item $\Phi_9\left(1^5\right)=\langle \alpha,\alpha_1,\alpha_2,\alpha_3,\alpha_4\mid [\alpha_i,\alpha]=\alpha_{i+1},\alpha^p=\alpha_1^{(p)}=\alpha_{i+1}^{(p)}=1 ~(i=1,2,3)\rangle$.\\
\end{enumerate}

The class $\Phi_{10}$ consists of the following groups:
\begin{enumerate}[label=(\roman*)]
\item $\Phi_{10}\left(2111\right)a_r=\langle \alpha,\alpha_1,\alpha_2,\alpha_3,\alpha_4\mid [\alpha_i,\alpha]=\alpha_{i+1},[\alpha_1,\alpha_2]^k=\alpha_4^k=\alpha^p, \alpha_1^{(p)}=\alpha_{i+1}^{(p)}=1~ (i=1,2,3)\rangle$, where $k=\zeta^r$ for $r+1=1,2,\ldots,(p-1,4)$

\item $\Phi_{10}\left(2111\right)b_r=\langle \alpha,\alpha_1,\alpha_2,\alpha_3,\alpha_4 \mid [\alpha_i,\alpha]=\alpha_{i+1},[\alpha_1,\alpha_2]^k=\alpha_4^k=\alpha_1^{(p)}, \alpha^p=\alpha_{i+1}^{(p)}=1 ~(i=1,2,3)\rangle$,  where $k=\zeta^r$ for $r+1=1,2,\ldots,(p-1,3)$ and  $p > 3$,

\item  $\Phi_{10}\left(1^5\right)=\langle \alpha,\alpha_1,\alpha_2,\alpha_3,\alpha_4\mid [\alpha_i,\alpha]=\alpha_{i+1},[\alpha_1,\alpha_2]=\alpha_4,\alpha^p=\alpha_1^{(p)}=\alpha_{i+1}^{(p)}=1~ (i=1,2,3)\rangle$.
\end{enumerate}
\end{document}